\documentclass[a4paper,12pt]{amsart}

\usepackage{amssymb}
\usepackage{amsmath}
\usepackage{tikz}
\usepackage[all,cmtip]{xy}
\newtheorem{thm}{Theorem}[section] 
\newtheorem{lem}[thm]{Lemma} 
 
\newtheorem{prop}[thm]{Proposition}
 
\theoremstyle{definition}
\newtheorem*{ack}{Acknowledgments}
\newtheorem{defi}[thm]{Definition} 
\newtheorem{rem}[thm]{Remark}

\newtheorem{example}[thm]{Example}

\title[]{Classification of finite irreducible \\conformal modules for $K'_4$}

\author{Lucia Bagnoli}\author{Fabrizio Caselli}
\address{Lucia Bagnoli and Fabrizio Caselli, Dipartimento di matematica, Universit\`a di Bologna, Piazza di Porta San Donato 5, 40126 Bologna, Italy}

\email{lucia.bagnoli4@unibo.it, luciabagnoli93@gmail.com}
\email{fabrizio.caselli@unibo.it}
\DeclareMathOperator{\Res}{Res}

\DeclareMathOperator{\Ind}{Ind}
\DeclareMathOperator{\Cur}{Cur}
\DeclareMathOperator{\Lie}{Lie}
\DeclareMathOperator{\Sing}{Sing}
\DeclareMathOperator{\Ker}{Ker}

\DeclareMathOperator{\End}{End}
\DeclareMathOperator{\I}{I}
\DeclareMathOperator{\Supp}{Supp}
\DeclareMathOperator{\Id}{Id}

\newcommand{\C}{\mathbb{C}}
\newcommand{\Z}{\mathbb{Z}}

\newcommand{\de}{\partial}

\newcommand{\inlinewedge}{\textrm{\raisebox{0.6mm}{\footnotesize $\bigwedge$}}}
\newcommand{\displaywedge}{\textrm{\raisebox{0.6mm}{\tiny $\bigwedge$}}}
\newcommand{\g}{\mathfrak {g}}

\newcommand*{\bigchi}{\mbox{\Large$\chi$}}
\allowdisplaybreaks
\begin{document}
	\maketitle
	\begin{abstract}
	We classify the finite irreducible modules over the conformal superalgebra $K'_{4}$ by their correspondence with finite conformal modules over the associated annihilation superalgebra $\mathcal A(K'_{4})$. This is achieved by a complete classification of singular vectors in generalized Verma modules for $\mathcal A(K'_{4})$. We also show that morphisms between generalized Verma modules can be arranged in infinitely many bilateral complexes.
		
	\end{abstract}

	\section{Introduction} 
	Finite simple conformal superalgebras were completely classified in \cite{fattorikac} and consist of the following list: $\Cur \mathfrak{g}$, where $\mathfrak{g}$ is a simple finite$-$dimensional Lie superalgebra, $W_{n} (n\geq 0)$, $S_{n,b}$, $\tilde{S}_{n}$ $(n\geq 2, \, b \in \mathbb{C})$, $K_{n} (n\geq 0, \, n \neq 4)$, $K'_{4}$, $CK_{6} $. The finite irreducible modules over the conformal superalgebras $\Cur \mathfrak{g}$, $K_{0}$, $K_{1}$ were studied in \cite{chengkac}. Boyallian, Kac, Liberati and Rudakov classified all finite irreducible modules over the conformal superalgebras of type $W$ and $S$ in \cite{bklr}; Boyallian, Kac and Liberati classified all finite irreducible modules over the conformal superalgebras of type $K_{n}$ in \cite{kac1}. The classification of all finite irreducible modules over the conformal superalgebras of type $K_{n}$, for $n  \leq 4$, had been previously studied also by Cheng and Lam in \cite{chenglam}. Finally, a classification of all finite irreducible modules over the conformal superalgebra $CK_{6}$ was obtained in \cite{ck6} and \cite{zm} with different approaches. For $n=4$ the conformal superalgebra $K_{4}$ is not simple and its the derived subalgebra $K'_{4}$ is instead a simple conformal superalgebra.\par
	A possible strategy for studying modules over conformal superalgebras is the following. If $R$ is a conformal superalgebra one considers the Lie superalgebra $\g=\mathcal A(R)$, called the annihilation superalgebra of $R$. The annihilation superalgebra has a fundamental role since the study of the finite modules over $R$ is equivalent to the study of \textit{finite conformal} modules over $\g$. Furthermore, if $R$ is $\Z$-graded then $\g$ is also $\Z$-graded and one can reduce the problem to finite Verma modules of $\g$, i.e. induced modules $\Ind(F)=U(\g) \otimes _{U(\g_{\geq 0})} F$, where $F$ is a finite dimensional $\g_{\geq 0}$-module \cite{kacrudakov,chenglam}. 
	 
	This is the case for the simple conformal superalgebra $K'_4$, and its annihilation superalgebra $\mathcal A(K'_{4})$. The main goal of this paper is therefore to classify irreducible conformal modules for $K'_4$ through the classification of all degenerate (i.e., non irreducible) finite Verma modules for $\mathcal A(K'_4)$; in turn, this is equivalent to the classification of (highest weight) singular vectors in these modules, i.e. vectors which are annihilated by $\mathcal A(K'_4)_{>0}$. The final result is much richer than in the "standard" conformal contact superalgebras $K_n$ where, up to duality, there is only one family of singular vectors, all of degree 1: we show that for $\mathcal A(K'_{4})$  there are four families of singular vectors of degree 1, four families of singular vectors of degree 2 and two "exceptional" singular vectors of degree 3.

Since the classification of singular vectors in finite Verma modules is equivalent to the classification of morphisms between such modules, we show that these morphisms can be arranged in an infinite number of bilateral complexes in a picture (see Figure \ref{figura}) which is similar to those obtained for the exceptional linearly compact Lie superalgebras $E(1,6)$, $E(3,6)$, $E(3,8)$ and $E(5,10)$ (see \cite{kacrudakovE36,kacrudakov,kacrudakovE38,E36III, cantacaselliE510, cantacasellikacE510}). In a subsequent publication we will compute the homology of these complexes and provide an explicit construction of all irreducible quotients.

%Each point represents the generalized Verma module $M(m,n,\mu_{t},\mu_{c})$, where $(m,n,\mu_{t},\mu_{c})$ is determined by its position with respect to the axes $m=0,n=0$ and $\mu_{t},\mu_{c}$ by the quadrant. The arrows represent the morphisms of $\g-$modules constructed as before.
%\begin{figure}[h!]
%\centering
%\includegraphics[width=10cm]{vett_sing.jpg}
%\caption{Complessi di moduli di Verma}
%\label{Complessi di moduli di Verma intro}
%\end{figure}
The paper is organized as follows. In section 2 we collect all preliminaries on conformal superalgebras which are needed, in section 3 we describe the conformal superalgebra $K'_4$ and in section 4 its annihilation superalgebra $\mathcal A(K'_{4})$. In section 5 we show explicitly how the conformal superalgebra $K'_4$ acts on a finite Verma module. In section 6 we deduce the crucial conditions that must be satisfied by a singular vector and we show that singular vectors have degree at most 3. Finally, section 7, 8, 9 contain the classification of singular vectors of degree 2, 3, 1 respectively.

\section{Preliminaries on conformal superalgebras}
In this section we introduce some notions on conformal superalgebras. For further details see \cite[Chapter 2]{kac1vertex}, \cite{dandrea}, \cite{bklr}, \cite{kac1}.\\
Let $\g$ be a Lie superalgebra; a formal distribution with coefficients in $\g$, or equivalently a $\g-$valued formal distribution, in the indeterminate $z$ is an expression of the following form:
\begin{align*}
a(z)=\sum_{n \in \Z}a_{n}z^{n},  
\end{align*}
with $a_{n} \in \g$ for every $n \in \Z$. We denote the vector space of formal distributions with coefficients in $\g$ in the indeterminate $z$ by $\g[[z,z^{-1}]]$. We denote by $\Res(a(z))=a_{-1}$ the coefficient of $z^{-1}$ of $a(z)$. The vector space $\g[[z,z^{-1}]]$ has a natural structure of $\C[\partial_{z}]-$module. We define for all $a(z) \in \g[[z,z^{-1}]]$ its derivative:
\begin{align*}
\partial_{z}a(z)=\sum_{n \in \Z}na_{n}z^{n-1}.
\end{align*} 
A formal distribution with coefficients in $\g$ in the indeterminates $z$ and $w$ is an expression of the following form:
\begin{align*}
a(z,w)=\sum_{m,n \in \Z}a_{m,n}z^{m}w^{n},  
\end{align*}
with $a_{m,n} \in \g$ for every $m,n \in \Z$. We denote the vector space of formal distributions with coefficients in $\g$ in the indeterminates $z$ and $w$ by $\g[[z,z^{-1},w,w^{-1}]]$.
Given two formal distributions $a(z) \in \g[[z,z^{-1}]]$ and $b(w) \in \g[[w,w^{-1}]]$, we define the commutator $[a(z),b(w)]$:
\begin{align*}
[a(z),b(w)]=\bigg[\sum_{n \in \Z}a_{n}z^{n} ,\sum_{m \in \Z}b_{m}w^{m}  \bigg]=\sum_{m,n \in \Z}[a_{n},b_{m}] z^{n}  w^{m} .
\end{align*}
\begin{defi}
	Two formal distributions $a(z),b(z) \in \g[[z,z^{-1}]]$ are called \emph{local} if
	\begin{align*}
	(z-w)^{N}[a(z),b(w)]=0 \text{ for  } N \gg 0.
	\end{align*}
\end{defi}
We call $\delta-$function the following formal distribution in the indeterminates $z$ and $w$:
\begin{align*}
\delta(z-w)=\sum_{m,n:\, m+n=-1} z^m w^n.
\end{align*}
See Corollary $2.2$ in \cite{kac1vertex} for the following equivalent condition of locality.
\begin{prop}
	Two formal distributions $a(z),b(z) \in \g[[z,z^{-1}]]$ are local if and only if $[a(z),b(w)]$ can be expressed as a finite sum of the form:
	\begin{align*}
	[a(z),b(w)]=\sum_{j}(a(w)_{(j)}b(w))\frac{\partial_{w}^{j}}{j!} \delta(z-w),
	\end{align*}
	where the coefficients $(a(w)_{(j)}b(w))$  are formal distributions in the indeterminate $w$. Moreover, if $a(z)$ and $b(z)$ are local then necessarily $(a(w)_{(j)}b(w))=\Res_{z}(z-w)^{j}[a(z),b(w)]$.
\end{prop}
\begin{defi}[Formal Distribution Superalgebra]
	Let $\g$ be a Lie superalgebra and $\mathcal{F}$ a family of mutually local $\g-$valued formal distributions in the indeterminate $z$. The pair $(\g,\mathcal{F})$ is called a \textit{formal distribution superalgebra} if the coefficients of all formal distributions in $\mathcal{F}$ span $\g$.
\end{defi}
We define the $\lambda-$bracket between two local formal distributions $a(z),b(z) \in \g[[z,z^{-1}]] $ as the generating series of the $(a(z)_{(j)}b(z))$'s:
\begin{align}
\label{bracketformal}
[a(z)_{\lambda}b(z)]=\sum_{j \geq 0} \frac{\lambda^{j}}{j!}(a(z)_{(j)}b(z)).
\end{align}
The $\lambda$-bracket of formal distributions satisfies some algebraic properties which are the motivation of the following definition. If $V$ is any $\Z_2$-graded vector space we denote by $p$ its parity function. As customary, whenever we write $p(v)$ for some $v\in V$ we always implicitly assume that $v$ is a homogeneous element of $V$.
\begin{defi}[Conformal superalgebra]\label{coal}
	A \textit{conformal superalgebra} $R$ is a left $\Z_{2}-$graded $\C[\partial]-$module endowed with a $\C-$linear map, called $\lambda-$bracket, $R \otimes R \rightarrow \C[\lambda]\otimes R$, $a \otimes b \mapsto [a_{\lambda}b]$, that satisfies the following properties for all $a,b,c \in R$:
	\begin{itemize}
		\item[(i)] $p(\partial a)=p(a)$;
		\item[(ii)] $[\partial a_{\lambda}b]=-\lambda [a_{\lambda}b], \quad  [a_{\lambda} \partial b]=(\lambda+\partial)[a_{\lambda}b]$;
		\item[(iii)] $[a_{\lambda}b]=-(-1)^{p(a)p(b)}[b_{-\lambda-\partial}a] $;
		\item[(iv)] $[a_{\lambda}[b_{\mu}c]]=[[a_{\lambda}b]_{\lambda+\mu}c]+(-1)^{p(a)p(b)}[b_{\mu}[a_{\lambda}c]]$.
	\end{itemize}
	
\end{defi}
We refer to properties (ii), (iii), (iv) in Definition \ref{coal} as the conformal linearity, conformal symmetry and conformal Jacobi identity respectively.
We call $n-$products the coefficients $(a_{(n)}b)$ that appear in $[a_{\lambda}b]=\sum_{n\geq 0} \frac{\lambda^{n}}{n!}(a_{(n)}b)$  and give an equivalent definition of conformal superalgebra.
\begin{defi}[Conformal superalgebra]
	\label{definizionesuperalgebraconforme}
	A \textit{conformal superalgebra} $R$ is a left $\Z_{2}-$graded $\C[\partial]-$module endowed with a $\C-$bilinear product $(a_{(n)}b): R\otimes R\rightarrow R$, defined for every $n \geq 0$, that satisfies the following properties for all $a,b,c \in R$:
	\begin{itemize}
		\item[(i)] $p(\partial a)=p(a)$;
		\item[(ii)] $(a_{(n)}b)=0, \,\, for \,\, n \gg 0$;
		\item[(iii)] $(\partial a_{(0)}b)=0$ and $(\partial a_{(n)}b)=-n (a_{(n-1)}b)$ for all $n\geq 1$;
		\item[(iv)] $(a_{(n)}b)=-(-1)^{p(a)p(b)}\sum_{j \geq 0}(-1)^{j+n} \frac{\partial^{j}}{j!}(b_{(n+j)}a)$ for all $n\geq 0$;
		\item[(v)] $(a_{(m)}(b_{(n)}c))=\sum^{m}_{j=0}\binom{m}{j}((a_{(j)}b)_{(m+n-j)}c)+(-1)^{p(a)p(b)}(b_{(n)}(a_{(m)}c))$ for all $m,n\geq 0$.
	\end{itemize}
	
\end{defi}
Using conditions (iii) and (iv) in Definition \ref{definizionesuperalgebraconforme} it is easy to show that for all $a,b \in R$, $n \geq 0$:
\begin{equation}\label{deadestra}
(a_{(n)}\partial b)=\partial (a_{(n)} b)+n (a_{(n-1)}b).
\end{equation}
In particular, by the first part of (iii) in Definition \ref{definizionesuperalgebraconforme}, the map $\partial: R\rightarrow R$, $a \mapsto \partial a$ is a derivation with respect to the $0-$product.
\begin{rem}
	Let $\mathcal F$ be a formal distribution superalgebra in the indeterminate $z$ which is a vector subspace of $\C[[z]]$ and is invariant under the operator $\partial_z$. Then the formal distribution algebra $\mathcal F$, endowed with $\lambda-$bracket (\ref{bracketformal}) and operator $\partial=\partial_z$ is a conformal superalgebra (for a proof see \cite[Proposition 2.3]{kac1vertex}).
\end{rem}
We say that a conformal superalgebra $R$ is \textit{finite} if it is finitely generated as a $\C[\partial]-$module. 
An \textit{ideal} $I$ of $R$ is a $\C[\partial]-$submodule of $R$ such that $a_{(n)}b\in I$ for every $a \in R$, $b \in I$, $n \geq 0$. A conformal superalgebra $R$ is \textit{simple} if it has no non-trivial ideals and the $\lambda-$bracket is not identically zero. We denote by $R'$ the \textit{derived subalgebra} of $R$, i.e. the $\C-$span of all $n-$products.
\begin{defi}
	A module $M$ over a conformal superalgebra $R$ is a $\Z_{2}-$graded $\C[\partial]-$module endowed with $\C-$linear maps $R \rightarrow \End_{\C} M$, $a\mapsto a_{(n)}$, defined for every $n \geq 0$, that satisfy the following properties for all $a,b \in R$, $v \in M$:
	\begin{enumerate}
		\item[(i)] $a_{(n)}v=0 \,\, for \,\, n \gg 0$;
		\item[(ii)] $(\partial a)_{(n)}v=[\partial,a_{(n)}]v=-n a_{(n-1)}v $ for all $n\geq 0$;
		\item[(iii)] $[a_{(m)},b_{(n)}]v=\sum_{j} \binom{m}{j}(a_{(j)}b)_{(m+n-j)}v$ for all $m,n\geq 0$.
	\end{enumerate}
\end{defi}
A module $M$ is called \textit{finite} if it is a finitely generated $\C[\partial]-$module.\\
We can construct a conformal superalgebra starting from a formal distribution superalgebra $(\g,\mathcal{F})$. Let $\mathcal{\overline{F}}$  be the closure of $\mathcal{F}$ under all the $n-$products, $\partial_{z}$ and linear combinations. By  Dong's Lemma, $\mathcal{\overline{F}}$ is still a family of mutually local formal distributions (see \cite{kac1vertex}) and it turns out that $\mathcal{\overline{F}}$ is a conformal superalgebra. We will refer to $\overline {\mathcal F}$ as the conformal superalgebra \emph{associated} with $(\g,\mathcal{F})$.\\
Let us recall the construction of the annihilation superalgebra associated with a conformal superalgebra $R$.
Let $\widetilde{R}=R[y,y^{-1}]$, set $p(y)=0$ and $\widetilde{\partial}=\partial+\partial_{y}$. We define the following $k-$products on $\widetilde{R}$, for all $a,b \in R$, $f,g \in \C[y,y^{-1}]$, $k\geq 0$:
\begin{align*}
(af_{(k)}bg)=\sum_{j \in \Z_{+}}(a_{(k+j)}b) \Big(\frac{\partial_{y}^{j}}{j!}f \Big)g .
\end{align*}
In particular if $f=y^{m}$ and $g=y^{n}$ we have for all $k \geq 0$:
\begin{align}\label{1,5}
({ay^{m}}_{ (k)}by^{n})=\sum_{j \in \Z_{+}}\binom{m}{j}(a_{(k+j)}b)y^{m+n-j}.
\end{align}
We observe that $\widetilde{\partial}\widetilde{R}$ is a two sided ideal of $\widetilde{R}$ with respect to the $0-$product. The quotient $\Lie R:=\widetilde{R}/ \widetilde{\partial}\widetilde{R}$ has a structure of Lie superalgebra with the bracket induced by the $0-$product, i.e. for all $a,b \in R$, $f,g \in \C[y,y^{-1}]$,
\begin{align}
\label{bracketannihilation}
[af,bg]=\sum_{j \in \Z_{+}}(a_{( j)}b)\Big(\frac{\partial_{y}^{j}}{j!}f \Big)g .
\end{align}
%The images in  $\Lie R$ of elements $ay^{m} \in \widetilde{R}$ are often denoted by $a_{m}$.
\begin{defi}
	The \emph{annihilation superalgebra} $\mathcal{A}(R)$ of a conformal superalgebra $R$ is the subalgebra of $\Lie R$ spanned by all elements $ay^{n}$ with $n\geq 0$ and $a\in R$. \\
	The extended annihilation superalgebra $\mathcal{A}(R)^{e}$ of a conformal superalgebra $R$ is the Lie superalgebra $\C \partial \ltimes \mathcal{A}(R)$. The semidirect sum $\C \partial \ltimes \mathcal{A}(R)$ is the vector space $\C \partial \oplus \mathcal{A}(R)$ endowed with the structure of Lie superalgebra uniquely determined by the bracket
	\begin{align*}
	[\partial,ay^{m}]=-\partial_{y}(ay^{m})=-m a y^{m-1},
	\end{align*}
	for all $a \in R$, and the fact that $\mathcal A(R)$ and $\C \partial$ are Lie subalgebras.
\end{defi}
For all $a \in R$ we consider the following formal power series in $\mathcal{A}(R)[[\lambda]]$:
\begin{align*}
a_{\lambda}=\sum_{n \geq 0}\frac{\lambda^{n}}{n!}ay^{n}.
\end{align*}
For all $a,b \in R$, we have: $[a_{\lambda},b_{\mu}]=[a_{\lambda}b]_{\lambda+\mu}$ and $(\partial a)_{\lambda}=-\lambda a_{\lambda}$ (for a proof see \cite{cantacasellikac}). This notation is coherent with the definition of conformal modules in the following sense.  
\begin{prop}[\cite{chengkac}]Let $R$ be a conformal superalgebra.
	\label{propcorrispmoduli}
	If $M$ is a finite conformal $R$-module then $M$ has a natural structure of $\mathcal A(R)^e$-module, where the action of $ay^n$ on $M$ is uniquely determined by $a_\lambda v=\sum_{n \geq 0}\frac{\lambda^{n}}{n!}ay^{n}.v$ for all $v\in V$. Viceversa if $M$ is a $\mathcal A(R)^e$-module such that for all $a\in R$, $v\in M$ we have $ay^n.v=0$ for $n\gg0$ then $M$ is also a finite conformal module by letting $a_\lambda v=\sum_{n}\frac{\lambda^{n}}{n!}ay^{n}.v$.
	\end{prop}
One usually refers to Proposition \ref{propcorrispmoduli} by saying that a module over a conformal superalgebra $R$ is the \emph{same} as a \emph{continuous} module over the Lie superalgebra $\mathcal{A}(R)^{e}$. 
Proposition \ref{propcorrispmoduli} reduces the study of modules over a conformal superalgebra $R$ to the study of a class of modules over its (extended) annihilation superalgebra.

In some cases one can even avoid to use the extended annihilation algebra and simply consider the annihilation algebra. 
Recall that a Lie superalgebra $\g$ is $\Z$-graded if $\g=\bigoplus_{n\in \Z} \g_n$ with $[\g_n,\g_m]\subseteq \g_{n+m}$ for all $n,m\in \Z$. We say in this case that $\g$ has finite depth $d\geq 0$ if $\g_n=0$ for all $n<-d$ and $\g_{-d}\neq 0$. 

\begin{prop}[\cite{kac1}]
	\label{keythmannihi}
	Let $\g$ be the annihilation superalgebra of a conformal superalgebra $R$. Assume that $\g$ satisfies the following conditions:
	\begin{description}
		\item[L1] $\g$ is $\Z-$graded with finite depth $d$;
		\item[L2] there exists $t\in \g$ such that the centralizer of $t$ is contained in $\g_{0}$;
		\item[L3] there exists $\Theta \in \g_{-d}$ such that $\g_{i-d}=[\Theta,\g_{i}]$, for all $i\geq 0$.
	\end{description}
	Finite modules over $R$ are the same as modules $V$ over $\g$, called \textit{finite conformal}, that satisfy the following properties:
	\begin{enumerate}
		\item for every $v \in V$, we have $\g_n.v=0$ for $n\gg 0$;
		\item $V$ is finitely generated as a $\C[\Theta]-$module.
	\end{enumerate}
\end{prop}
\begin{rem}
	\label{gradingelement}
	We point out that condition {\bf L2} is automatically satisfied when $\g$ contains a grading element, i.e. an element $t \in \g$ such that $[t,b]=\deg (b) b$ for all $b \in \g$.
\end{rem}
Let $\g=\bigoplus_{i \in \Z} \g_{i}$ be a $\Z-$graded Lie superalgebra. We will use the notation $\g_{>0}=\bigoplus_{i>0}\g_{i}$, $\g_{<0}=\bigoplus_{i<0}\g_{i}$ and $\g_{\geq 0}=\bigoplus_{i\geq 0}\g_{i}$. We denote by $U(\g) $ the universal enveloping algebra of $\g$.
\begin{defi}
	Let $F$ be a $\g_{\geq 0}-$module. The \textit{generalized Verma module} associated with $F$ is the $\g-$module $\Ind (F)$ defined by,
	\begin{equation*}
	\Ind (F):= \Ind ^{\g}_{\g_{\geq 0}} (F)=U(\g) \otimes _{U(\g_{\geq 0})} F.
	\end{equation*}
\end{defi}
If $F$ is a finite dimensional irreducible $\g_{\geq 0}-$module we will simply say that $\Ind(F)$ is a finite Verma module. We will identify $\Ind (F)$ with $U(\g_{<0}) \otimes  F$ as vector spaces via the Poincar\'e$-$Birkhoff$-$Witt Theorem. The $\Z-$grading of $\g$ induces a $\Z-$grading on $U(\g_{<0})$ and $\Ind (F)$. We will invert the sign of the degree, so that we have a $\Z_{\geq 0}-$grading on $U(\g_{<0})$ and $\Ind (F)$. We will say that an element $v \in U(\g_{<0})_{k}$ is homogeneous of degree $k$. Analogously an element $m \in U(\g_{<0})_{k}  \otimes  F$ is homogeneous of degree $k$.
\begin{prop}
	Let $\g=\bigoplus_{i \in \Z} \g_{i}$ be a $\Z-$graded Lie superalgebra. If $F$ is an irreducible finite$-$dimensional $\mathfrak{g}_{\geq 0}-$module, then $\Ind (F)$ has a unique maximal submodule. We denote by $\I (F)$ the quotient of $\Ind (F)$ by the unique maximal submodule.
\end{prop}
\begin{proof}
	First we point out that a submodule $V \neq \left\{ 0 \right\}$ of $\Ind (F)$ is proper if and only if it does not contain nontrivial elements of degree 0. Indeed, if $V$ contains an element $v_{0} \neq 0$ of degree 0, then it contains $1 \otimes F= \mathfrak{g}_{\geq 0}. v_{0} $, due to irreducibility of $F$. Therefore $\g_{<0}. F =\Ind (F) \subseteq V$. The union $S$ of all proper submodules is still a proper submodule of $\Ind (F)$, since $S$ does not contain nontrivial elements of degree 0, thus $S$ is the unique maximal  proper submodule.
\end{proof}
\begin{defi}
	Given a $\g-$module $V$, we call \textit{singular vectors} the elements of:
	\begin{align*}
	\Sing (V) =\left\{v \in V \,\, | \, \, \g_{>0}.v=0\right\}.
	\end{align*}
	Homogeneous components of singular vectors are still singular vectors so we often assume that singular vectors are homogeneous without loss of generality.
	In the case $V=\Ind (F)$ for a $\g_{\geq 0}-$module $F$, we will call \textit{trivial singular vectors} the elements of $\Sing (V) $ of degree 0 and \textit{nontrivial singular vectors} the nonzero elements of $\Sing (V) $ of positive degree.
\end{defi}
\begin{thm}[\cite{kacrudakov},\cite{chenglam}]
	\label{keythmsingular}
	Let $\g$ be a Lie superalgebra that satisfies L1, L2, L3 in Proposition \ref{keythmannihi}; then
	\begin{enumerate}
		\item if $F$ is an irreducible finite$-$dimensional $\mathfrak{g}_{\geq 0}-$module, then the action of  $\mathfrak{g}_{> 0}$ on $F$ is trivial;
		\item the map $F \mapsto \I (F)$ is a bijective map between irreducible finite$-$dimensional $\mathfrak{g}_{ 0}-$modules and irreducible finite conformal $\mathfrak{g}-$modules;
		\item the $\mathfrak{g}-$module $\Ind (F)$ is irreducible if and only if the $\mathfrak{g}_{0}-$module $F$ is irreducible and $\Ind (F)$ has no nontrivial singular vectors.
	\end{enumerate}
\end{thm}
We recall the notion of duality for conformal modules (see for further details \cite{bklr}, \cite{cantacasellikac}). Let $R$ be a conformal superalgebra and $M$ a conformal module over $R$.
\begin{defi}
	The conformal dual $M^{*}$ of $M$ is defined by
	\begin{align*}
	M^{*}=\left\{f_{\lambda}:M\rightarrow \C[\lambda] \,\, | \,\, f_{\lambda}(\partial m)=\lambda f_{\lambda}(m), \,\, \forall m \in M\right\}.
	\end{align*}
	The structure of $\C[\partial]-$module is given by $(\partial f)_{\lambda}(m)=-\lambda  f_{\lambda}(m)$, for all $f \in M^{*}$, $ m \in M$. The $\lambda-$action of $R$ is given, for all $a \in R$, $m \in M$, $f \in M^{*}$, by:
	\begin{align*}
	(a_{\lambda}f)_{\mu}(m)=-(-1)^{p(a)p(f)}f_{\mu-\lambda}(a_{\lambda}m).
	\end{align*}
\end{defi}
\begin{defi}
	Let $T:M\rightarrow N$ be a morphism of $R-$modules, i.e. a linear map such that for all $a\in R$ and $m\in M$:
	\begin{itemize}
		\item [i:] $T(\partial m)=\partial T(m)$,
		\item [ii:] $T(a_{\lambda} m)=a_{\lambda} T(m)$.
	\end{itemize}
	The dual morphism $T^{*}:N^{*} \rightarrow M^{*}$ is defined, for all $f \in N^{*}$ and $m \in M$, by:
	\begin{align*}
	\left[T^{*}(f)\right]_{\lambda}(m)=-f_{\lambda}\left(T(m)\right).
	\end{align*}
\end{defi}

\section{The conformal superalgebra $K'_{4}$}
\label{section$K'_{4}$}
In this section we introduce and study the contact Lie superalgebras and related conformal superalgebras. Let $\inlinewedge(N)$ be the Grassmann superalgebra in the $N$ odd indeterminates $\xi_{1},...,\xi_{N}$. Let $t$ be an even indeterminate and $\inlinewedge (1,N)=\C[t,t^{-1}] \otimes \inlinewedge(N)$ which we consider as an associative algebra in the natural way omitting the symbol $\wedge$ between the indeterminates $\xi_i$'s. We also consider the Lie superalgebra of derivations of $\inlinewedge (1,N)$:
\begin{equation*}
W(1,N)=\bigg\{ D=a \partial_{t}+\sum ^{N}_{i=1} a_{i} \partial_{i} \,\, | \,\, a,a_{i} \in \displaywedge (1,N)\bigg\},
\end{equation*}
where $\partial_{t}=\frac{\partial}{\partial{t}}$ and $\partial_{i} =\frac{\partial}{\partial{\xi_{i}}}$ for every $i \in \left\{1,...,N \right\}$.\\
Let us consider the contact form $\omega = dt-\sum_{i=1}^{N}\xi_{i} d\xi_{i} $. The contact Lie superalgebra $K(1,N)$ is defined by:
\begin{equation*}
K(1,N)=\left\{D \in W(1,N) \,\, | \,\, D\omega=f_{D}\omega \text{ for some } f_{D} \in \displaywedge (1,N)\right\}.
\end{equation*}
We denote by $K'(1,N)$ the derived algebra $[K(1,N),K(1,N)]$ of $K(1,N)$.
Analogously, let $\inlinewedge  (1,N)_{+}=\C[t] \otimes \inlinewedge (N)$. We consider the Lie superalgebra of derivations of $\inlinewedge  (1,N)_{+}$:
\begin{equation*}
W(1,N)_{+}=\bigg\{ D=a \partial_{t}+\sum ^{N}_{i=1} a_{i} \partial_{i} \,\, | \,\, a,a_{i} \in \displaywedge (1,N)_{+}\bigg\}.
\end{equation*}
The Lie superalgebra $K(1,N)_{+}$ is defined by:
\begin{equation*}
K(1,N)_{+}=\left\{D \in W(1,N)_{+} \,\, | \,\, D\omega=f_{D}\omega \,\, for \,\, some \,\, f_{D} \in \displaywedge (1,N) _{+}\right\}.
\end{equation*}
One can define on $\inlinewedge (1,N)$ a Lie superalgebra structure as follows: for all $f,g \in \inlinewedge(1,N)$ we let:
\begin{equation}
\label{bracketlie}
[f,g]=\Big(2f-\sum_{i=1}^{N} \xi_{i}  \partial_{i} f \Big)\partial_{t}{g}-\partial_{t}{f}\Big(2g-\sum_{i=1}^{N} \xi_{i} \partial_{i} g\Big)+(-1)^{p(f)}\Big(\sum_{i=1}^{N} \partial_{i}  f \partial_{i}  g \Big).
\end{equation}
It is useful to restate \eqref{bracketlie} in a more explicit way. We adopt the following notation: we let $\mathcal I$ be the set of (finite) sequences of elements in $\{1,\ldots,N\}$; for notational convenience  we usually write $I=i_1\cdots i_r$ instead of $I=(i_1,\ldots,i_r)$ and we think of $\mathcal I$ as a monoid by juxtaposition (i.e. if $I=i_1\cdots i_r$ and $J=j_1\cdots j_s$ we let $IJ=i_1\cdots i_rj_1\cdots j_s$); if $I=i_1 \cdots i_r\in \mathcal I$  we let $\xi_{I}=\xi_{i_1}\cdots \xi_{i_r}$ and $|I|=r$. For $m,n\in \Z$ and $I,J\in \mathcal I$ we have
\begin{equation}\label{bracketliebis}
	[t^m\xi_{I},t^n\xi_J]=(2n-2m-n|I|+m|J|)t^{m+n-1}\xi_{IJ}+(-1)^{|I|}t^{m+n}\sum_i \de_i\xi_I\,\de_i \xi_J.
\end{equation}
We recall that $K(1,N) \cong \inlinewedge(1,N)$ as Lie superalgebras via the following map (see \cite{{chengkac2}}):
\begin{gather*}
%\begin{align*}
\displaywedge(1,N) \longrightarrow  K(1,N) \\
f \longmapsto 2f \partial_{t}+(-1)^{p(f)} \sum_{i=1}^{N} (\xi_{i} \partial_{t} f+ \partial_{i}f )(\xi_{i} \partial_{t} + \partial_{i}).
%\end{align*}
\end{gather*}
From now on we will always identify elements of $K(1,N) $ with elements of $\inlinewedge(1,N)$. We consider on $K(1,N)$ the standard grading, i.e. for every $t^{m} \xi_{i_{1}} \cdots \xi_{i_{s}} \in K(1,N)$ we have $\deg(t^{m} \xi_{i_{1}} \cdots \xi_{i_{s}})=2m+s-2$.\\
Next target is to realize $K(1,N)_{+}$ as the annihilation superalgebra of a conformal superalgebra.
In order to do this, we construct a formal distribution superalgebra using the following family of formal distributions:
\begin{equation*}
\mathcal{F}=\bigg\{A(z):= \sum_{m \in \Z}(t^{m}A)z^{-m-1}=A \delta(t-z), \,\, \forall    A \in \displaywedge(N)\bigg\}. 
\end{equation*}
Note that the set of all the coefficients of formal distributions in $\mathcal{F}$ spans $\inlinewedge(1,N)$.
\begin{prop}
	\label{propnproducts}
	The pair $(\inlinewedge(1,N),\mathcal F)$ is a formal distribution superalgebra. More precisely, for all $I,J\in \mathcal I$ we have
	\begin{align}
	\label{nproducts}
	(\xi_I(z)_{(0)}\xi_J(z))&=(|I|-2)\partial_{z}\xi_{IJ}(z)+(-1)^{|I|} \sum _{i=1}^N (\partial _{i}\xi_I \: \partial _{i}\xi_J)(z); \\ \nonumber
	(\xi_I(z)_{(1)}\xi_J(z))&=(|I|+|J|-4)\xi_{IJ}(z); \\ \nonumber
	(\xi_I(z)_{(n)}\xi_J(z))&=0 \, \, for \, \, n > 1.
	\end{align}
	In particular the conformal superalgebra associated with $(\inlinewedge(1,N),\mathcal F)$  is $\mathcal{\bar{F}}=\mathbb{C}[\partial_{z}]\mathcal{F}$. 
\end{prop}
\begin{proof}
Let's show that $\xi_I(z)$ and $\xi_J(z)$ are local. We have:
	\begin{align*}
	[\xi_I(z)&,\xi_J(w)]=\sum_{m,n \in \Z} [t^{m}\xi_I,t^{n}\xi_J]z^{-m-1} w^{-n-1}\\
	&=\sum_{m,n \in \Z}\Big(  \left(n\left(2-|I| \right) -m \left(2-|J|\right) \right) t^{m+n-1}\xi_{IJ} \, +(-1)^{|I|}t^{m+n}\sum_{i=1}^{N} \partial_{i}\xi_I \: \partial_{i}\xi_J \Big)  z^{-m-1} w^{-n-1} \\
	%&=\sum_{m,n \in \Z}\left( n\left(2-r \right) -m \left(2-s\right) \right) AB \, t^{m+n-1} \, \frac{z^{-m-1} }{w^{n+1}}+ \sum_{m,n \in \Z}(-1)^{r}\sum_{i=1}^{N} \partial_{i}A \: \partial_{i}B \, t^{m+n} \, \frac{z^{-m-1} }{w^{n+1}}.
	\end{align*}
	We let $h=m+n-1$ in the former sum and $l=m+n$ in the latter and we obtain
	\begin{align*}
	&[\xi_I(z),\xi_J(w)]\\
	&=	\sum_{h,m \in \Z}  \left(  \left(h-m+1\right) \left(2-|I|\right) -m \left(2-|J|\right) \right) t^h \xi_{IJ} \frac{z^{-m-1} }{w^{-(m-h-2)}}\\
	&\hspace{5mm}+ \sum_{l,m \in \Z}(-1)^{r}\sum_{i=1}^{N} t^l\partial_{i}\xi_I \partial_{i}\xi_J \frac{z^{-m-1} }{w^{-(m-l-1)}} \\
	&=\sum_{h,m \in \Z} (h+1)(2-|I|)t^h\xi_{IJ} w^{-h-2} z^{-m-1} w^{m}+\sum_{h,m \in \Z}m(|I|+|J|-4)t^h\xi_{IJ}w^{-h-1} z^{-m-1} w^{m-1} \\
	&\hspace{5mm}+\sum_{l,m \in \Z}(-1)^{|I|}\sum_{i=1}^{N} t^l\partial_{i}\xi_I \partial_{i}\xi_J w^{-l-1} z^{-m-1} w^{m}\\
	&=(|I|-2) \partial_{w} (\xi_{IJ}(w)) \delta(z-w) +(|I|+|J|-4) \xi_{IJ}(w) \partial_{w} \delta(z-w)\\
	&\hspace{5mm}+(-1)^{|I|}\sum_{i=1}^{N} (\partial_{i}\xi_I \: \partial_{i}\xi_J) (w) \delta(z-w).
	%=&\Big(\left(r-2\right) \partial_{w} \left(AB\right)(w) +(-1)^{r}\sum_{i=1}^{N} \left(\partial_{i}A \: \partial_{i}B\right) (w)\Big) \delta(z-w)+ (r+s-4) \left(AB \right)(w) \partial_{w} \delta(z-w).
	\end{align*}
	All results follow.
	%Therefore, for all $A, B \in \inlinewedge(N)$, $A=\xi_{i_{1}} \cdots \xi_{i_{r}}$ and $B=\xi_{j_{1}} \cdots \xi_{j_{s}}$, the $n-$products are given by:
	%\begin{align*}
	%(A(z)_{(0)}B(z))&=(r-2)\partial_{z}(AB)(z)+(-1)^{r} \sum _{i=1}^N (\partial _{i} A \:\partial _{i}B)(z); \\
	%(A(z)_{(1)}B(z))&=(r+s-4)(AB)(z); \\
	%(A(z)_{(n)}B(z))&=0 \, \, for \, \, n > 1.
	%\end{align*}
	%For all $A, B \in \inlinewedge(N)$, $n \geq 0 $, all other $n-$products can be found by linearity and the %relations:
	%\begin{align*}
	%((\partial_{z} A)(z) _{(n)}B(z))&=-n(A(z) _{(n-1)}B(z)),\\
	%(A(z)_{(n)}\partial_{z} B(z))&=\partial_{z} (A(z)_{(n)} B(z))+n (A(z)_{(n-1)} B(z)).
	%\end{align*}
	%Hence the closure of $\mathcal{F}$ under all $n-$products and $\partial_{z}$ is $\mathcal{\bar{F}}=\mathbb{C}[\partial_{z}]\mathcal{F}$.
\end{proof}
We can say something more about the conformal superalgebra $\mathcal{\bar{F}}$ associated with the formal distribution superalgebra ($K(1,N)$,$\mathcal{F}$).
\begin{prop}
	\label{linearindep}
	The conformal superalgebra $\mathcal{\bar{F}}=\mathbb{C}[\partial_{z}]\mathcal{F}$ is a free $\mathbb{C}[\partial_{z}]-$module.
\end{prop}
\begin{proof}
	If $A_1,A_2,\ldots,A_s$ is a basis of $\inlinewedge(N)$ then $A_1\delta(t-z),A_2\delta(t-z),\ldots,A_s \delta(t-z)$ is a basis of  $\mathcal{F}$.
	Let us consider a finite linear combination, with coefficients in $\mathbb{C}[\partial_{z}]$, of elements of this basis:
	\begin{align*}
	\sum^{s}_{i=1} P_{i}(\partial_{z}) A_{i} \delta(t-z)=0, \\
	\end{align*}
	where  $P_{i}(\partial_{z}) \in \mathbb{C}[\partial_{z}]$ for every $ 1 \leq i \leq s$. From linear independence of the $A_{i}$'s, we obtain for every $1\leq i \leq s$:
	\begin{align*}
	P_{i}(\partial_{z}) \delta(t-z)=0 .
	\end{align*}
	Therefore every coefficient $P_{i}$ must be 0.
\end{proof}
We will identify $\mathcal{\bar{F}}=\mathbb{C} [\partial_{z}]\otimes\mathcal{F}$ with $K_{N}:=\C[\partial] \otimes \inlinewedge(N)$. We also identify $\partial_{z}$ with $\partial $ and every $A(z) \in \mathcal{F}$ with $A \in \inlinewedge(N)$.
We will refer to $K_{N}$ as the conformal superalgebra associated with $K(1,N)$. For all $I,J\in \mathcal I$ the $\lambda-$bracket is given by
\begin{equation}
\label{lambdabracket}
[{\xi_I}_{\lambda}\xi_J]=(|I|-2)\partial \xi_{IJ}+(-1)^{|I|}\sum^{N}_{i=1}\partial_{i}\xi_I\partial_{i}\xi_J+\lambda(|I|+|J|-4)\xi_{IJ},
\end{equation}
by Proposition \ref{propnproducts}.
In \cite{kac1} it is shown that the annihilation superalgebra of $K_{N}$ is $\mathcal{A}(K_{N})=K(1,N)_{+}$ and that it satisfies conditions L1, L2, L3. Thus, the study of finite irreducible modules over the conformal superalgebra $K_{N}$ is reduced to the study of singular vectors of Verma modules on $K(1,N)_{+}$.\\
Now we concentrate in the special case $N=4$, because the conformal superalgebra $K_{4}$ is not simple. The derived superalgebra $K'_{4}$ is one of the exceptional cases appearing in the classification of finite simple conformal superalgebras  in \cite{fattorikac}. Our main target is to study all finite irreducible modules over the conformal superalgebra $K'_{4}$.
%$K_{N}$ is simple except for the case $N=4$. If $N=4$, $K_{4}=K'_{4}\oplus \C \xi_{1}\xi_{2}\xi_{3}\xi_{4}$, where $K'_{4}$ is the derived subalgebra, i.e. the $\C-$span of the $j-$products and $K'_{4}$ is simple (see \cite{fattorikac}).

In order to describe $K_4'$ explicitly we need to introduce the following terminology.
Let $V$ be a vector space and $B=\left\{b_{i}\right\}_{i \in \I}$ be a basis of $V$. An element $v \in V$ can be uniquely expressed as $v= \sum_{i} c_{i} b_{i}$. The support of $v$ with respect to $B$ is $\Supp_B v=\left\{b_{i} \, : \, c_{i} \neq 0\right\}$. We will usually drop the index $B$ if there is no risk of confusion.

Recall that we denote by $\mathcal I$ the set of all sequences with entries in $\{1,2,3,4\}$. We also denote by $\mathcal I_{\neq}$ the set of sequences in $\mathcal I$ with distinct entries and by $\mathcal I_<$ the set of sequences in $\mathcal I$ with strictly increasing entries. For typographical reasons we simply denote by $i_1\cdots i_r$ the sequence $(i_1,\ldots,i_r)$.
\begin{prop}\label{xi1234}
	The element $\xi_{1234} \notin K'_{4}$. More precisely:
	\begin{align*}
	K'_{4}=\langle \{\partial^{k}\xi_{I}, \, \partial^{l}\xi_{1234}:\,   I\in \mathcal I_<,\,I\neq 1234, \, k\geq 0,\,l>0 \}\rangle .
	\end{align*}
	%The conformal superalgebra $K_{4} \cong K'_{4}\oplus \C \xi_{1}\xi_{2}\xi_{3}\xi_{4}$ as vector spaces.
\end{prop}
\begin{proof}
	By Proposition \ref{linearindep}, we know that $\left\{ \partial^{k}\xi_{I}:\,k \geq 0,\,I\in \mathcal I_< \right\} $ is a basis for $K_{4}$.
	We first show that $ \xi_{1234} \notin K'_{4}$. Since the $j-$products are bilinear maps, it is sufficient to show that $ \xi_{1234}$ does not belong to $\Supp ({\xi_I}_{(j)}\xi_J)$, for all $I,J\in \mathcal I_<$.  This is an immediate consequence of \eqref{lambdabracket}. %Indeed it does not belong to the support of $(-1)^{r}(\sum_{i=1}^{4} (\partial_{i}f )(\partial_{i}  g ))$, because, for all $1 \leq i \leq 4$, $ \xi_{1}\xi_{2}\xi_{3}\xi_{4} \notin \Supp(  (\partial_{i}f )(\partial_{i}  g) ) $.  Clearly it does not belong to the support of $(r-2)\partial(fg)$. \\%because $ \xi_{1}\xi_{2}\xi_{3}\xi_{4} \notin \Supp(\partial(fg))$.\\ The element $ \xi_{1}\xi_{2}\xi_{3}\xi_{4}$ does not belong to $\Supp(f_{(1)}g)$. Indeed if $ \xi_{1}\xi_{2}\xi_{3}\xi_{4} \in \Supp ((r+s-4)fg )$, then $r+s=4$, that is a contradiction.\\
	
	Now we show that every element $\partial^{k}\xi_I$ with $k>0$ or $I\neq 1234$ lies in  $K'_{4}$: \begin{enumerate}\item if $k>0$, then $\partial^{k}\xi_I=\left(-\frac{1}{2} _{(0)} \partial^{k-1}\xi_I \right)$ by \eqref{lambdabracket};
	\item if $k=0$ and $I\neq 1234$ let $i\in \{1234\}$ be such that $\xi_{iI}\neq 0$. Then we have $\xi_I=-({\xi_{i}\,}_{(0)}\xi_{iI})$ by \eqref{lambdabracket}.
	\end{enumerate}
	
\end{proof}
\begin{prop}
	\label{3.4}
	The element $t^{-1}\xi_{1234} \notin K'(1,4)$. More precisely:
	\begin{align*}
	K'(1,4)=\langle\{ t^{k}\xi_{I}, \, t^{l}\xi_{1234}:\, I\in \mathcal I_<,\,I\neq 1234, \, k,l \in \Z, \, l \neq -1\}\rangle .
	\end{align*}
\end{prop}
\begin{proof}
	We know that $\left\{ t^{k}\xi_{I}:\,k \in \Z , I\in\mathcal I_< \right\} $ is a basis for $K(1,4)$.
	Let us first show that $t^{-1}\xi_{1234} \notin K'(1,4)$.
	Since the bracket (\ref{bracketlie}) is bilinear, it is sufficient to prove that $t^{-1} \xi_{1234}$ does not belong to $\Supp[t^m\xi_I,t^n\xi_J]$ for all $m,n\in \Z$ and $I,J\in \mathcal I_<$. Indeed, if $t^{-1} \xi_{1234}\in \Supp[t^m\xi_I,t^n\xi_J]$ then necessarily $m+n=0$ and $|I|+|J|=4$, but these conditions imply that the coefficient $2n-2m-n|I|+m|J|$ in \eqref{bracketliebis} vanishes, leading to a contradiction.
	 
	Next we show that every monomial $t^n\xi_I$ with $n\neq -1$ or $I\neq 1234$ belongs to $K'(1,4)$:
	\begin{enumerate}
		\item recall that $[t,t^n\xi_I]=\deg(t^n\xi_I)t^n \xi_I$. In particular, if $\deg(t^n\xi_I)\neq 0$ the result follows.
		\item if $\deg(t^n\xi_I) = 0$, then either $n=0$ and $I=ij$, or $n=1$ and $I=\emptyset$. The result follows since $\xi_{ij}=-[\xi_{kij},\xi_{ij}]$ (for any $k\neq i,j$) and  $ t=-[t \xi_{1}, \xi_{1}]$.
	\end{enumerate}
\end{proof}
\section{The annihilation superalgebra of $K'_{4}$}
Motivated by Proposition \ref{keythmannihi} and Theorem \ref{keythmsingular}, we want to understand the structure of $\mathcal{A}(K'_{4})$.\\
Let us recall some notions on central extensions of Lie superalgebras.
\begin{defi}
	Let $\g$ be a Lie superalgebra. A \textit{$2-$cocycle} on $\g$ is a bilinear map $\psi: \g \times \g \rightarrow \C$ that satisfies the following conditions: 
	\begin{enumerate}
		\item $\psi(a,b)=-(-1)^{p(a)p(b)}\psi(b,a)$,
		\item $(-1)^{p(a)p(c)}\psi(a,[b,c])+(-1)^{p(a)p(b)}\psi(b,[c,a])+(-1)^{p(a)p(c)}\psi(c,[a,b])=0$,
	\end{enumerate}
	for all $a,b,c \in \g$.
	The set of all $2-$cocycles on $\g$ is denoted by $Z^{2}(\g, \C)$.
\end{defi}
\begin{rem}
	We denote the set of linear maps $\g \rightarrow \C$ by $C^{1}(\g, \C)$ and  we call its elements \textit{$1-$cochains}. For every $1-$cochain $f \in C^{1}(\g, \C)$, it is possible to construct a $2-$cocycle $\delta f$ on $\g$. For all $a,b \in \g$ we define:
	\begin{equation*}
	\delta f (a,b)=f([a,b]).
	\end{equation*}
	It is a straightforward verification that $\delta f$ is a $2-$cocycle on $\g$.
	The map $\delta: C^{1}(\g, \C) \rightarrow Z^{2}(\g, \C)$, $f \rightarrow \delta f$, is called \textit{coboundary operator}.
\end{rem}
\begin{defi}
	We denote by $B^{2}(\g, \C)$ the image of $\delta: C^{1}(\g, \C) \rightarrow Z^{2}(\g, \C)$. Two $2-$cocycles $\psi_{1},\psi_{2} \in Z^{2}(\g, \C) $ are \textit{cohomologous} when $\psi_{1}-\psi_{2} \in B^{2}(\g, \C)$. We denote by $H^{2}(\g, \C)$ the quotient $\frac{Z^{2}(\g, \C)}{B^{2}(\g, \C)}$.
\end{defi}
\begin{defi}
	A Lie superalgebra $\hat{\g }$ is a \textit{central extension} of $\g$ by a one$-$dimensional center $\C C$ if there exist two (Lie superalgebras) homomorphisms $i: \C C \rightarrow  \hat{\g }$  and $s:  \hat{\g } \rightarrow  \g$ such that the following sequence is exact:
	\begin{equation*}
	0 \rightarrow \C C \xrightarrow{i}  \hat{\g } \xrightarrow{s}  \g \rightarrow 0,
	\end{equation*}
	and $\Ker (s)$ lies in the center of $\hat{\g }$.
\end{defi}
\begin{defi}
	Two central extentions $\hat{\g }_{1}$ and $\hat{\g }_{2}$ of $\g$ by a one$-$dimensional center $\C C$ are isomorphic if there exists an isomorphism of Lie superalgebras $\Phi: \hat{\g }_{1}\rightarrow \hat{\g }_{2}$ such that the following diagram is commutative:
	\[
	\xymatrix{
		0 \ar[r] &\C C \ar[d]^{\Id} \ar[r]^{i_{1}} &\hat{\g }_{1} \ar[d]^{\Phi} \ar[r]^{s_{1}}  &\g \ar[d]^{\Id} \ar[r] & 0 \\
		0 \ar[r] &\C C \ar[r]^{i_{2}}             &\hat{\g }_{2} \ar[r]^{s_{2}}                &\g \ar[r]^d          & 0. }
\]
\end{defi}
Next result is certainly well-known but we include a sketch of the proof for completeness and for the reader's convenience.
\begin{prop}
	\label{centralex}
	There is a bijection between (isomorphism classes of) central extensions of $\g$ by a one$-$dimensional center and elements of $H^{2}(\g, \C)$. If $\psi \in Z^{2}(\g, \C)$ the corresponding central extension is, up to isomorphism, $\hat{\g } = \g \oplus \C C$ where:
	\begin{align*}
	[C,a]=0 \quad  \text{and} \quad [a,b]_{\hat{\g } }=[a,b]_{\g}+\psi (a,b) C,
	\end{align*}
	for all $a,b \in \g$.
\end{prop}
\begin{proof}Let 	\begin{equation*}
	0 \rightarrow \C C \xrightarrow{i}  \hat{\g } \xrightarrow{s}  \g \rightarrow 0,
	\end{equation*} be a central extension of $\g$. In particular, $\hat{\g } \cong \g \oplus \C i(C)$ as vector spaces and we have the following relation between the bracket $[\cdot,\cdot]_{\hat{\g }}$ in $\hat{\g } $ and the bracket $[\cdot,\cdot]_{\g}$ in $\g$ for all $a,b \in   \g  $, $\alpha,\beta \in \C$:
	\begin{equation*}
	[a+\alpha i(C),b+\beta i(C)]_{\hat{\g } }=[a,b]_{\g}+\psi (a,b) i(C),
	\end{equation*}
	where $\psi: \g \times \g \rightarrow \C$ is a $2-$cocycle.\\
	Conversely, given $\psi \in C^{2}(\g, \C)$, we can construct a central extension $\hat{\g }$ of $\g$. We define $\hat{\g } := \g \oplus \C C$. For all $a,b \in \g$, $\alpha, \beta \in \C$, we set $i(\alpha C):=\alpha C$, $s(a+\alpha C):=a$  and $[a+\alpha C,b+\beta C]_{\hat{\g } }:=[a,b]_{\g}+\psi (a,b) C$. It follows directly from the definition of $2-$cocycles that it is a central extension.\\
	Finally we show that two isomorphic central extensions $\hat{\g }_{1}\cong \g \oplus \C C$ and $\hat{\g }_{2} \cong \g \oplus \C C$ correspond to cohomologous $2-$cocycles. Since $\hat{\g }_{1}$ and $\hat{\g }_{2}$ are isomorphic, we have an isomorphism $\Phi:\hat{\g }_{1} \rightarrow \hat{\g }_{2} $ such that the following diagram is commutative:
	\[
	\xymatrix{
		0 \ar[r] &\C C \ar[d]^{\Id} \ar[r]^{i_{1}} &\hat{\g }_{1} \ar[d]^{\Phi} \ar[r]^{s_{1}}  &\g \ar[d]^{\Id} \ar[r] & 0 \\
		0 \ar[r] &\C C \ar[r]^{i_{2}}             &\hat{\g }_{2} \ar[r]^{s_{2}}                &\g \ar[r]^d          & 0. }
	\]
	Thus for all $a \in \g$, $\alpha \in \C$:
	\begin{gather}
	\label{cocicli}
	\Phi(a+\alpha C )= a+\rho(a)C+\alpha  C  ,
	\end{gather}
	where $\rho \in C^{1}(\g, \C)$.\\
	We call $\psi_{1}$(resp. $\psi_{2}$) the $2-$cocycle that corresponds to $\hat{\g }_{1}$(resp. $\hat{\g }_{2}$). We have for all $a,b \in \g$:
	\begin{align*}
	\Phi ([a,b]_{\hat{\g }_{1}})&=  \Phi ([a,b]_{\g}+\psi_{1} (a,b) C)\\
	&=[a,b]_{\g}+ \left( \rho([a,b]_{\g})+\psi_{1} (a,b) \right)C.
	\end{align*}
	But from the fact that $\Phi$ is an isomorphism we also have:
	\begin{align*}
	\Phi ([a,b]_{\hat{\g }_{1}})&=  [\Phi(a),\Phi(b)]_{\hat{\g }_{2}}\\
	&=[a+\rho(a)C,b+\rho(b)C]_{\hat{\g }_{2}}\\
	&=[a,b]_{\g}+\psi_{2}(a,b)C.
	\end{align*}
	Therefore, $\delta \rho+\psi_{1}=\psi_{2}$.\\
	Analogously, if $\psi_{1},\psi_{2} \in Z^{2}(\g, \C) $ are cohomologous, i.e. $\psi_{1}-\psi_{2}=\delta \eta \in B^{2}(\g, \C)$, then we can construct an isomorphism between the central extensions defined by $\psi_{1}$ and $\psi_{2}$ as in \eqref{cocicli} letting $\rho:=\eta$.
\end{proof}
The following proposition is the main result of this section.
\begin{prop}
	\label{anniliK4}
	There exists a (unique) surjective morphism of Lie superalgebras $\phi:\Lie K'_4\rightarrow K'(1,4)$ such that for all $m\in \Z$
	
	\begin{align*}
	&\phi(\xi_I y^m)=t^m\xi_I ,\, \textrm{for all }I\in \mathcal I_<,\,I\neq 1234\\
	&\phi(\de \xi_{1234}y^n)=-m\xi_{1234}t^{m-1}
	\end{align*}
	
	and $\Ker(\phi)=\C\, \de \xi_{1234}$.
	%The Lie superalgebra $\Lie K'_{4}$ is a central extension of $K'(1,4)$ by a one$-$dimensional center.\\
	The annihilation superalgebra of $K'_{4}$ is a central extension of $K(1,4)_{+}$ by a one$-$dimensional center $\C C$:
	\begin{align*}
	\mathcal{A}(K'_{4})=K(1,4)_{+} \oplus \C C.
	\end{align*}
	The extension is given by the $2-$cocycle $\psi \in  Z^{2}(K(1,4)_{+}, \C)$ which computed on basis elements returns non$-$zero values in the following cases only (up to skew-symmetry of $\psi$):
	\begin{align*}
	&\psi(1  ,\xi_{1234} )=-2,\\
	&\psi(\xi_i,\partial_{i} \xi_{1234}  )=-1.
	\end{align*} 
\end{prop}
We need a lemma in order to prove Proposition \ref{anniliK4}.
%We know that $K_{4}$ is not simple since $K_{4}=K'_{4} \oplus \mathbb{C}\xi_{1}\xi_{2}\xi_{3}\xi_{4}$. We construct $\widetilde{K'_{4}}=K'_{4}[y,y^{-1}]$ and $LieK'_{4}=\widetilde{K'_{4}} / \widetilde{\partial} \widetilde{K'_{4}}$. We denote the image of an element $ay^{m} \in \widetilde{K'_{4}}$ as $a_{m}$ in $\Lie K'_{4}$. 
\begin{lem}
	\label{lemcentral}
	The element $\partial \xi_{1234} y^{0} \in \Lie K'_{4}$ is central.
\end{lem}
\begin{proof}
	By \eqref{1,5} and \eqref{bracketannihilation} we have, for all $ay^{l} \in \Lie K'_{4}$, with $a \in K'_{4}$:
	\begin{align*}
	\big[\partial \xi_{1234}y^{0}, ay^{l}\big]=\left(\partial {\xi_{1234}}_{(0)} a\right)y^{l}=0.
	\end{align*}
	In the last equality we used the fact that $\left(\partial {\xi_{1234}}_{(0)} a\right)$ is computed as the restriction of the $0-$product in $K_{4}$ and (iii) in Definition \ref{definizionesuperalgebraconforme}.
\end{proof}
We adopt the following notation. Given a proposition $P$, we let
\begin{equation*}
\bigchi_{P}=
\begin{cases}
1 \quad \text{if $P$ is true,}\\
0 \quad \text{if $P$ is false.}
\end{cases}
\end{equation*}
\begin{rem}
	\label{bendef}
	Recall that by the definition of $\Lie K'_{4}$, for all $a \in K'_{4}$ and $m \in \Z$, we have  $\partial a y^{m}=-m a y^{m-1}$ and that $\xi_{1234} \notin  K'_{4}$ by Proposition \ref{xi1234}. Hence, every  monomial $\partial^{k} P(\xi)y^{n}$ can be represented  in $\Lie K'_{4}$ as a scalar multiple of a monomial $\xi_I y^n$ for some $I\neq 1234$ or of a monomial $\de \xi_{1234} y^n$.
	
	More precisely we have that the set $\{\xi_Iy^{m}, \, \partial \xi_{1234} y^{m}:\, m \in \Z, \, I\in \mathcal I_<,\, I\neq 1234\}$ is a basis of $\Lie K'_{4}$.
\end{rem}

\begin{proof}[Proof of Proposition \ref{anniliK4}]
	By Remark \ref{bendef} we know that there exists a unique linear map $\phi$ satisfying the prescribed conditions.
	It is clear from its definition that $\phi$ is surjective and so we only need to prove that $\phi$ is a morphism of Lie superalgebras. 
	We have to distinguish four cases:
\begin{enumerate}%[leftmargin=0.6cm]
\item Let $I,J\in \mathcal I_<$ with $I,J\neq 1234$ and $\xi_{IJ}\neq \pm \xi_{1234}$, and $m,n\in \Z$. In $\Lie K'_{4}$ we have, by \eqref{1,5} and (\ref{bracketannihilation}) and $j-$products (\ref{nproducts}):
		\begin{align*}
		[\xi_Iy^m,\xi_J y^n]&=\sum_{j \in \Z_{+}}\binom{m}{j}({\xi_I}_{(j)}\xi_J) y^{m+n-j}\\
		&=({\xi_I}_{(0)}\xi_J) y^{m+n}+m({\xi_I}_{(1)}\xi_J)y^{m+n-1}\\
		&=(|I|-2) \partial \xi_{IJ} y^{m+n}+(-1)^{|I|} \sum_{i=1}^{4}\partial_{i} \xi_I  \, \partial_{i} \xi_J \, y^{m+n}+m(|I|+|J|-4) \xi_{IJ} y^{m+n-1}\\
		&	=(2n-2m-n|I|+m|J|)  \xi_{IJ} y^{m+n-1} +(-1)^{|I|} \sum_{i=1}^{4}\partial_{i} \xi_I \, \partial_{i}  \xi_J \, y^{m+n}.
		\end{align*}
		Therefore, by (\ref{bracketliebis}), we have:
		\begin{align*}
		[\phi(\xi_Iy^m),\phi(\xi_J y^n)]&=\big[t^{m}\xi_I,t^{n}\xi_J\big]\\
		&=(2n-2m-n|I|+m|J|)t^{m+n-1}\xi_{IJ}+(-1)^{|I|}t^{m+n}\sum_{i=1}^4\de_i \xi_I\, \de_i \xi_J\\
		&=\phi([\xi_I y^m,\xi_J y^n]).
		\end{align*}
\item Let $I,J\in \mathcal I_<$ with $|I|,|J|\neq 4$ and $\xi_{IJ}= \xi_{1234}$, and $m,n \in \Z$. We proceed like in the previous case and we have
		\begin{align*}
		[\xi_Iy^m,\xi_J y^n]&=({\xi_I}_{(0)}{\xi_J}) y^{m+n}+m({\xi_I}_{(1)}\xi_J)y^{m+n-1}\\
		&=(|I|-2) \partial \xi_{1234} y^{m+n}.
		\end{align*}
		and so in $K'(1,4)$ we have
		\begin{align*}
		\left[\phi(\xi_Iy^m),\phi(\xi_J y^n)\right]&=\big[t^{m}\xi_I,t^{n}\xi_J\big]\\
		&=(2-|I|)(m+n)t^{m+n-1}\xi_{1234} \\
		&= \phi([\xi_Iy^m,\xi_J y^n]).
		\end{align*}
		
\item Let $m,n\in \Z$. We have $f=\partial \xi_{1}\xi_{2}\xi_{3}\xi_{4} y^{m}$ and $g=\partial \xi_{1}\xi_{2}\xi_{3}\xi_{4} y^{n}$ in $\Lie K'_{4}$, with $m,l \in \Z$. In $\Lie K'_{4}$ we have, using bracket (\ref{bracketannihilation}) and $n-$products (\ref{nproducts}):
		\begin{align*}
		[\de \xi_{1234}y^m,\de \xi_{1234}y^n]=\sum_{j \in \Z_{+}}\binom{h}{j}({\partial \xi_{1234}\,}_{(j)}\partial \xi_{1234}) y^{m+n-j}=0.
		\end{align*}
		by (iii) of Definition \ref{definizionesuperalgebraconforme}, \eqref{deadestra} and  (\ref{bracketliebis}). On the other hand
		\begin{align*}
		\left[\phi(\de \xi_{1234}y^m), \phi(\de \xi_{1234} y^n)\right]=\big[-m\xi_{1234} t^{m-1},-n\xi_{1234} t^{n-1}\big]=0,
		\end{align*}
		by \eqref{bracketliebis}.
		\item Finally, let $J\in \mathcal I _<$, $J\neq 1234$ and $m,n\in \Z$. 
		First, we point out that $({\partial \xi_{1234}\,}_{(j)} \xi_J)=-j({\xi_{1234}\,}_{(j-1)} \xi_J)=0$ for all $j\geq 2$ by \eqref{nproducts}. Therefore in $\Lie K'_{4}$ we have
		\begin{align*}
		\big[\partial \xi_{1234} y^{m},\xi_J y^{n}\big]&=({\partial \xi_{1234}}_{(0)} \xi_J)y^{m+n}+m({\partial \xi_{1234}\,}_{(1)} \xi_J)y^{m+n-1}\\
		&=-m ({ \xi_{1234}\,}_{(0)} \xi_J) y^{m+n-1}\\
		&=-2m \bigchi_{J=\emptyset} \, \partial \xi_{1234} \,  \, y^{m+n-1}-m\sum_{i=1}^{4} \partial_{i}\xi_{1234} \partial_{i} \xi_J \, y^{m+n-1}.
		\end{align*}
		In $K'(1,4)$ we have, using bracket (\ref{bracketliebis}):
		\begin{align*}
		\left[\phi(\de \xi_{1234}y^m),\phi(\xi_J y^n)\right]&=\big[- m\xi_{1234} t^{m-1},\xi_J t^{n}\big]\\
		&=-\bigchi_{J=\emptyset}m (-2n-2m+2) \,  t^{m+n-2}\xi_{1234}  \, -m \sum_{i=1}^{4} \partial_{i}\xi_{1234} \, \partial_{i}\xi_J \, t^{m+n-1}\\
		%&=-m(-2l-2m+2)t^{m+l-2}\xi_{1}\xi_{2}\xi_{3}\xi_{4} \chi_{P(\xi) \in \mathbb{C}}-m \sum_{i=1}^{N} \partial_{i}(\xi_{1}\xi_{2}\xi_{3}\xi_{4}) \, \partial_{i}(P(\xi)) \, t^{m+l-1}\\
		&=\phi([\de \xi_{1234}y^m,\xi_J y^n]).
		\end{align*}
	\end{enumerate}
	%\end{proof}
	The previous computations imply that the kernel of the map $\phi: \Lie K'_{4}  \longrightarrow  K'(1,4)$ is $\Ker \phi= \langle \partial \xi_{1}\xi_{2}\xi_{3}\xi_{4} \rangle $ and so the following sequence is exact:
	\begin{align*}
	0 \rightarrow \langle \partial \xi_{1}\xi_{2}\xi_{3}\xi_{4} \rangle \xrightarrow{i} \Lie K'_{4} \xrightarrow{\phi} K'(1,4) \rightarrow 0.
	\end{align*}
	By Lemma \ref{lemcentral} the Lie superalgebra $\Lie K'_{4}$ is therefore a central extension of $K'(1,4)$ by the one$-$dimensional center $\langle \partial \xi_{1}\xi_{2}\xi_{3}\xi_{4} \rangle$.\\
	In particular, we point out that $\phi: \Lie K'_{4} / \C \partial \xi_{1}\xi_{2}\xi_{3}\xi_{4} \rightarrow K'(1,4) $ is an isomorphism.
	%\begin{proof}[Proof of Proposition \ref{anniliK4}]
	In the previous computations we computed all the possible brackets between monomials in $\Lie K'_{4}$, therefore in particular all the possible brackets between monomials in $\mathcal{A}(K'_{4})$ and we can observe that the central element $\partial \xi_{1234}$ lies in the support of the bracket of two basis elements  only in the case $(2)$ of the previous computations.
	The other parts of the statement follow.
	
\end{proof}

\section{Verma modules}
\label{Verma modules}
In this section we study the action of $\mathfrak{g}:=\mathcal{A}(K'_{4})=K(1,4)_{+} \oplus \C C$ on a finite Verma module $\Ind (F)$, where $F$ is a finite$-$dimensional irreducible $\mathfrak{g}_{\geq 0}-$module, on which $\mathfrak{g}_{> 0}$  acts trivially.
The grading on $\mathfrak{g}$ is the standard grading of $K(1,4)_{+} $ and $C$ has degree $0$. We have:
\begin{align*}
&\g_{-2} =\left< 1 \right>, \\
&\g_{-1} =\left< \xi_{1},\xi_{2},\xi_{3},\xi_{4} \right> ,\\
&\g_{0} =\left< \left\{C,t, \xi_{ij}:\, 1 \leq i<j \leq 4 \right\}\right>.
\end{align*}
%According to the notation of \cite{kac1}, we call ${ t }:=t$ and $F_{ij}:=-\xi_{i} \xi_{j}$.
\begin{rem}
	The annihilation superalgebra $\mathfrak{g}$ satisfies conditions L1, L2, L3 of Proposition \ref{keythmannihi}. Indeed:
	\begin{enumerate}
		\item [L1.] This is obvious.
		\item [L2.] The element $t$ is a grading element, i.e. $[t,a]=\deg (a) a $ for all $a \in \g$. Hence, by Remark \ref{gradingelement}, $t$ satisfies condition L2.
		\item [L3.] The element $\Theta$ is chosen as $-\frac{1}{2}\xi_{\emptyset}=-\frac{1}{2} \in \mathfrak{g}_{-2}$. Indeed for all $m\geq 0$ and $I\in \mathcal I_<$ we have $t^{m} \xi_{I} \in \mathfrak{g}_{2m+|I|-2}$ and 
		\[t^{m} \xi_{I}=-\frac{1}{m+1}[\Theta, t^{m+1} \xi_{I}]\] and $C=[\Theta, \xi_{1234}]$.
	\end{enumerate}
\end{rem}
\begin{rem}
	Since $\Ind (F) \cong U(\mathfrak{g}_{<0}) \otimes F$, it follows that $\Ind (F) \cong \C[\Theta] \otimes \inlinewedge(4) \otimes F$. Indeed, let us denote by $\eta_{i}$ the image in $U(\g)$ of $\xi_{i} \in  \inlinewedge(4) $, for all $i \in \left\{1,2,3,4\right\}$.
	In $U(\g)$ we have that $\eta_{i}^{2}=\Theta$, for all $i \in \left\{1,2,3,4\right\}$: since $[\xi_{i},\xi_{i}]=-1$ in $\g$, we have $\eta_{i}\eta_{i}=-\eta_{i}\eta_{i}-1$ in $U(\g)$.
\end{rem}
From now on it is always assumed that $F$ is a finite$-$dimensional irreducible $\mathfrak{g}_{\geq 0}-$module. 

We will study the action of $\g$ on $\Ind (F)$ using the $\lambda-$action notation by Proposition \ref{propcorrispmoduli}:
\begin{align*}
{\xi_I\,}_{\lambda}(g \otimes v)= \sum_{j \geq 0} \frac{\lambda^{j}}{j!}t^{j}\xi_I .(g \otimes v),
\end{align*}
for $I\in \mathcal I$, $g \in U(\g_{<0})$ and $v \in F$.
In order to find an explicit formula for ${\xi_I\,}_{\lambda}(g \otimes v)$ we need some preliminary lemmas.

We will make the following slight abuse of notation: if $I,J\in \mathcal I_{\neq}$ we will denote by $I \cap J$ (resp. $I \setminus J$) the increasingly ordered sequence whose elements are the elements of the intersection of the underlying sets of $I$ and $J$ (resp. the elements of the difference of the underlying sets of $I$ and $J$). We will say $I \subseteq J$ when the underlying set of $I$ is contained in the underlying set of $J$. Analogously we will denote by $I ^{c}$ the increasingly ordered sequence whose elements are the elements of the complement of the underlying set of $I$. Given $I=(i_{1}, \,i_{2},\cdots i_{k})\in \mathcal I_{\neq}$, we will use the notation  $\eta_{I}$ to denote the element  $\eta_{i_{1}} \eta_{i_{2}} \cdots \eta_{i_{k}} \in  U(\g_{<0})$ and we will denote  $|\eta_{I}|=|I|=k$. We will denote $\xi_{*}=\xi_{1234}$ (resp. $\eta_{*}=\eta_{1234}$). Given $I=(i_{1}, \, i_{2},  \cdots i_{k})$ and $I^{c}=(j_{k+1}, \, j_{k+2} ,\cdots j_{4})$, we will denote by $\varepsilon _{I}$ the sign of the permutation
\[
\bigl(\begin{smallmatrix}
1 & 2 & \cdots & k & k+1 & \cdots & 4 \\
i_{1} & i_{2} & \cdots & i_{k} & j_{k+1} &  \cdots &  j_{4}
\end{smallmatrix}\bigr).
\]
%we will denote by $\sgn (i_{1}i_{2}...i_{k})$ the number of exchanges that we need to sort the integers $i_{1} $ $i_{2}$... $i_{k} $ in increasing order.
We will also use the following notation: if $(i_1,\ldots,i_k)\in \mathcal I_{\neq}$ and $i\in \{1,2,3,4\}$ we let 

\[
\de_i \eta_{i_1,\ldots,i_k}=\begin{cases}
(-1)^{j+1} \eta_{i_1,\ldots,\hat{i_j},\ldots,i_k}& \textrm{if $i=i_j$ for some (necessarily unique) $j$}\\
0& \textrm{otherwise.}
\end{cases}
\]
and for $a \in \C$, $I=(i_{1} ,\, i_{2}, \cdots i_{k}),J\in \mathcal I_{\neq}$: %, for $I =\left\{i_{1}<i_{2} \cdots <i_{k}\right\}\subseteq \left\{,2,3,4 \right\}$, $S \subseteq \left\{1,2,3,4 \right\}$, $f=\xi_{J} \in \inlinewedge(4)$ where $J =\left\{j_{1}<j_{2} \cdots <j_{s}\right\}\subseteq \left\{1,2,3,4 \right\}$ and $i,j \in J$, $i<j$:
\begin{align*}
\partial_{I}\eta_{J}&=\partial_{i_{1}}\partial_{i_{2}}\dots \partial_{i_{k}}\eta_{J} \ \ &\partial_{I}\xi_{J}&=\partial_{i_{1}}\partial_{i_{2}}\dots \partial_{i_{k}}\xi_{J} ;\\
\partial_{a \xi_{I}}\eta_{J}&=a \partial_{I}\eta_{J} \ \ &\partial_{a \xi_{I}}\xi_{J}&=a \partial_{I}\xi_{J};\\
%&\partial_{\partial_{i}f}\xi_{S}=\partial_{(-1)^{i-1}\xi_{j_{1}} \xi_{j_{2}} \cdots \widehat{\xi_{i}}\cdots \xi_{j_{s}}} =(-1)^{i-1}\partial_{\xi_{j_{1}}} \partial_{\xi_{j_{2}}} \cdots \widehat{\partial_{\xi_{i}}}\cdots \partial_{\xi_{j_{s}}} \xi_{S};\\
%&\partial_{\partial_{i}\partial_{j}f}\xi_{S}=(-1)^{i+j}\partial_{\xi_{j_{1}}} \partial_{\xi_{j_{2}}} \cdots \widehat{\partial_{\xi_{i}}} \cdots \widehat{\partial_{\xi_{j}}}\cdots \partial_{\xi_{j_{s}}} \xi_{S};\\
\partial_{\emptyset}\eta_{S}&=\eta_{S} \ \ &\partial_{\emptyset}\xi_{S}&=\xi_{S}.
\end{align*}
Given $I,J\in \mathcal I_{\neq}$ we let
\begin{align*}
\xi_{I} \star \eta_{J}= \bigchi_{I \cap J= \emptyset}\eta_{IJ},\\
\eta_{J} \star \xi_{I}= \bigchi_{I \cap J= \emptyset}\eta_{JI}.
\end{align*}
and we extend this notation by linearity in both arguments.

We observe that in $\g$, by \eqref{bracketliebis} and Proposition \ref{anniliK4} we have
\begin{align*}
[t^{m}\xi_{I},\xi_{r}]=-mt^{m-1}\xi_{Ir}+(-1)^{|I|}t^{m}\partial_{r}\xi_{I}+ \psi(t^{m}\xi_{I},\xi_{r})C.
\end{align*}
and in particular
\begin{align}
\label{bracket}
[t^{m}\xi_{I},\xi_{r}]=-mt^{m-1}\xi_{Ir}+(-1)^{|I|}t^{m}\partial_{r}\xi_{I}+ \bigchi_{m=0}\bigchi_{r=I^{c}}\varepsilon_{I} C
\end{align}
for all $I\in \mathcal I_{\neq}$, $m\geq 0$ and $r\in \{1,2,3,4\}$.
\begin{lem}\label{deg3} Let $I,L\in \mathcal I_{\neq}$, $v \in F$ and $m \geq 3$. We have:
	\[
	t^{m} \xi_{I}(\eta_{L} \otimes v)=\begin{cases}
	-6\varepsilon_L\otimes Cv &\textrm{if $m=3$, $|I|=0$ and $|L|=4$, }\\
	0&\textrm{otherwise.}
	\end{cases}
	\]
	
\end{lem}
\begin{proof}
	We can always assume, without loss of generality, that $\eta_{L}=\eta_{J}\eta_{K}$ with $I \cap J= \emptyset$, $K \subseteq I$.\\
	We first point out that $t^{m} \xi_{I}(\eta_{L} \otimes v)=0$ when $m >3$ because $\deg (t^{m} \xi_{I})=2m+|I|-2>4 \geq \deg(\eta_{L})$.
	
If $  m=3$, $|I|>0$, $t^{3} \xi_{I} (\eta_{L} \otimes v)=0$ because $\deg (t^{3} \xi_{I})=2m+|I|-2>4 \geq \deg(\eta_{L})$.

If $m=3$, $|I|=0$  and $|L| \neq 4 $,  $t^{3} (\eta_{L} \otimes v)=0$ because $\deg (t^{3} )=4>\deg(\eta_{L})$.

If $m=3$, $|I|=0$ and $|L|=4$ we can assume $L=1234$ without loss of generality and by \eqref{bracket}, we have
	\begin{align*}
	t^{3}(\eta_{1234}\otimes v)&=-3(t^{2}\xi_{1})\eta_{234}\otimes v-3\eta_{1}(t^{2}\xi_{2})\eta_{34}\otimes v-3\eta_{12}(t^{2}\xi_{3})\eta_{4}\otimes v-3\eta_{123}(t^{2}\xi_{4})\otimes v\\
	&=6(t\xi_{12})\eta_{34}\otimes v\\
	&=-6(\xi_{123})\eta_{4}\otimes v\\
	&=-6  \otimes Cv.
	\end{align*}
\end{proof}
Lemma \ref{deg3} describes the terms of degree at least 3 in the variable $\lambda$ in the $\lambda$-action of $K'_4$ on a Verma module. Next target is to study the terms of degree 0, 1 and 2: this will be accomplished in Lemmas \ref{lambdadeg0}, \ref{lambdadeg1} and \ref{lambdadeg2} respectively. But we first state a technical lemma.

\begin{lem}
	\label{lemma 0 action}
	Let $I,J,K \in \mathcal I_{\neq}$ with $I \cap J =\emptyset$, $K \subseteq I$. We have:
	\begin{align*}
	\xi_{I}(\eta_{JK} \otimes v)%=&% \sum_{L \subseteq K} (-1)^{|I|(|J|+|K|)+|L|(|L|-1)/2-|L|(|K|-|L|)} \eta_{J} (\partial_{L}\eta_{K})(\partial_{L}\xi_{I}) \otimes  v\\
	%&+ \bigchi_{I=(i)^{c}} \bigchi_{J=(i)}(-1)^{i} \eta_{K} \otimes C v.\\
	=& \sum_{L \subseteq K} (-1)^{|I|(|J|+|K|)+|L|(|L|-1)/2-|L|(|K|-|L|)} \eta_{J} \partial_{L}\eta_{K}\otimes\partial_{L}\xi_{I}.   v\\
	&+ \bigchi_{|I|=3} \bigchi_{J=I^c}\, \varepsilon_{I} \, \eta_{K} \otimes C v. 
	\end{align*}
\end{lem}
\begin{proof}
	From repeated applications of \eqref{bracket} we have
	\begin{align}
	\label{condizionegrado0}
	\xi_{I}(\eta_{J} \eta_{K}) \otimes v=(-1)^{|I||J|} \eta_{J} \xi_{I} \eta_{K} \otimes v+ \bigchi_{|I|=3}\bigchi_{J=I^c} \, \varepsilon_{I} \eta_{K} \otimes C v. 
	\end{align}
	Indeed, from \eqref{bracket}, if $|I|=1, 2$, or $|I|=3$ with $J\neq I^c$ , then $[\xi_{I},\xi_{r}]=0$ for all $r \in J$ and formula \eqref{condizionegrado0} is straightforward. In the case $|I|=3$ and $J=I^{c}$, using \eqref{bracket}, we have:
	\begin{align*}
	\xi_{I}(\eta_{I^{c}} \eta_{K}) \otimes v=-\eta_{I^{c}} \xi_{I} \eta_{K} \otimes v+ \bigchi_{|I|=3}\bigchi_{J=I^c} \varepsilon_{I}  \eta_{K} \otimes C v. 
	\end{align*}
	
	The rest of the proof is the same as the proof of Lemma A.2 in \cite{kac1} and it is done  by induction on $|K|$ using formula \eqref{condizionegrado0} and is therefore omitted.
\end{proof}
\begin{lem}\label{lambdadeg0}Let $I,L\in \mathcal I_{\neq}$. We have:
	\begin{align*}
	\xi_I(\eta_L \otimes v)=& (-1)^{|I|}(|I|-2) \Theta \partial_{I} \eta_L \otimes v +\sum_{i=1}^{4} \partial_{(\partial_{i}\xi_I)}(\xi_{i} \star \eta_L)\otimes v+(-1)^{|I|} \sum_{i<j}\partial_{(\partial_{ij}\xi_I)}\eta_L \otimes \xi_{j,i}. v \\
	&+ \bigchi_{|I|=3} \, \varepsilon_{I} \, \partial_{I^{c}}\eta_L\otimes C v.
	%& + \bigchi_{I=(jkl)=(i)^{c}}   \bigchi_{ \partial_{i} g \neq 0}(-1)^{i}(-1)^{\sgn(jkl)}  \partial_{i} g\otimes C v .
	\end{align*}
\end{lem}
\begin{proof}
	The proof is analogous to the proof in \cite{kac1} of Lemma A.3, and it is based on Lemma \ref{lemma 0 action}. The extra term  in $C$ is due to the additional term of Lemma \ref{lemma 0 action}, which is not present in Lemma A.2 of \cite{kac1}. 
\end{proof}
Now we study the term of degree 1 in $\lambda$ of the $\lambda-$action.
\begin{lem}\label{lambdadeg1}
	Let $I,L\in \mathcal I_{\neq}$. We have:
	\begin{align*}
	t\xi_I(\eta_L \otimes v)= &(-1)^{|I|}\partial_{I}\eta_L \otimes { t }. v+(-1)^{|I|+|L|} \sum^{4}_{i=1}(\partial_{Ii}\eta_L \star \xi_{i}) \otimes v \\
	&-\sum_{i \neq j} \partial_{ \partial_{i}\xi_I}(\partial_{j}\eta_L) \otimes \xi_{i,j}. v%+(-1)^{\sgn(k l i)} (-1)^{j}\bigchi_{I=(k,l)=(i,j)^{c}}   \bigchi_{ \partial_{i} \partial_{j} g \neq 0} \partial_{i} \partial_{j}g \otimes C v.\\
	+\bigchi_{|I|=2} \, \varepsilon_{I} \, \partial_{I^{c}}\eta_L\otimes C v.
	\end{align*}
\end{lem}
\begin{proof}
	Without loss of generality we can suppose that $\eta_L=\eta_{J}\eta_{K}$ with $I \cap J =\emptyset$, $K \subseteq I$. Let us prove that:
	\begin{align}
	\label{formulagrado1}
	t\xi_{I}(\eta_{J}\eta_{K} \otimes v)=&(-1)^{|I||J|} \eta_{J}(t\xi_{I})\eta_{K} \otimes v+\sum^{4}_{j=1}(-1)^{|I||J|-|I|+|J|}(\partial_{j}\eta_{J})\xi_{Ij}\eta_{K}\otimes v   \\ \nonumber
	&+ \bigchi_{|I|=2}    \, \varepsilon_{I} \partial_{I^{c}}\eta_{JK} \otimes C v.
	\end{align}
	The formula is the same as the relation proved for $K(1,N)_{+}$ in the proof of Lemma A.4 of \cite{kac1}, except for an additional term in $C$. We point out that a term with $C$ is involved only if $|I|=2$ and $|J|=2$. 
	Let us prove \eqref{formulagrado1} by induction on $|J|$. If $|J|=0$, \eqref{formulagrado1} is straightforward. Let $\tilde J$ be such that $|\tilde J|>0$ and $\widetilde{J}\cap I=\emptyset$. Let $J,s$ be such that $\eta_{\widetilde{J}}=\eta_{J}\eta_{s}$. We have, using \eqref{formulagrado1} for $\eta_{J}$, that:
	\begin{align*}
	t\xi_{I}(\eta_{J}\eta_{s}\eta_{K} \otimes v)=&(-1)^{|I||J|} \eta_{J}(t\xi_{I})\eta_{s}\eta_{K} \otimes v+\sum^{4}_{j=1}(-1)^{|I||J|-|I|+|J|}(\partial_{j}\eta_{J})\xi_{Ij}\eta_{s}\eta_{K}  \otimes v \\ 
	&+ \bigchi_{|I|=2}    \, \varepsilon_{I}( \partial_{I^{c}}\eta_{J}) \eta_{s}\eta_{K} \otimes C v.
	\end{align*}
	Notice that, since we are supposing  $\eta_{\widetilde{J}}=\eta_{J}\eta_{s}$ with $\widetilde{J}\cap I=\emptyset$  and $s \notin J$, the term \\
	$\bigchi_{|I|=2}    \, \varepsilon_{I} (\partial_{I^{c}}\eta_{J}) \eta_{s}\eta_{K} \otimes C v$ is 0 because if $|I|=2 $, then $|J|<2$.
	We have, using \eqref{bracket}, that:
	\begin{align*}
	t\xi_{I}(\eta_{J}\eta_{s}\eta_{K} \otimes v)=&(-1)^{|I|(|J|+1)} \eta_{J}\eta_{s}(t\xi_{I})\eta_{K} \otimes v-(-1)^{|I||J|} \eta_{J}\xi_{Is}\eta_{K} \otimes v \\
	&+\sum^{4}_{j=1}(-1)^{|I||J|-|I|+|J|+|I|+1}(\partial_{j}\eta_{J})\eta_{s}\,\xi_{Ij}\,\eta_{K} \otimes v \\
	&-(-1)^{|J|}\bigchi_{|I|=2}\bigchi_{|J|=1} \varepsilon_{I}\partial_{I^{c}}\eta_{JsK} \otimes C v.
	\end{align*}
	We observe that:
	\begin{align*}
	-(-1)^{|I||J|} \eta_{J}\xi_{Is}\eta_{K} \otimes v&=(-1)^{|I||J|+1+|J|}(\partial_{s}\eta_{\widetilde{J}})\xi_{Is}\eta_{K} \otimes v \\
	&=(-1)^{|I||\widetilde{J}|-|I|+|\widetilde{J}|}(\partial_{s}\eta_{\widetilde{J}})\xi_{Is}\eta_{K} \otimes v.
	\end{align*} 
	Therefore:
	\begin{align*}
	t\xi_{I}(\eta_{J}\eta_{s}\eta_{K} \otimes v)=&(-1)^{|I|(|\widetilde{J}|)} \eta_{\widetilde{J}}(t\xi_{I})\eta_{K} \otimes v
	+\sum^{4}_{j=1}(-1)^{|I||\widetilde{J}|-|I|+|\widetilde{J}|}(\partial_{j}\eta_{\widetilde{J}})\,\xi_{Ij}\,\eta_{K} \otimes v\\
	&+\bigchi_{|I|=2} \varepsilon_{I}\partial_{I^{c}}\eta_{\widetilde{J}K} \otimes C v.
	\end{align*}
	Hence, formula \eqref{formulagrado1} is proved.
	The rest of the proof is analogous to the proof of Lemma A.4 in \cite{kac1} and it is based on \eqref{formulagrado1}.  
\end{proof}
Now we study the term of degree 2 in $\lambda$ of the $\lambda-$action.
\begin{lem}\label{lambdadeg2}
	Let $I,L\in \mathcal I_{\neq}$ and $v\in F$. We have
	\begin{align*}
	\frac{1}{2}t^{2}\xi_I(\eta_L \otimes v)=&(-1)^{|I|}  \sum_{i<j} \partial_{Iij}\eta_L \otimes \xi_{j,i}. v
	-\bigchi_{|I|=1}  \, \varepsilon_{I} \, \partial_{I^{c}}\eta_L \otimes Cv.
	\end{align*}
\end{lem}
\begin{proof}
	As before, without loss of generality, we can suppose that $\eta_L=\eta_{J}\eta_{K}$ with $I \cap J =\emptyset$, $K \subseteq I$.
	Let us prove that:
	\begin{align}
	\label{formulagrado2}
	\frac{1}{2}t^{2}&\xi_{I}(\eta_{J} \eta_{K} \otimes v)\\ \nonumber
	&= - \bigchi_{I=K}\sum_{i<j}(-1) ^{|I||J|+|I|(|I|+1)/2} \partial_{ij}\eta_{J } \,  \partial_{I}\eta_{K } \otimes \xi_{ij} . v 
	%&\quad  + \bigchi_{I=(i)}\bigchi_{J=(j,k,l)=(i)^{c}}  (-1)^{l}(-1)^{\sgn(ijk)}  \eta_{K} \otimes Cv.\\ \nonumber
	- \bigchi_{|I|=1} \, \varepsilon_{I} \,\partial_{I^{c}} \eta_{JK} \otimes Cv.
	\end{align}
	In order to establish \eqref{formulagrado2}, we need to prove the following:
	\begin{equation}
	\label{formulagrado2app}
	\big(\frac{1}{2}t^{2}\xi_{I}\big)\,\eta_{J} =\sum_{S\in \mathcal I_{<}:\, |S|\leq 2} \pm \frac{1}{(2-|S|)!}\de_S \eta_J \,(t^{2-|S|}\xi_{IS})
	 - \bigchi_{|I|=1} \, \varepsilon_{I} \,\partial_{I^{c}} \eta_{J}  C. 
	\end{equation}
	
	We prove \eqref{formulagrado2app} by induction on $|J|$. If $|J|=0$, the result is straightforward.
	
	Let   us consider $\eta_{\widetilde{J}}=\eta_{J}\eta_{r}$ with $\widetilde{J}\cap I=\emptyset$ and $r \notin J$. We have, using \eqref{formulagrado2app} for $\eta_{J}$, that:
	\[
	\big(\frac{1}{2}t^{2}\xi_{I}\big)\,\eta_{J} \eta_{r} =\sum_{S\in \mathcal I_{<}:\, |S|\leq  2} \frac{1}{(2-|S|)!} \partial_{S}\eta_{J} \,(t^{2-|S|}\xi_{IS}) \eta_{r}- \bigchi_{|I|=1} \, \varepsilon_{I}\,(\partial_{I^{c}}\eta_{J}) \eta_{r}  C.
	\]
	Notice that, since we are supposing  $\eta_{\widetilde{J}}=\eta_{J}\eta_{r}$ with $\widetilde{J}\cap I=\emptyset$ and $r \notin J$, the term \\
	$- \bigchi_{|I|=1} \, \varepsilon_{I} \,(\partial_{I^{c}} \eta_{J}) \eta_{r} C$ is 0 because if $|I|=1 $, then $|J|<3$.
	Now, by \eqref{bracket},  we have
	
	\begin{align*}
	\big(\frac{1}{2}t^{2}\xi_{I}\big)\,\eta_{J} \,\eta_{r} =&
	\sum_{S:\,|S|\leq 2}\pm \frac{1}{(2-|S|)!}(\de_S\eta_J)\eta_r(t^{2-|S|}\xi_{IS})
	+\sum_{|S|\leq 2}\pm \frac{2-|S|}{(2-|S|)!}\de_S\eta_J \,(t^{1-|S|}\xi_{ISr})\\
	&\pm\bigchi_{|I|=1}\bigchi_{|J|=2}\varepsilon_{\tilde J} C\\
	=&\sum_{S:\,|S|\leq 2,\, r\notin S}\pm \frac{1}{(2-|S|)!}\de_S\eta_{Jr} \,(t^{2-|S|}\xi_{IS})\pm\sum_{S:\, |S|\leq 1}\de_{Sr}\eta_{\tilde J}(t^{1-|S|}\xi_{ISr})\\
	& \pm\bigchi_{|I|=1}\bigchi_{|J|=2}\varepsilon_{\tilde J} C\\
	=&\sum_{S:\, |S|\leq 2}\pm \frac{1}{(2-|S|)!} \partial_{S}\eta_{\widetilde{J}}(t^{2-|S|} \xi_{IS}) \eta_{K}\otimes v\pm \bigchi_{|I|=1} \bigchi_{|J|=2}\varepsilon_{\tilde J} C
	\end{align*}
Now we compute explicitly the sign of the last summand above.
Hence we consider $I$ with $|I|=1$ and $|\widetilde{J}|=3$. For $I=(i)$ and $\widetilde{J}=(j,k,l)=I^{c}$, by repeated applications of \eqref{bracket}, we have:
\begin{align*}
\big(\frac{1}{2}t^{2}\xi_{i}\big)\eta_{j}\eta_{k}\eta_{l} &=-(t \xi_{ij})\eta_{k}\eta_{l} +\eta_{j}(t\xi_{ik})\eta_{l} -\eta_{j}\eta_{k}(t\xi_{il}) -\eta_{j}\eta_{k}\eta_{l}\big(\frac{1}{2}t^{2}\xi_{i} \big) \\
&= \xi_{ijk}\eta_{l} +\eta_{k}( \xi_{ijl}) -\eta_{k}\eta_{l}(t \xi_{ij}) -\eta_{j} \xi_{ikl}+\eta_{j}\eta_{l} (t\xi_{ik})-\eta_{j}\eta_{k}(t\xi_{il}) -\eta_{j}\eta_{k}\eta_{l}\big(\frac{1}{2}t^{2}\xi_{i} \big) \\
&= -\eta_{l}  \xi_{ijk}+\varepsilon_{I}C+\eta_{k} \xi_{ijl} -\eta_{kl}(t \xi_{ij}) -\eta_{j} \xi_{ikl})+\eta_{jl} (t\xi_{ik})-\eta_{jk}(t\xi_{il}) -\eta_{jkl}\big(\frac{1}{2}t^{2}\xi_{i} \big) \\
&= \sum_{S: |S|\leq 2} \pm \frac{1}{(2-|S|)!}\partial_{S}\eta_{\widetilde{J}}(t^{2-|S|} \xi_{IS})-\varepsilon_{I}\, \partial_{I^{c}}\eta_{\widetilde{J}}   C,
\end{align*}
The result follows with the simple verification that the coefficient of $C$ above agrees with the coefficient of $C$ in \eqref{formulagrado2app}.

Now we observe that if $|S|=0,1$ then $\deg(t^{2} \xi_{I} \xi_{S})>\deg(\eta_{K})$ and $\deg(t \xi_{I} \xi_{S})>\deg(\eta_{K})$ and therefore, by \eqref{formulagrado2app}, we have

	\begin{align*}
	&\frac{1}{2}t^{2}\xi_{I}(\eta_{J} \eta_{K} \otimes v)= \sum_{S\,: |S|=2} \pm (\partial_{S}\eta_{J})  \xi_{IS} \eta_{K}\otimes v - \bigchi_{|I|=1} \, \varepsilon_{I} \,\partial_{I^{c}} \eta_{J} \otimes Cv.
	\end{align*}
	Proceeding as in the proof of Lemma A.5 in \cite{kac1}, one can show that the signs in this sum do not depend on $S$ and are all equal to  $ -  (-1) ^{|I||J|}$. It follows that this relation reduces to:
	\begin{align}
	\label{formulagrado2appnew}
	&\frac{1}{2}t^{2}\xi_{I}(\eta_{J} \eta_{K} \otimes v)=-  (-1) ^{|I||J|} \sum_{i<j} \partial_{ij}\eta_{J}( \xi_{Iij} \big) \eta_{K}\otimes v-\bigchi_{|I|=1} \varepsilon_{I} \partial_{I^{c}}\eta_{J} \eta_{K} \otimes Cv.
	\end{align}
	%follows from repeated applications of \eqref{bracket} to $(t^{2-i} \xi_{I} \xi_{S})$ and $\xi_{l}$ with $l \in K$, for which $[t^{2-i} \xi_{I} \xi_{S},\xi_{l}]=(-1)^{|I|+|S|}t^{m-i}\partial_{l}( \xi_{I} \xi_{S})$ and repeated applications of \eqref{bracket} to the results of these brackets and the others $\xi_{l}$'s with $l \in K$. \\
	Formula \eqref{formulagrado2} can be proved using \eqref{formulagrado2appnew}, \eqref{bracket} and induction on $|K|$. The proof is similar to the proof of \eqref{formulagrado2appnew}.
	Finally, the rest of the proof is analogous to the proof of Lemma A.5 in \cite{kac1} and it is based on \eqref{formulagrado2}.  
\end{proof}
It is convenient to summarize the previous lemmas in the following result.
\begin{prop}
	\label{action}
	Let $I,L\in \mathcal I_{\neq}$. The $\lambda-$action has the following expression:
	\begin{align*}
	{\xi_I}_{\lambda}(\eta_L \otimes v)=&(-1)^{|I|}(|I|-2) \Theta \partial_{I} \eta_L \otimes v +\sum_{i=1}^{4} \partial_{(\partial_{i}\xi_I)}(\xi_{i} \star \eta_L)\otimes v\\
	&+ (-1)^{|I|} \sum_{i<j} \partial_{(\partial_{ij}\xi_I)}\eta_L \otimes \xi_{j,i}.v   + \bigchi_{|I|=3} \, \varepsilon_{I} \, \partial_{I^{c}}\eta_L\otimes C v\\
	&+\lambda \bigg( (-1)^{|I|}\partial_{I}\eta_L \otimes { t }. v+(-1)^{|I|+|L|} \sum^{4}_{i=1}(\partial_{Ii}\eta_L \star \xi_{i}) \otimes v \\
	& +\sum_{i \neq j} \partial_{ \partial_{i}\xi_I}(\partial_{j}\eta_L) \otimes \xi_{j,i}. v + \bigchi_{|I|=2}    \, \varepsilon_{I} \,\partial_{I^{c}}\eta_L\otimes C v \bigg)\\
	&+\lambda^{2} \Big( (-1)^{|I|} \sum_{i<j}  \partial_{Iij} \eta_L \otimes \xi_{j,i} v -\bigchi_{|I|=1}  \, \varepsilon_{I} \, \partial_{I^{c}}\eta_L \otimes Cv \Big) \\
	& +\lambda^{3}\left( -\bigchi_{|I|=0} \, \partial_{I^{c}} \eta_L \otimes C v \right).
	\end{align*}
\end{prop}

 For $\eta_{I} \in \inlinewedge(4)$ we indicate with $\overline{\eta_{I}}$ its Hodge dual in $U(\g_{<0})$, i.e. the unique monomial such that $\overline{\eta_{I}}  \star \xi_{I}=\eta_{1234}$. Then we extend by linearity the definition of Hodge dual to elements $\sum_{I}\alpha_{I}\eta_{I} \in U(\g_{<0})$ and we set $\overline{\Theta^{k}\eta_{I}}=\Theta^{k}\overline{\eta_{I}}$. \\
 %For notational convenience we will denote the Hodge dual $\overline \eta_I$ by $\eta_I$ and we extend in a natural way the $\star$ product and the $\delta_i$ actions also to elements $\eta_I$'s so that for example $\xi_1\star \eta_{234}=-\eta_{234}\xi_1=\eta_{1234}$ and $\xi_1\star \eta_{12}=0$, $\de_2 \eta_{124}=-\eta_{14}$.
We recall Lemma 4.2 from \cite{kac1}.
\begin{lem}\label{lem4}
	For $f\in \inlinewedge (4)$, $L \in \mathcal I_{\neq}$, $i \in \left\{ 1,2,3,4\right\}$, we have:
	\begin{align}
	\label{lem4.2dualeprima}
	\overline{\partial_{i}\eta_L}& =\overline{\eta_L} \star \xi_{i}=(-1)^{|L|}\xi_{i} \star \overline{\eta_L}, \\
	\label{lem4.2dualeseconda}
	\overline{\partial_{f}\eta_L}& =(-1)^{(|f|(|f|-1)/2)+|f||L|}f\star \overline{\eta_L}, \\
	\label{lem4.2dualeterza}
	\overline{\xi_{i} \star \eta_L}& =-(-1)^{|L|}\partial_{i}\overline{\eta_L} , \\
	\label{lem4.2dualequarta}
	\overline{\eta_L \star \xi_{i}}& =-\partial_{i}\overline{\eta_L} .
	\end{align}
\end{lem}
 Next result is a consequence of Proposition \ref{action}.
\begin{prop}
	\label{actiondual}
	Let $I,L\in \mathcal I_{\neq}$. Let $T$ be the vector space isomorphism $T: \Ind (F) \rightarrow \Ind (F)$ defined by $T(g \otimes v)=\overline{g} \otimes v$, for all $g \in U(\g_{<0})$, $v \in F$. Then:
	\begin{align*}
	& (T\circ {\xi_I}_{\lambda} \circ T^{-1}) (\eta_L \otimes v)\\
	&=(-1)^{(|I|(|I|+1)/2)+|I||L|} \bigg\{ (|I|-2) \Theta (\xi_I \star \eta_L) \otimes v -(-1)^{|I|} \sum^{4}_{i=1}(\partial_{i}\xi_I \star \partial_{i}\eta_L) \otimes v \\
	&\quad-\sum_{r<s}  (\partial_{rs}\xi_I \star \eta_L) \otimes \xi_{s,r}.v     +\bigchi_{|I|=3} \, \varepsilon_{I} \,(\xi_{I^{c}} \star \eta_L ) \otimes C v \\
	& \quad +\lambda \bigg[ (\xi_I \star \eta_L ) \otimes { t }.v-(-1)^{|I|}\sum^{4}_{i=1} \partial_{i}(\xi_{Ii} \star \eta_L ) \otimes v+ (-1)^{|I|} \sum _{i \neq j} (\partial_{i}\xi_{Ij} \star \eta_L) \otimes \xi_{j,i}.v  \\
	& \quad  +\bigchi_{|I|=2}    \, \varepsilon_{I} \, (\xi_{I^{c}} \star \eta_L )\otimes C v \bigg]\\
	&\quad + \lambda^{2} \bigg[- \sum _{i < j} (  \xi_{Iij} \star \eta_L ) \otimes   \xi_{j,i}. v -\bigchi_{|I|=1}    \, \varepsilon_{I} \,( \xi_{I^{c}} \star \eta_L)   \otimes Cv \bigg] +\lambda^{3} \Big[-\bigchi_{|I|=0}    ( \xi_{*} \star \eta_L ) \otimes Cv \Big]\bigg\}.
	\end{align*}
\end{prop}
\begin{proof}
	The proof follows by Proposition \ref{action} with a straightforward application of Lemma \ref{lem4}.
\end{proof}
In the following lemma we give a recursive formula in order to compute ${\xi_I} _{\lambda}(\Theta^{k}g \otimes v)$ for $I\in \mathcal I_{\neq}$ and $g\in U(g_{<0})$. %starting from $f _{\lambda}(\Theta^{k-1}g \otimes v)$. 
%This recursive formula holds both if we use the standard basis for $U(g_{<0})$ and if we use the Hodge dual basis for $U(g_{<0})$. In the former case we obtain the $\lambda-$action in the standard basis, in the latter case we obtain the $\lambda-$action in the Hodge dual basis.
\begin{lem}
	\label{actiontheta}
	Let $I\in \mathcal I_{\neq}$,  $g\in U(\g_{<0})$ and $k \in \Z_{> 0}$. We have:
	\begin{align*}
	{\xi_I} _{\lambda}(\Theta^{k}g \otimes v)=(\Theta+\lambda)({\xi_I} _{\lambda} \Theta^{k-1} g\otimes v)-\bigchi_{|I|=4} \varepsilon_{I} \Theta^{k-1}g \otimes C v. 
	\end{align*}
\end{lem}
\begin{proof}
	%We proceed by induction on $k$. For $k=1$:
	%\begin{align*}
	%&f _{\lambda}(\Theta g \otimes v)=\sum_{l \geq 0} \frac{\lambda^{j}}{j!} (t^{j}f).(\Theta g \otimes v)=\\
	%& =\sum_{l \geq 0} \frac{\lambda^{j}}{j!} \Theta (t^{j}f) g \otimes v+  \sum_{l \geq 0} \frac{\lambda^{j}}{j!} (jt^{j-1}f) g \otimes v-\bigchi_{|f|=4} \varepsilon_{I}  g \otimes C v\\
	%&=(\Theta+\lambda)(f _{\lambda} g\otimes v)-\bigchi_{|I|=4} \varepsilon_{I} (\Theta^{k-1})g \otimes C v. 
	%\end{align*}
	We have by \eqref{bracketliebis} and Proposition \ref{anniliK4}:
	\begin{align*}
	&{\xi_I} _{\lambda}(\Theta^{k}  g \otimes v)=\sum_{j \geq 0} \frac{\lambda^{j}}{j!} (t^{j}\xi_I)(\Theta ^{k} g \otimes v)\\
	& =\sum_{j \geq 0} \frac{\lambda^{j}}{j!} \Theta (t^{j}\xi_I) (\Theta ^{k-1}  g \otimes v)+  \sum_{j \geq 0} \frac{\lambda^{j}}{j!} (jt^{j-1}\xi_I) ( \Theta ^{k-1}g \otimes v)-\bigchi_{|I|=4}  \varepsilon_{I} \Theta ^{k-1} g \otimes C v\\
	&=(\Theta+\lambda)(f _{\lambda}  \Theta ^{k-1} g\otimes v)-\bigchi_{|I|=4}  \varepsilon_{I} \Theta ^{k-1} g \otimes C v. 
	\end{align*}
	%For $|f| \neq 4$ the formula reduces to:
	%\begin{align*}
	%f _{\lambda}(\Theta^{k}g \otimes v)=(\Theta+\lambda)^{k}f _{\lambda}(  g\otimes v).
	%\end{align*}
\end{proof}
\section{Singular vectors}
In this section we deduce some necessary conditions that singular vectors must satisfy. These conditions are obtained generalizing some ideas developed in \cite{kac1}.

We first give a more explicit description of $\g_0$: we have $\g_0=\left< \left\{ C,t, \xi_{ij} :\, 1 \leq i<j \leq 4 \right\} \right>\cong \mathfrak{so}(4)\oplus \C { t }\oplus \C C$, where $\mathfrak{so}(4)$ is the Lie algebra of $4\times 4$ skew$-$symmetric matrices. In the above homomorphism the element $\xi_{ij}$ corresponds to the skew-symmetric matrix $-E_{i,j}+E_{j,i} \in \mathfrak{so}(4)$. 
We consider the following basis of a Cartan subalgebra $\mathfrak{h}$:
%\begin{align*}
%H_{1}=iF_{1,2}, \,\, H_{2}=iF_{3,4}.
%\end{align*}
%h_x=H_1-H_2   h_y=H_1+H_2
\begin{equation}\label{hx} h_{x}:=-i\xi_{12}+i\xi_{34}, \,\,h_{y}:=-i\xi_{12}-i\xi_{34}.
\end{equation}
%Let $\varepsilon_{j}  \in \mathfrak{h}^{*}$ be such that $\varepsilon_{j}(H_{k})=\delta_{j,k}$.
Let $\alpha_x,\alpha_y\in \mathfrak{h}^*$ be such that $\alpha_x(h_x)=\alpha_y(h_y)=2$ and $\alpha_x(h_y)=\alpha_y(h_x)=0$. 
%The set of roots is $\Delta=\left\{\varepsilon_{1}-\varepsilon_{2}, \varepsilon_{1}+\varepsilon_{2},-(\varepsilon_{1}-\varepsilon_{2}),-(\varepsilon_{1}+\varepsilon_{2})\right\}$, the set of positive roots is $\Delta^{+}=\left\{\varepsilon_{1}-\varepsilon_{2}, \varepsilon_{1}+\varepsilon_{2}\right\}$.
The set of roots is $\Delta=\{\alpha_x,-\alpha_x,\alpha_y,-\alpha_y\}$ and we have the following root decomposition:
\begin{align*}
\mathfrak{so}(4)=\mathfrak{h} \oplus \left(\oplus_{\alpha \in \Delta} \g_{\alpha}\right) \,\, \text{with} \,\, \g_{\alpha_x}=\C e_{x},\,\g_{-\alpha_x}=\C f_{x},\,\g_{\alpha_y}=\C e_{y},\,\g_{-\alpha_y}=\C f_{y}
\end{align*}
where
%\begin{align*}
%&E_{\varepsilon_{1}-\varepsilon_{2}}=F_{1,3}+F_{2,4}+iF_{1,4}-iF_{2,3},\\
%&E_{\varepsilon_{1}+\varepsilon_{2}}=F_{1,3}-F_{2,4}-iF_{1,4}-iF_{2,3},\\
%E_{-(\varepsilon_{1}-\varepsilon_{2})}=F_{1,3}+F_{2,4}-iF_{1,4}+iF_{2,3},\\
%&E_{-(\varepsilon_{1}+\varepsilon_{2})}=F_{1,3}-F_{2,4}+iF_{1,4}+iF_{2,3}.
%\end{align*}
\begin{align*}
&e_{x}=\frac{1}{2}(-\xi_{1,3}-\xi_{2,4}-i\xi_{1,4}+i\xi_{2,3}),\\
&e_{y}=\frac{1}{2}(-\xi_{1,3}+\xi_{2,4}+i\xi_{1,4}+i\xi_{2,3}),\\
&f_{x}=\frac{1}{2}(\xi_{1,3}+\xi_{2,4}-i\xi_{1,4}+i\xi_{2,3}),\\
&f_{y}=\frac{1}{2}(\xi_{1,3}-\xi_{2,4}+i\xi_{1,4}+i\xi_{2,3}).
\end{align*}

It will be convenient to use the following notation:
\begin{align}
\label{notazionealfa}
e_1=e_x+e_y=-\xi_{13}+i\xi_{23},\\
\label{notazionebeta}
e_2=e_x-e_y=-\xi_{24}-i\xi_{14}.
\end{align}
The set $\left\{e_1,e_2\right\}$ is a basis of the nilpotent subalgebra $\g_{\alpha_x} \oplus \g_{\alpha_y}$. \\
We will write the weights $\mu=(m,n,\mu_{t},\mu_{C})$ of weight vectors of $\g_{0}-$modules with respect to action of the vectors $h_{x},h_{y},{ t } $ and $ C$.
\begin{rem}
	%The linear application $\phi$:
	%\begin{align*}
	%\phi: \, F &\longrightarrow F\\
	%v &\longmapsto C.v
	%\end{align*}
	%is a morphism of Lie superalgebras modules. Indeed for every $x \in \g_{\geq 0}$:
	%\begin{align*}
	%\phi (x.v)=C.x.v=(-1)^{p(x)p(C)}x.C.v +[C,x].v=(-1)^{p(x)p(C)}x.\phi(v).
	%\end{align*}
	Since $C$ is central, by Schur's lemma, $C$ acts as a scalar on $F$.%, and so we will denote also this scalar by $C$.
\end{rem}
\begin{rem}
	\label{notazioinecasellig0}
	The sets $\{e_x,f_x,h_x\}$ and $\{e_y,f_y,h_y\}$ span two copies of $\mathfrak{sl}_2$ and we think of $\g^{ss}_0$ in the standard way as a Lie algebra of derivations
	We have that:
	\begin{align*}
	\g^{ss}_{0}= \langle e_{x},f_{x},h_{x}\rangle \oplus \langle e_{y},f_{y},h_{y} \rangle \cong \langle  x_{1} \partial_{x_{2}}, x_{2} \partial_{x_{1}},x_{1} \partial_{x_{1}}-x_{2} \partial_{x_{2}} \rangle \oplus \langle  y_{1} \partial_{y_{2}}, y_{2} \partial_{y_{1}},y_{1} \partial_{y_{1}}-y_{2} \partial_{y_{2}} \rangle .
	\end{align*}
\end{rem}
%We have that:
%\begin{align*}
%\g^{ss}_{0}= \langle e_{1},f_{1},h_{1}\rangle \oplus \langle e_{2},f_{2},h_{2} \rangle \cong \langle  x \partial_{y}, y \partial_{x},x \partial_{x}-y \partial_{y} \rangle \oplus \langle  a \partial_{b}, b\partial_{a},a \partial_{a}-b \partial_{b} \rangle .
%\end{align*}
%\end{rem}
Thanks to Remark \ref{notazioinecasellig0} we will identify the irreducible $\g_{0}^{ss}-$module of highest weight $(m,n)$ with respect to $h_{x},h_{y}$ with the space of bihomogeneous polynomials in the four variables $x_1,x_2,y_1,y_2$ of degree $m$ in the variables $x_{1},x_{2}$, and of degree $n$ in the variables $y_{1},y_{2}$.\\
By direct computations, we obtain the following results.
\begin{lem}\label{g1}
The subalgebra $\g_{>0}$ is generated by $\g_{1}$, i.e. $\g_{i}=\g_{1}^{i}$ for all $i\geq 2$ and as $\g_{0}-$modules:
	\begin{align*}
	\g_{1} \cong \langle t\xi_{i} :\, 1\leq i\leq 4\rangle \oplus \langle \xi_{I}:\, I\in \mathcal I_<, \, |I|=3\rangle.
	\end{align*}
	The $\g_{0}-$modules $\langle t\xi_{i}:\, 1\leq i\leq 4\rangle$ and $\langle \xi_{I}:\, I\in \mathcal I_<, \, |I|=3\rangle$ are irreducible and the corresponding lowest weight vectors are $t (\xi_{1}+i \xi_{2})$ and $(\xi_{1}+i\xi_{2})\xi_{3}\xi_{4}$.
	\end{lem}
	\begin{lem}
	As $\g_{0}^{ss}-$modules:
	\begin{align*}
	\g_{-1} \cong \langle x_{1}y_1,x_1y_2,x_{2}y_1,x_2y_2\rangle.
	\end{align*}
	The isomorphism is given by:
	\begin{align*}
	\xi_{2}+i\xi_{1} \leftrightarrow x_{1}y_{1}, \,\,\, \xi_{2}-i\xi_{1} \leftrightarrow x_{2}y_{2}, \,\,\,  -\xi_{4}+i\xi_{3} \leftrightarrow x_{1}y_{2},  \,\,\, \xi_{4}+i\xi_{3} \leftrightarrow x_{2}y_{1}.
	\end{align*}
\end{lem}
%From direct computations, we obtain the following result.
%\begin{lem}
%As $\g_{0}^{ss}-$modules:
%\begin{align*}
%\g_{-1} \cong \langle x,y\rangle \otimes \langle a,b \rangle.
%\end{align*}
%The isomorphism is given by:
%\begin{align*}
%\xi_{2}+i\xi_{1} \leftrightarrow x a, \,\,\, \xi_{2}-i\xi_{1} \leftrightarrow yb, \,\,\,  -\xi_{4}+i\xi_{3} \leftrightarrow xb,  \,\,\, \xi_{4}+i\xi_{3} \leftrightarrow y a.
%\end{align*}
%\end{lem}
Motivated by the previous lemma, we will use the notation 
\begin{align}
\label{notazionewcaselli}
w_{11}=\eta_{2}+i\eta_{1}, \,\, w_{22}=\eta_{2}-i\eta_{1}, \,\, w_{12}= -\eta_{4}+i\eta_{3}, \,\, w_{21}=\eta_{4}+i\eta_{3}.
\end{align}
We point out that $[w_{11},w_{22}]=4 \Theta$, $[w_{12},w_{21}]=-4 \Theta$ and all other brackets between the $w'$s are 0.\\

By Lemma \ref{g1} to check whether a vector $\vec{m}$ in a $\g$-module is a highest weight singular vector it is enough to show that it is annihilated by $e_1$, $e_2$, $t(\xi_1+i\xi_2)$ and $(\xi_1+i\xi_2)\xi_3\xi_4$. 
Nevertheless in the determination of all possible highest weight singular vectors it is convenient to consider the action of all elements in $\g_{>0}$ and for this it is extremely convenient to use the $\lambda$-action.

\begin{rem}
\label{lemmatecnicos1s2s3}
	From the definition of the $\lambda-$action we deduce that $\vec{m} \in \Ind (F)$ is a highest weight singular vector if and only if the following hold:
	\begin{description}
		%\item [S1] $\frac{d^{3}}{d \lambda^{3}}(f _{\lambda} \vec{m})=0 \quad \forall f \in \Lambda(4)$; 
		\item[S0] $e_1. \vec{m}=e_2.\vec{m}=0$;
		\item [S1] $\frac{d^{2}}{d \lambda^{2}}({\xi_I} _{\lambda} \vec{m})=0$ for all $I \in \mathcal I_{\neq}$;
		\item [S2] $\frac{d}{d \lambda}({\xi_I}\, _{\lambda} \vec{m}) _{| \lambda=0}=0$ for all $I \in \mathcal I_{\neq}$ such that $ |I| \geq 1$; 
		\item [S3] $({\xi_I} _{\lambda} \vec{m})_{| \lambda=0}=0$ for all $I \in \mathcal I_{\neq}$ such that $ |I| \geq 3 $.
	\end{description}
Indeed condition \textbf{S0} implies that $\vec{m}$ is a highest weight vector. Condition \textbf{S1} is equivalent to
	\begin{align*}
	\sum_{j\geq 2}j(j-1)\frac{\lambda^{j-2}}{j!}(t^{j}\xi_I)\vec{m}=0 ,
	\end{align*}
	which implies $(t^{j}\xi_I)\vec{m}=0$ for all $I\in \mathcal I_{\neq}$ and $j\geq 2$.\\
	Condition \textbf{S2} is equivalent to $(t\xi_I)\vec{m}=0 $ for all for all $I\in \mathcal I_{\neq}$ such that $ |I| \geq 1$.\\
	Condition \textbf{S3} is equivalent to $\xi_I \vec{m}=0 $ for all $I\in \mathcal I_{\neq}$ such that $ |I| \geq 3 $. % or $ f \in B_{\so(4)}$.
	
\end{rem}
The aim of this section is to solve equations \textbf{S0}--\textbf{S3} in order to obtain the following classification of singular vectors. We recall that the highest weight of $F$ is always written with respect to the elements $h_{x}, \, h_{y}, \, { t }$ and $C$. Let us call $M(m,n,\mu_t,\mu_C)$ the Verma module $\Ind (F (m,n,\mu_{t},\mu_{C}))$, where 
$F (m,n,\mu_t,\mu_C)$ is the irreducible $\g_{0}-$module with highest weight $(m,n,\mu_t,\mu_C)$.

%Motivated by the previous lemma, we will use the notation 
%\begin{align}
%\label{notazionewcaselli}
%w_{xa}=\eta_{2}+i\eta_{1}, \,\, w_{yb}=\eta_{2}-i\eta_{1}, \,\, w_{xb}= -\eta_{4}+i\eta_{3}, \,\, w_{ya}=\eta_{4}+i\eta_{3}.
%\end{align}
%We point out that $[w_{xa},w_{yb}]=4 \Theta$, $[w_{xb},w_{ya}]=-4 \Theta$ and all other brackets between the $w'$s are 0.\\
%We will think of the irreducible $\g_{0}^{ss}-$module of highest weight $(m,n)$ with respect to $h_{1},h_{2}$ as the space of homogeneous polynomials of degree $m$ in the variables $x,y$ and of degree $n$ in the variables $a,b$.\\
%The aim of this section is to solve equations \textbf{S1}, \textbf{S2}, \textbf{S3} in order to obtain the following classification of singular vectors. We recall that the highest weight of $F$ is always written with respect to the action of $h_{1}, \, h_{2}, \, { t }$ and $C$.
\begin{thm}
	\label{sing1}
	Let $F$ be an irreducible finite$-$dimensional $\g_{0}-$module, with highest weight $\mu$. A vector in $  \Ind (F) $ is a non trivial highest weight singular vector of degree 1 if and only if $\vec{m}$ is (up to a scalar) one of the following vectors:
	
	\begin{description}
		\item [a] $\mu=(m,n,-\frac{m+n}{2},\frac{m-n}{2}) $ with $m,n \in \Z_{\geq 0}$,
		$$\vec{m}_{1a}=w_{11}\otimes x^{m}_{1}y^{n}_{1};$$
		\item [b] $\mu=(m,n,1+\frac{m-n}{2},-1-\frac{m+n}{2})$, with $m  \in \Z_{>0}$, $n \in \Z_{\geq 0}$,
		$$\vec{m}_{1b}=w_{21}\otimes x^{m}_{1}y^{n}_{1}-w_{11}\otimes x^{m-1}_{1}x_{2}y^{n}_{1};$$
		%%\begin{align*}
		%%\bar{w_{1}}=\frac{1}{2\mu_{1}}E_{-(\varepsilon_{1}-\varepsilon_{2})} v_{\mu};
		%%\end{align*}
		\item [c] $\mu=(m,n,2+\frac{m+n}{2},\frac{n-m}{2})$, with $m,n \in \Z_{>0}$,
		$$\vec{m}_{1c}=w_{22}\otimes x^{m}_{1}y^{n}_{1}-w_{12}\otimes x^{m-1}_{1}x_{2}y^{n}_{1}-w_{21}\otimes x^{m}_{1}y^{n-1}_{1}y_{2}+w_{11}\otimes x^{m-1}_{1}x_{2}y^{n-1}_{1}y_{2};$$
		%\begin{align*}
		%w_{2}&=-\frac{1}{2\mu_{1}}E_{-(\varepsilon_{1}-\varepsilon_{2})}w_{1},\\
		%\bar{w_{2}}&=\frac{1}{2\mu_{2}}E_{-(\varepsilon_{1}+\varepsilon_{2})}w_{1},\\
		%\bar{w_{1}}&=\frac{1}{(2\mu_{2})(2\mu_{1})}E_{-(\varepsilon_{1}-\varepsilon_{2})}E_{-(\varepsilon_{1}+\varepsilon_{2})}w_{1} ;
		%\end{align*}
		\item [d] $\mu=(m,n,1+\frac{n-m}{2},1+\frac{m+n}{2})$, with $m \in \Z_{\geq 0}$, $n \in \Z_{>0}$,
		$$\vec{m}_{1d}=w_{12}\otimes x^{m}_{1}y^{n}_{1}-w_{11}\otimes x^{m}_{1}y^{n-1}_{1}y_{2}.$$
		%\begin{align*}
		%\bar{w_{1}}=-\frac{1}{2\mu_{2}}E_{-(\varepsilon_{1}+\varepsilon_{2})}v_{\mu}.
		%\end{align*}
	\end{description}
\end{thm}

\begin{thm}
	\label{sing2}
	Let $F$ be an irreducible finite$-$dimensional $\g_{0}-$module, with highest weight $\mu$. A vector $\vec{m} \in \Ind (F) $ is a non trivial highest weight singular vector of degree 2 if and only if $\vec{m}$ is (up to a scalar) one of the following vectors:
	
	\begin{description}
		\item [a] $\mu=(0,n,1-\frac{n}{2},-1-\frac{n}{2}) $ with $n \in \Z_{\geq 0}$,
		$$\vec{m}_{2a}=w_{11}w_{21} \otimes y^{n}_{1};$$
		\item [b] $\mu=(m,0,1-\frac{m}{2},1+\frac{m}{2}) $ with $m \in \Z_{\geq 0}$,
		$$\vec{m}_{2b}=w_{11}w_{12}\otimes x^{m}_{1};$$ 
		\item [c] $\mu=(m,0,2+\frac{m}{2},-\frac{m}{2}) $ with $m \in \Z_{> 1}$,
		$$\vec{m}_{2c}=w_{22}w_{21}\otimes x^{m}_{1}+(w_{11}w_{22}+w_{21}w_{12})\otimes x^{m-1}_{1}x_{2}-w_{11}w_{12}\otimes x^{m-2}_{1}x^{2}_{2};$$
		%\begin{align*}
		%&w_{5}=-\frac{1}{2 \mu}E_{-(\varepsilon_{1}-\varepsilon_{2})}v_{\mu},\\
		%&w_{6}=\frac{1}{2 \mu}E_{-(\varepsilon_{1}-\varepsilon_{2})}v_{\mu},\\
		%&w_{2}=\frac{1}{4\mu (\mu-1)}E_{-(\varepsilon_{1}-\varepsilon_{2})}E_{-(\varepsilon_{1}-\varepsilon_{2})}v_{\mu};
		%\end{align*}
		\item [d] $\mu=(0,n,2+\frac{n}{2},\frac{n}{2}) $ with $n \in \Z_{> 1}$,
		$$\vec{m}_{2d}=w_{22}w_{12}\otimes y^{n}_{1}-(w_{22}w_{11}+w_{21}w_{12})\otimes  y^{n-1}_{1}y_{2} -w_{11}w_{21} \otimes y^{n-2}_{1}y^{2}_{2}.$$
		%\begin{align*}
		%&w_{5}=-\frac{1}{2 \mu}E_{-(\varepsilon_{1}+\varepsilon_{2})}v_{\mu},\\
		%&w_{6}=-\frac{1}{2 \mu}E_{-(\varepsilon_{1}+\varepsilon_{2})}v_{\mu},\\
		%&w_{1}=\frac{1}{4\mu (\mu-1)}E_{-(\varepsilon_{1}+\varepsilon_{2})}E_{-(\varepsilon_{1}+\varepsilon_{2})}v_{\mu},\\
		%&w_{7}=\frac{2}{ \mu}E_{-(\varepsilon_{1}+\varepsilon_{2})}v_{\mu}.
		%\end{align*}
	\end{description}
\end{thm}

\begin{thm}
	\label{sing3}
	Let $F$ be an irreducible finite$-$dimensional $\g_{0}-$module, with highest weight $\mu$. A vector $\vec{m} \in \Ind (F) $ is a non trivial highest weight singular vector of degree 3 if and only if $\vec{m}$ is (up to a scalar) one of the following vectors:
	
	\begin{description}
		\item [a] $\mu=(1,0,\frac{5}{2},-\frac{1}{2})$,
		$$\vec{m}_{3a}=w_{11}w_{22}w_{21}\otimes x_{1}+w_{21}w_{12}w_{11}\otimes x_{2};$$ 
		%\begin{align*}
		%w_{3}=\frac{1}{2}E_{-(\varepsilon_{1}-\varepsilon_{2})} v_{\mu};
		%\end{align*}
		\item [b] $\mu=(0,1,\frac{5}{2},\frac{1}{2})$,
		$$\vec{m}_{3b}=w_{11}w_{22}w_{12}\otimes y_{1}+w_{12}w_{21}w_{11}\otimes y_{2}.$$ 
	\end{description}
\end{thm}

\begin{thm}
	\label{greaterthan3}
	There are no singular vectors of degree greater than 3. 
\end{thm}
\begin{rem}
	\label{costruzionemorfismi}
	 We call a Verma module \textit{degenerate} if it is not irreducible.
	We point out that, given $M(m,n,\mu_{t},\mu_{C})$ and $M(\widetilde{m},\widetilde{n},\widetilde{\mu}_{t},\widetilde{\mu}_{C})$ Verma modules, we can construct a non trivial morphism of $\g-$modules from the former to the latter if and only if there exists a highest weight singular vector $\vec{m}$ in $M(\widetilde{m},\widetilde{n},\widetilde{\mu}_{t},\widetilde{\mu}_{C})$ of highest weight $(m,n,\mu_{t},\mu_{C})$. The map is uniquely determined by:
	\begin{align*}
	\nabla: \, M(m,n,\mu_{t},\mu_{C})&\longrightarrow M(\widetilde{m},\widetilde{n},\widetilde{\mu}_{t},\widetilde{\mu}_{C})\\
	v_{\mu} &\longmapsto \vec{m},
	\end{align*}
	where $v_{\mu}$ is a highest weight vector of $F(m,n,\mu_{t},\mu_{C})$.
	If $\vec{m}$ is a singular vector of degree $d$, we say that $\nabla$ is a morphism of degree $d$.
\end{rem}	
We use Remark \ref{costruzionemorfismi} to construct the maps in Figure \ref{figura} of all possible morphisms in the case of $K'_{4}$. We also observe that the symmetry of this picture is coherent with conformal duality. Indeed, by the main result in \cite{cantacasellikac} the conformal dual of a Verma module $M(m,n,\mu_{t},\mu_{C})$ is $M(m,n,-\mu_{t}+a,-\mu_{C}+b)$, with
\[
a=\text{str}(\text{ad}(t)_{|\g_{<0}})=2
\] 
and 
\[
b=\text{str}(\text{ad}(C)_{|\g_{<0}})=0,
\]
where $\g=\mathcal A(K'_4)$, "str" denotes supertrace, and "ad" denotes the adjoint representation. In particular the duality is obtained with the rotation by 180 degrees of the whole picture. Note also that all compositions of two morphisms in Figure \ref{figura} must vanish by the classification of singular vectors, and hence we obtain an infinite number of bilateral complexes of morphisms. 

	\begin{figure}[b!]
	\centering
	\caption{}
	\label{figura}
	\begin{tikzpicture}
	
	%\draw[step=1cm,gray!25!,very thin] (-5,-5) grid (5,5);
	\draw[->,black] (0,0) -- (6,0) node[anchor=north west]{};
	\draw[->,black] (0,0) -- (0,6) node[anchor=south east]{};
	\node[black] at (3,6.5) {$(m,n,-\frac{m+n}{2},\frac{m-n}{2})$};
	\node[black] at (6,6.5) {\textbf{A}};
	\node[black] at (6.5,0) {$m$};
	\node[black] at (0,6.5) {$n$};
	\draw[black] (0,0) -- (5.5,5.5);
	\draw[black] (0,1) -- (4.5,5.5);
	\draw[black] (0,2) -- (3.5,5.5);
	\draw[black] (0,3) -- (2.5,5.5);
	\draw[black] (0,4) -- (1.5,5.5);
	\draw[black] (0,5) -- (0.5,5.5);
	\draw[black] (1,0) -- (5.5,4.5);
	\draw[black] (2,0) -- (5.5,3.5);
	\draw[black] (3,0) -- (5.5,2.5);
	\draw[black] (4,0) -- (5.5,1.5);
	\draw[black] (5,0) -- (5.5,0.5);

	\draw[-latex] (1,0) to[out=200,in=70] (-0.9,-1.9);
	\draw[-latex] (0,1) to[out=-110,in=20] (-1.9,-0.9);
	%\draw[-latex,black] (1,0) -- (-1,-2);
	%\draw[-latex,black] (0,1) -- (-2,-1);

	\draw[-latex,black] (0,1.9) -- (-0.9,0.1);
	\draw[-latex,black] (0,2.9) -- (-0.9,1.1);
	\draw[-latex,black] (0,3.9) -- (-0.9,2.1);
	\draw[-latex,black] (0,4.9) -- (-0.9,3.1);
	\draw[-latex,black] (0,5.9) -- (-0.9,4.1);
	\draw[-latex,black] (-0.5,5.9) -- (-0.9,5.1);
	
	\draw[-latex,black] (2,0) -- (0.1,-0.9);
	\draw[-latex,black] (3,0) -- (1.1,-0.9);
	\draw[-latex,black] (4,0) -- (2.1,-0.9);
	\draw[-latex,black] (5,0) -- (3.1,-0.9);
	\draw[-latex,black] (6,0) -- (4.1,-0.9);
	\draw[-latex,black] (6,-0.5) -- (5.1,-0.9);

	\draw[-latex,black] (0,-1) -- (-0.9,-2.9);
	\draw[-latex,black] (0,-2) -- (-0.9,-3.9);
	\draw[-latex,black] (0,-3) -- (-0.9,-4.9);
	\draw[-latex,black] (0,-4) -- (-0.9,-5.9);
	\draw[-latex,black] (0,-5) -- (-0.9,-6.9);
	\draw[-latex,black] (0,-6) -- (-0.5,-7.1);
	
	\draw[-latex,black] (-1,0) -- (-2.9,-0.9);
	\draw[-latex,black] (-2,0) -- (-3.9,-0.9);
	\draw[-latex,black] (-3,0) -- (-4.9,-0.9);
	\draw[-latex,black] (-4,0) -- (-5.9,-0.9);
	\draw[-latex,black] (-5,0) -- (-6.9,-0.9);
	\draw[-latex,black] (-6,0) -- (-7.1,-0.5);

	\draw[-latex,black] (5.5,5.5) -- (5.1,5.1);
	\draw[-latex,black] (5,5) -- (4.1,4.1);
	\draw[-latex,black] (4,4) -- (3.1,3.1);
	\draw[-latex,black] (3,3) -- (2.1,2.1);
	\draw[-latex,black] (2,2) -- (1.1,1.1);
	\draw[-latex,black] (1,1) -- (0.1,0.1);
	
	\draw[-latex,black] (5.5,4.5) -- (5.1,4.1);
	\draw[-latex,black] (5,4) -- (4.1,3.1);
	\draw[-latex,black] (4,3) -- (3.1,2.1);
	\draw[-latex,black] (3,2) -- (2.1,1.1);
	\draw[-latex,black] (2,1) -- (1.1,0.1);

	\draw[-latex,black] (5.5,3.5) -- (5.1,3.1);
	\draw[-latex,black] (5,3) -- (4.1,2.1);
	\draw[-latex,black] (4,2) -- (3.1,1.1);
	\draw[-latex,black] (3,1) -- (2.1,0.1);

	\draw[-latex,black] (5.5,2.5) -- (5.1,2.1);
	\draw[-latex,black] (5,2) -- (4.1,1.1);
	\draw[-latex,black] (4,1) -- (3.1,0.1);

	\draw[-latex,black] (5.5,1.5) -- (5.1,1.1);
	\draw[-latex,black] (5,1) -- (4.1,0.1);

	\draw[-latex,black] (5.5,0.5) -- (5.1,0.1);

	\draw[-latex,black] (4.5,5.5) -- (4.1,5.1);
	\draw[-latex,black] (4,5) -- (3.1,4.1);
	\draw[-latex,black] (3,4) -- (2.1,3.1);
	\draw[-latex,black] (2,3) -- (1.1,2.1);
	\draw[-latex,black] (1,2) -- (0.1,1.1);

	\draw[-latex,black] (3.5,5.5) -- (3.1,5.1);
	\draw[-latex,black] (3,5) -- (2.1,4.1);
	\draw[-latex,black] (2,4) -- (1.1,3.1);
	\draw[-latex,black] (1,3) -- (0.1,2.1);

	\draw[-latex,black] (2.5,5.5) -- (2.1,5.1);
	\draw[-latex,black] (2,5) -- (1.1,4.1);
	\draw[-latex,black] (1,4) -- (0.1,3.1);

	\draw[-latex,black] (1.5,5.5) -- (1.1,5.1);
	\draw[-latex,black] (1,5) -- (0.1,4.1);

	\draw[-latex,black] (0.5,5.5) -- (0.1,5.1);

	\coordinate (center) at (0,0);
	\fill[white] (center) + (0, 0.1) arc (90:270:0.1);
	\fill[white] (center)+ (0, -0.1) arc (270:450:0.1);
	\draw[black] (center)+ (0, -0.1) arc (270:450:0.1);
	\draw[black] (center)+ (0, 0.1) arc (90:270:0.1);
	\coordinate (center) at (0,1);
	\fill[white] (center) + (0, 0.1) arc (90:270:0.1);
	\fill[white] (center)+ (0, -0.1) arc (270:450:0.1);
	\draw[black] (center)+ (0, -0.1) arc (270:450:0.1);
	\draw[black] (center)+ (0, 0.1) arc (90:270:0.1);
	\coordinate (center) at (0,2);
	\fill[white] (center) + (0, 0.1) arc (90:270:0.1);
	\fill[white] (center)+ (0, -0.1) arc (270:450:0.1);
	\draw[black] (center)+ (0, -0.1) arc (270:450:0.1);
	\draw[black] (center)+ (0, 0.1) arc (90:270:0.1);
	\coordinate (center) at (0,3);
	\fill[white] (center) + (0, 0.1) arc (90:270:0.1);
	\fill[white] (center)+ (0, -0.1) arc (270:450:0.1);
	\draw[black] (center)+ (0, -0.1) arc (270:450:0.1);
	\draw[black] (center)+ (0, 0.1) arc (90:270:0.1);
	\coordinate (center) at (0,4);
	\fill[white] (center) + (0, 0.1) arc (90:270:0.1);
	\fill[white] (center)+ (0, -0.1) arc (270:450:0.1);
	\draw[black] (center)+ (0, -0.1) arc (270:450:0.1);
	\draw[black] (center)+ (0, 0.1) arc (90:270:0.1);
	\coordinate (center) at (0,5);
	\fill[white] (center) + (0, 0.1) arc (90:270:0.1);
	\fill[white] (center)+ (0, -0.1) arc (270:450:0.1);
	\draw[black] (center)+ (0, -0.1) arc (270:450:0.1);
	\draw[black] (center)+ (0, 0.1) arc (90:270:0.1);
	\coordinate (center) at (1,0);
	\fill[white] (center) + (0, 0.1) arc (90:270:0.1);
	\fill[white] (center)+ (0, -0.1) arc (270:450:0.1);
	\draw[black] (center)+ (0, -0.1) arc (270:450:0.1);
	\draw[black] (center)+ (0, 0.1) arc (90:270:0.1);
	\coordinate (center) at (2,0);
	\fill[white] (center) + (0, 0.1) arc (90:270:0.1);
	\fill[white] (center)+ (0, -0.1) arc (270:450:0.1);
	\draw[black] (center)+ (0, -0.1) arc (270:450:0.1);
	\draw[black] (center)+ (0, 0.1) arc (90:270:0.1);
	\coordinate (center) at (3,0);
	\fill[white] (center) + (0, 0.1) arc (90:270:0.1);
	\fill[white] (center)+ (0, -0.1) arc (270:450:0.1);
	\draw[black] (center)+ (0, -0.1) arc (270:450:0.1);
	\draw[black] (center)+ (0, 0.1) arc (90:270:0.1);
	\coordinate (center) at (4,0);
	\fill[white] (center) + (0, 0.1) arc (90:270:0.1);
	\fill[white] (center)+ (0, -0.1) arc (270:450:0.1);
	\draw[black] (center)+ (0, -0.1) arc (270:450:0.1);
	\draw[black] (center)+ (0, 0.1) arc (90:270:0.1);
	\coordinate (center) at (5,0);
	\fill[white] (center) + (0, 0.1) arc (90:270:0.1);
	\fill[white] (center)+ (0, -0.1) arc (270:450:0.1);
	\draw[black] (center)+ (0, -0.1) arc (270:450:0.1);
	\draw[black] (center)+ (0, 0.1) arc (90:270:0.1);

	\coordinate (center) at (1,1);
	\fill[white] (center) + (0, 0.1) arc (90:270:0.1);
	\fill[white] (center)+ (0, -0.1) arc (270:450:0.1);
	\draw[black] (center)+ (0, -0.1) arc (270:450:0.1);
	\draw[black] (center)+ (0, 0.1) arc (90:270:0.1);
	\coordinate (center) at (1,2);
	\fill[white] (center) + (0, 0.1) arc (90:270:0.1);
	\fill[white] (center)+ (0, -0.1) arc (270:450:0.1);
	\draw[black] (center)+ (0, -0.1) arc (270:450:0.1);
	\draw[black] (center)+ (0, 0.1) arc (90:270:0.1);
	\coordinate (center) at (1,3);
	\fill[white] (center) + (0, 0.1) arc (90:270:0.1);
	\fill[white] (center)+ (0, -0.1) arc (270:450:0.1);
	\draw[black] (center)+ (0, -0.1) arc (270:450:0.1);
	\draw[black] (center)+ (0, 0.1) arc (90:270:0.1);
	\coordinate (center) at (1,4);
	\fill[white] (center) + (0, 0.1) arc (90:270:0.1);
	\fill[white] (center)+ (0, -0.1) arc (270:450:0.1);
	\draw[black] (center)+ (0, -0.1) arc (270:450:0.1);
	\draw[black] (center)+ (0, 0.1) arc (90:270:0.1);
	\coordinate (center) at (1,5);
	\fill[white] (center) + (0, 0.1) arc (90:270:0.1);
	\fill[white] (center)+ (0, -0.1) arc (270:450:0.1);
	\draw[black] (center)+ (0, -0.1) arc (270:450:0.1);
	\draw[black] (center)+ (0, 0.1) arc (90:270:0.1);
	\coordinate (center) at (2,1);
	\fill[white] (center) + (0, 0.1) arc (90:270:0.1);
	\fill[white] (center)+ (0, -0.1) arc (270:450:0.1);
	\draw[black] (center)+ (0, -0.1) arc (270:450:0.1);
	\draw[black] (center)+ (0, 0.1) arc (90:270:0.1);
	\coordinate (center) at (2,2);
	\fill[white] (center) + (0, 0.1) arc (90:270:0.1);
	\fill[white] (center)+ (0, -0.1) arc (270:450:0.1);
	\draw[black] (center)+ (0, -0.1) arc (270:450:0.1);
	\draw[black] (center)+ (0, 0.1) arc (90:270:0.1);
	\coordinate (center) at (2,3);
	\fill[white] (center) + (0, 0.1) arc (90:270:0.1);
	\fill[white] (center)+ (0, -0.1) arc (270:450:0.1);
	\draw[black] (center)+ (0, -0.1) arc (270:450:0.1);
	\draw[black] (center)+ (0, 0.1) arc (90:270:0.1);
	\coordinate (center) at (2,4);
	\fill[white] (center) + (0, 0.1) arc (90:270:0.1);
	\fill[white] (center)+ (0, -0.1) arc (270:450:0.1);
	\draw[black] (center)+ (0, -0.1) arc (270:450:0.1);
	\draw[black] (center)+ (0, 0.1) arc (90:270:0.1);
	\coordinate (center) at (2,5);
	\fill[white] (center) + (0, 0.1) arc (90:270:0.1);
	\fill[white] (center)+ (0, -0.1) arc (270:450:0.1);
	\draw[black] (center)+ (0, -0.1) arc (270:450:0.1);
	\draw[black] (center)+ (0, 0.1) arc (90:270:0.1);
	\coordinate (center) at (3,1);
	\fill[white] (center) + (0, 0.1) arc (90:270:0.1);
	\fill[white] (center)+ (0, -0.1) arc (270:450:0.1);
	\draw[black] (center)+ (0, -0.1) arc (270:450:0.1);
	\draw[black] (center)+ (0, 0.1) arc (90:270:0.1);
	\coordinate (center) at (3,2);
	\fill[white] (center) + (0, 0.1) arc (90:270:0.1);
	\fill[white] (center)+ (0, -0.1) arc (270:450:0.1);
	\draw[black] (center)+ (0, -0.1) arc (270:450:0.1);
	\draw[black] (center)+ (0, 0.1) arc (90:270:0.1);
	\coordinate (center) at (3,3);
	\fill[white] (center) + (0, 0.1) arc (90:270:0.1);
	\fill[white] (center)+ (0, -0.1) arc (270:450:0.1);
	\draw[black] (center)+ (0, -0.1) arc (270:450:0.1);
	\draw[black] (center)+ (0, 0.1) arc (90:270:0.1);
	\coordinate (center) at (3,4);
	\fill[white] (center) + (0, 0.1) arc (90:270:0.1);
	\fill[white] (center)+ (0, -0.1) arc (270:450:0.1);
	\draw[black] (center)+ (0, -0.1) arc (270:450:0.1);
	\draw[black] (center)+ (0, 0.1) arc (90:270:0.1);
	\coordinate (center) at (3,5);
	\fill[white] (center) + (0, 0.1) arc (90:270:0.1);
	\fill[white] (center)+ (0, -0.1) arc (270:450:0.1);
	\draw[black] (center)+ (0, -0.1) arc (270:450:0.1);
	\draw[black] (center)+ (0, 0.1) arc (90:270:0.1);
	\coordinate (center) at (4,1);
	\fill[white] (center) + (0, 0.1) arc (90:270:0.1);
	\fill[white] (center)+ (0, -0.1) arc (270:450:0.1);
	\draw[black] (center)+ (0, -0.1) arc (270:450:0.1);
	\draw[black] (center)+ (0, 0.1) arc (90:270:0.1);
	\coordinate (center) at (4,2);
	\fill[white] (center) + (0, 0.1) arc (90:270:0.1);
	\fill[white] (center)+ (0, -0.1) arc (270:450:0.1);
	\draw[black] (center)+ (0, -0.1) arc (270:450:0.1);
	\draw[black] (center)+ (0, 0.1) arc (90:270:0.1);
	\coordinate (center) at (4,3);
	\fill[white] (center) + (0, 0.1) arc (90:270:0.1);
	\fill[white] (center)+ (0, -0.1) arc (270:450:0.1);
	\draw[black] (center)+ (0, -0.1) arc (270:450:0.1);
	\draw[black] (center)+ (0, 0.1) arc (90:270:0.1);
	\coordinate (center) at (4,4);
	\fill[white] (center) + (0, 0.1) arc (90:270:0.1);
	\fill[white] (center)+ (0, -0.1) arc (270:450:0.1);
	\draw[black] (center)+ (0, -0.1) arc (270:450:0.1);
	\draw[black] (center)+ (0, 0.1) arc (90:270:0.1);
	\coordinate (center) at (4,5);
	\fill[white] (center) + (0, 0.1) arc (90:270:0.1);
	\fill[white] (center)+ (0, -0.1) arc (270:450:0.1);
	\draw[black] (center)+ (0, -0.1) arc (270:450:0.1);
	\draw[black] (center)+ (0, 0.1) arc (90:270:0.1);
	\coordinate (center) at (5,1);
	\fill[white] (center) + (0, 0.1) arc (90:270:0.1);
	\fill[white] (center)+ (0, -0.1) arc (270:450:0.1);
	\draw[black] (center)+ (0, -0.1) arc (270:450:0.1);
	\draw[black] (center)+ (0, 0.1) arc (90:270:0.1);
	\coordinate (center) at (5,2);
	\fill[white] (center) + (0, 0.1) arc (90:270:0.1);
	\fill[white] (center)+ (0, -0.1) arc (270:450:0.1);
	\draw[black] (center)+ (0, -0.1) arc (270:450:0.1);
	\draw[black] (center)+ (0, 0.1) arc (90:270:0.1);
	\coordinate (center) at (5,3);
	\fill[white] (center) + (0, 0.1) arc (90:270:0.1);
	\fill[white] (center)+ (0, -0.1) arc (270:450:0.1);
	\draw[black] (center)+ (0, -0.1) arc (270:450:0.1);
	\draw[black] (center)+ (0, 0.1) arc (90:270:0.1);
	\coordinate (center) at (5,4);
	\fill[white] (center) + (0, 0.1) arc (90:270:0.1);
	\fill[white] (center)+ (0, -0.1) arc (270:450:0.1);
	\draw[black] (center)+ (0, -0.1) arc (270:450:0.1);
	\draw[black] (center)+ (0, 0.1) arc (90:270:0.1);
	\coordinate (center) at (5,5);
	\fill[white] (center) + (0, 0.1) arc (90:270:0.1);
	\fill[white] (center)+ (0, -0.1) arc (270:450:0.1);
	\draw[black] (center)+ (0, -0.1) arc (270:450:0.1);
	\draw[black] (center)+ (0, 0.1) arc (90:270:0.1);

	\draw[->,black] (-1,-1) -- (-7,-1) node[anchor=north west]{};
	\draw[->,black] (-1,-1) -- (-1,-7) node[anchor=south east]{};
	\node[black] at (-1,-7.5) {$n$};
	\node[black] at (-7.5,-1) {$m$};
	\node[black] at (-4,-7) {($m,n,\frac{m+n}{2}+2,\frac{n-m}{2}$)};
	\node[black] at (-7,-7) {\textbf{C}};
	%\coordinate (center) at (-0.5,-7);
	%\fill[black] (center) + (0, 0.1) arc (90:450:0.1);
	%%\fill[black] (center) + (0, 0.1) arc (90:270:0.1);
	%%\fill[black] (center) + (0, -0.1) arc (270:450:0.1);
	%%\draw[black] (center)+ (0, -0.1) arc (270:450:0.1);
	%\node[black] at (1.5,-7.5) {($\mu_{1},\mu_{2},\mu_{2}+1,1+\mu_{1}$)};
	%\coordinate (center) at (-0.5,-7.5);
	%\draw[black] (center)+ (0, -0.1) arc (270:450:0.1);
	%\draw[black] (center)+ (0, 0.1) arc (90:270:0.1);
	%\draw[line width=0.45mm,black] (-2,-1) -- (-1,-2);
	%\draw[line width=0.45mm,black] (-3,-1) -- (-1,-3);
	%\draw[line width=0.45mm,black] (-4,-1) -- (-1,-4);
	%\draw[line width=0.45mm,black] (-5,-1) -- (-1,-5); 
	%\draw[line width=0.45mm,black] (-6,-1) -- (-1,-6);
	%\draw[line width=0.45mm,black] (-7,-1) -- (-1,-7);
	%\draw[line width=0.45mm,black] (-7,-2) -- (-2,-7);
	%\draw[line width=0.45mm,black] (-7,-3) -- (-3,-7);
	%\draw[line width=0.45mm,black] (-7,-4) -- (-4,-7); 
	%\draw[line width=0.45mm,black] (-7,-5) -- (-5,-7);
	%\draw[line width=0.45mm,black] (-7,-6) -- (-6,-7);
	\draw[black] (-1,-1) -- (-6.5,-6.5);
	\draw[black] (-2,-1) -- (-6.5,-5.5);
	\draw[black] (-3,-1) -- (-6.5,-4.5);
	\draw[black] (-4,-1) -- (-6.5,-3.5);
	\draw[black] (-5,-1) -- (-6.5,-2.5);
	\draw[black] (-6,-1) -- (-6.5,-1.5);
	\draw[black] (-1,-2) -- (-5.5,-6.5);
	\draw[black] (-1,-3) -- (-4.5,-6.5);
	\draw[black] (-1,-4) -- (-3.5,-6.5);
	\draw[black] (-1,-5) -- (-2.5,-6.5);
	\draw[black] (-1,-6) -- (-1.5,-6.5);

	\draw[-latex,black] (-1,-1) -- (-1.9,-1.9);
	\draw[-latex,black] (-2,-2) -- (-2.9,-2.9);
	\draw[-latex,black] (-3,-3) -- (-3.9,-3.9);
	\draw[-latex,black] (-4,-4) -- (-4.9,-4.9);
	\draw[-latex,black] (-5,-5) -- (-5.9,-5.9);
	\draw[-latex,black] (-6,-6) -- (-6.5,-6.5);
	
	\draw[-latex,black] (-1,-2) -- (-1.9,-2.9);
	\draw[-latex,black] (-2,-3) -- (-2.9,-3.9);
	\draw[-latex,black] (-3,-4) -- (-3.9,-4.9);
	\draw[-latex,black] (-4,-5) -- (-4.9,-5.9);
	\draw[-latex,black] (-5,-6) -- (-5.5,-6.5);
	
	\draw[-latex,black] (-1,-3) -- (-1.9,-3.9);
	\draw[-latex,black] (-2,-4) -- (-2.9,-4.9);
	\draw[-latex,black] (-3,-5) -- (-3.9,-5.9);
	\draw[-latex,black] (-4,-6) -- (-4.5,-6.5);
	
	\draw[-latex,black] (-1,-4) -- (-1.9,-4.9);
	\draw[-latex,black] (-2,-5) -- (-2.9,-5.9);
	\draw[-latex,black] (-3,-6) -- (-3.5,-6.5);
	
	\draw[-latex,black] (-1,-5) -- (-1.9,-5.9);
	\draw[-latex,black] (-2,-6) -- (-2.5,-6.5);
	
	\draw[-latex,black] (-1,-6) -- (-1.5,-6.5);

	\draw[-latex,black] (-2,-1) -- (-2.9,-1.9);
	\draw[-latex,black] (-3,-2) -- (-3.9,-2.9);
	\draw[-latex,black] (-4,-3) -- (-4.9,-3.9);
	\draw[-latex,black] (-5,-4) -- (-5.9,-4.9);
	\draw[-latex,black] (-6,-5) -- (-6.5,-5.5);
	
	\draw[-latex,black] (-3,-1) -- (-3.9,-1.9);
	\draw[-latex,black] (-4,-2) -- (-4.9,-2.9);
	\draw[-latex,black] (-5,-3) -- (-5.9,-3.9);
	\draw[-latex,black] (-6,-4) -- (-6.5,-4.5);
	
	\draw[-latex,black] (-4,-1) -- (-4.9,-1.9);
	\draw[-latex,black] (-5,-2) -- (-5.9,-2.9);
	\draw[-latex,black] (-6,-3) -- (-6.5,-3.5);
	
	\draw[-latex,black] (-5,-1) -- (-5.9,-1.9);
	\draw[-latex,black] (-6,-2) -- (-6.5,-2.5);
	
	\draw[-latex,black] (-6,-1) -- (-6.5,-1.5);

	%\draw[-latex,black] (-1,-2) -- (-1.9,-1.1);
	%
	%\draw[-latex,black] (-1,-3) -- (-1.9,-2.1);
	%\draw[-latex,black] (-2,-2) -- (-2.9,-1.1);
	%
	%\draw[-latex,black] (-1,-4) -- (-1.9,-3.1);
	%\draw[-latex,black] (-2,-3) -- (-2.9,-2.1);
	%\draw[-latex,black] (-3,-2) -- (-3.9,-1.1);
	%
	%
	%\draw[-latex,black] (-1,-5) -- (-1.9,-4.1);
	%\draw[-latex,black] (-2,-4) -- (-2.9,-3.1);
	%\draw[-latex,black] (-3,-3) -- (-3.9,-2.1);
	%\draw[-latex,black] (-4,-2) -- (-4.9,-1.1);
	%
	%
	%\draw[-latex,black] (-1,-6) -- (-1.9,-5.1);
	%\draw[-latex,black] (-2,-5) -- (-2.9,-4.1);
	%\draw[-latex,black] (-3,-4) -- (-3.9,-3.1);
	%\draw[-latex,black] (-4,-3) -- (-4.9,-2.1);
	%\draw[-latex,black] (-5,-2) -- (-5.9,-1.1);
	%
	%\draw[-latex,black] (-1,-7) -- (-1.9,-6.1);
	%\draw[-latex,black] (-2,-6) -- (-2.9,-5.1);
	%\draw[-latex,black] (-3,-5) -- (-3.9,-4.1);
	%\draw[-latex,black] (-4,-4) -- (-4.9,-3.1);
	%\draw[-latex,black] (-5,-3) -- (-5.9,-2.1);
	%\draw[-latex,black] (-6,-2) -- (-6.9,-1.1);
	%
	%
	%\draw[-latex,black] (-2,-7) -- (-2.9,-6.1);
	%\draw[-latex,black] (-3,-6) -- (-3.9,-5.1);
	%\draw[-latex,black] (-4,-5) -- (-4.9,-4.1);
	%\draw[-latex,black] (-5,-4) -- (-5.9,-3.1);
	%\draw[-latex,black] (-6,-3) -- (-6.9,-2.1);
	%
	%\draw[-latex,black] (-3,-7) -- (-3.9,-6.1);
	%\draw[-latex,black] (-4,-6) -- (-4.9,-5.1);
	%\draw[-latex,black] (-5,-5) -- (-5.9,-4.1);
	%\draw[-latex,black] (-6,-4) -- (-6.9,-3.1);
	%
	%
	%\draw[-latex,black] (-4,-7) -- (-4.9,-6.1);
	%\draw[-latex,black] (-5,-6) -- (-5.9,-5.1);
	%\draw[-latex,black] (-6,-5) -- (-6.9,-4.1);
	%
	%\draw[-latex,black] (-5,-7) -- (-5.9,-6.1);
	%\draw[-latex,black] (-6,-6) -- (-6.9,-5.1);
	%
	%\draw[-latex,black] (-6,-7) -- (-6.9,-6.1);
	
	\coordinate (center) at (-1,-1);
	\fill[white] (center) + (0, 0.1) arc (90:270:0.1);
	\fill[white] (center)+ (0, -0.1) arc (270:450:0.1);
	\draw[black] (center)+ (0, -0.1) arc (270:450:0.1);
	\draw[black] (center)+ (0, 0.1) arc (90:270:0.1);
	\coordinate (center) at (-1,-2);
	\fill[white] (center) + (0, 0.1) arc (90:270:0.1);
	\fill[white] (center)+ (0, -0.1) arc (270:450:0.1);
	\draw[black] (center)+ (0, -0.1) arc (270:450:0.1);
	\draw[black] (center)+ (0, 0.1) arc (90:270:0.1);
	\coordinate (center) at (-1,-3);
	\fill[white] (center) + (0, 0.1) arc (90:270:0.1);
	\fill[white] (center)+ (0, -0.1) arc (270:450:0.1);
	\draw[black] (center)+ (0, -0.1) arc (270:450:0.1);
	\draw[black] (center)+ (0, 0.1) arc (90:270:0.1);
	\coordinate (center) at (-1,-4);
	\fill[white] (center) + (0, 0.1) arc (90:270:0.1);
	\fill[white] (center)+ (0, -0.1) arc (270:450:0.1);
	\draw[black] (center)+ (0, -0.1) arc (270:450:0.1);
	\draw[black] (center)+ (0, 0.1) arc (90:270:0.1);
	\coordinate (center) at (-1,-5);
	\fill[white] (center) + (0, 0.1) arc (90:270:0.1);
	\fill[white] (center)+ (0, -0.1) arc (270:450:0.1);
	\draw[black] (center)+ (0, -0.1) arc (270:450:0.1);
	\draw[black] (center)+ (0, 0.1) arc (90:270:0.1);
	\coordinate (center) at (-1,-6);
	\fill[white] (center) + (0, 0.1) arc (90:270:0.1);
	\fill[white] (center)+ (0, -0.1) arc (270:450:0.1);
	\draw[black] (center)+ (0, -0.1) arc (270:450:0.1);
	\draw[black] (center)+ (0, 0.1) arc (90:270:0.1);
	\coordinate (center) at (-2,-1);
	\fill[white] (center) + (0, 0.1) arc (90:270:0.1);
	\fill[white] (center)+ (0, -0.1) arc (270:450:0.1);
	\draw[black] (center)+ (0, -0.1) arc (270:450:0.1);
	\draw[black] (center)+ (0, 0.1) arc (90:270:0.1);
	\coordinate (center) at (-3,-1);
	\fill[white] (center) + (0, 0.1) arc (90:270:0.1);
	\fill[white] (center)+ (0, -0.1) arc (270:450:0.1);
	\draw[black] (center)+ (0, -0.1) arc (270:450:0.1);
	\draw[black] (center)+ (0, 0.1) arc (90:270:0.1);
	\coordinate (center) at (-4,-1);
	\fill[white] (center) + (0, 0.1) arc (90:270:0.1);
	\fill[white] (center)+ (0, -0.1) arc (270:450:0.1);
	\draw[black] (center)+ (0, -0.1) arc (270:450:0.1);
	\draw[black] (center)+ (0, 0.1) arc (90:270:0.1);
	\coordinate (center) at (-5,-1);
	\fill[white] (center) + (0, 0.1) arc (90:270:0.1);
	\fill[white] (center)+ (0, -0.1) arc (270:450:0.1);
	\draw[black] (center)+ (0, -0.1) arc (270:450:0.1);
	\draw[black] (center)+ (0, 0.1) arc (90:270:0.1);
	\coordinate (center) at (-6,-1);
	\fill[white] (center) + (0, 0.1) arc (90:270:0.1);
	\fill[white] (center)+ (0, -0.1) arc (270:450:0.1);
	\draw[black] (center)+ (0, -0.1) arc (270:450:0.1);
	\draw[black] (center)+ (0, 0.1) arc (90:270:0.1);

	\coordinate (center) at (-2,-2);
	\fill[white] (center) + (0, 0.1) arc (90:270:0.1);
	\fill[white] (center)+ (0, -0.1) arc (270:450:0.1);
	\draw[black] (center)+ (0, -0.1) arc (270:450:0.1);
	\draw[black] (center)+ (0, 0.1) arc (90:270:0.1);
	\coordinate (center) at (-2,-3);
	\fill[white] (center) + (0, 0.1) arc (90:270:0.1);
	\fill[white] (center)+ (0, -0.1) arc (270:450:0.1);
	\draw[black] (center)+ (0, -0.1) arc (270:450:0.1);
	\draw[black] (center)+ (0, 0.1) arc (90:270:0.1);
	
	\coordinate (center) at (-2,-4);
	\fill[white] (center) + (0, 0.1) arc (90:270:0.1);
	\fill[white] (center)+ (0, -0.1) arc (270:450:0.1);
	\draw[black] (center)+ (0, -0.1) arc (270:450:0.1);
	\draw[black] (center)+ (0, 0.1) arc (90:270:0.1);
	\coordinate (center) at (-2,-5);
	\fill[white] (center) + (0, 0.1) arc (90:270:0.1);
	\fill[white] (center)+ (0, -0.1) arc (270:450:0.1);
	\draw[black] (center)+ (0, -0.1) arc (270:450:0.1);
	\draw[black] (center)+ (0, 0.1) arc (90:270:0.1);
	\coordinate (center) at (-2,-6);
	\fill[white] (center) + (0, 0.1) arc (90:270:0.1);
	\fill[white] (center)+ (0, -0.1) arc (270:450:0.1);
	\draw[black] (center)+ (0, -0.1) arc (270:450:0.1);
	\draw[black] (center)+ (0, 0.1) arc (90:270:0.1);
	\coordinate (center) at (-3,-2);
	\fill[white] (center) + (0, 0.1) arc (90:270:0.1);
	\fill[white] (center)+ (0, -0.1) arc (270:450:0.1);
	\draw[black] (center)+ (0, -0.1) arc (270:450:0.1);
	\draw[black] (center)+ (0, 0.1) arc (90:270:0.1);
	\coordinate (center) at (-3,-3);
	\fill[white] (center) + (0, 0.1) arc (90:270:0.1);
	\fill[white] (center)+ (0, -0.1) arc (270:450:0.1);
	\draw[black] (center)+ (0, -0.1) arc (270:450:0.1);
	\draw[black] (center)+ (0, 0.1) arc (90:270:0.1);
	\coordinate (center) at (-3,-4);
	\fill[white] (center) + (0, 0.1) arc (90:270:0.1);
	\fill[white] (center)+ (0, -0.1) arc (270:450:0.1);
	\draw[black] (center)+ (0, -0.1) arc (270:450:0.1);
	\draw[black] (center)+ (0, 0.1) arc (90:270:0.1);
	\coordinate (center) at (-3,-5);
	\fill[white] (center) + (0, 0.1) arc (90:270:0.1);
	\fill[white] (center)+ (0, -0.1) arc (270:450:0.1);
	\draw[black] (center)+ (0, -0.1) arc (270:450:0.1);
	\draw[black] (center)+ (0, 0.1) arc (90:270:0.1);
	\coordinate (center) at (-3,-6);
	\fill[white] (center) + (0, 0.1) arc (90:270:0.1);
	\fill[white] (center)+ (0, -0.1) arc (270:450:0.1);
	\draw[black] (center)+ (0, -0.1) arc (270:450:0.1);
	\draw[black] (center)+ (0, 0.1) arc (90:270:0.1);
	\coordinate (center) at (-4,-2);
	\fill[white] (center) + (0, 0.1) arc (90:270:0.1);
	\fill[white] (center)+ (0, -0.1) arc (270:450:0.1);
	\draw[black] (center)+ (0, -0.1) arc (270:450:0.1);
	\draw[black] (center)+ (0, 0.1) arc (90:270:0.1);
	\coordinate (center) at (-4,-3);
	\fill[white] (center) + (0, 0.1) arc (90:270:0.1);
	\fill[white] (center)+ (0, -0.1) arc (270:450:0.1);
	\draw[black] (center)+ (0, -0.1) arc (270:450:0.1);
	\draw[black] (center)+ (0, 0.1) arc (90:270:0.1);
	\coordinate (center) at (-4,-4);
	\fill[white] (center) + (0, 0.1) arc (90:270:0.1);
	\fill[white] (center)+ (0, -0.1) arc (270:450:0.1);
	\draw[black] (center)+ (0, -0.1) arc (270:450:0.1);
	\draw[black] (center)+ (0, 0.1) arc (90:270:0.1);
	\coordinate (center) at (-4,-5);
	\fill[white] (center) + (0, 0.1) arc (90:270:0.1);
	\fill[white] (center)+ (0, -0.1) arc (270:450:0.1);
	\draw[black] (center)+ (0, -0.1) arc (270:450:0.1);
	\draw[black] (center)+ (0, 0.1) arc (90:270:0.1);
	\coordinate (center) at (-4,-6);
	\fill[white] (center) + (0, 0.1) arc (90:270:0.1);
	\fill[white] (center)+ (0, -0.1) arc (270:450:0.1);
	\draw[black] (center)+ (0, -0.1) arc (270:450:0.1);
	\draw[black] (center)+ (0, 0.1) arc (90:270:0.1);
	\coordinate (center) at (-5,-2);
	\fill[white] (center) + (0, 0.1) arc (90:270:0.1);
	\fill[white] (center)+ (0, -0.1) arc (270:450:0.1);
	\draw[black] (center)+ (0, -0.1) arc (270:450:0.1);
	\draw[black] (center)+ (0, 0.1) arc (90:270:0.1);
	\coordinate (center) at (-5,-3);
	\fill[white] (center) + (0, 0.1) arc (90:270:0.1);
	\fill[white] (center)+ (0, -0.1) arc (270:450:0.1);
	\draw[black] (center)+ (0, -0.1) arc (270:450:0.1);
	\draw[black] (center)+ (0, 0.1) arc (90:270:0.1);
	\coordinate (center) at (-5,-4);
	\fill[white] (center) + (0, 0.1) arc (90:270:0.1);
	\fill[white] (center)+ (0, -0.1) arc (270:450:0.1);
	\draw[black] (center)+ (0, -0.1) arc (270:450:0.1);
	\draw[black] (center)+ (0, 0.1) arc (90:270:0.1);
	\coordinate (center) at (-5,-5);
	\fill[white] (center) + (0, 0.1) arc (90:270:0.1);
	\fill[white] (center)+ (0, -0.1) arc (270:450:0.1);
	\draw[black] (center)+ (0, -0.1) arc (270:450:0.1);
	\draw[black] (center)+ (0, 0.1) arc (90:270:0.1);
	
	\coordinate (center) at (-5,-6);
	\fill[white] (center) + (0, 0.1) arc (90:270:0.1);
	\fill[white] (center)+ (0, -0.1) arc (270:450:0.1);
	\draw[black] (center)+ (0, -0.1) arc (270:450:0.1);
	\draw[black] (center)+ (0, 0.1) arc (90:270:0.1);
	\coordinate (center) at (-6,-2);
	\fill[white] (center) + (0, 0.1) arc (90:270:0.1);
	\fill[white] (center)+ (0, -0.1) arc (270:450:0.1);
	\draw[black] (center)+ (0, -0.1) arc (270:450:0.1);
	\draw[black] (center)+ (0, 0.1) arc (90:270:0.1);
	\coordinate (center) at (-6,-3);
	\fill[white] (center) + (0, 0.1) arc (90:270:0.1);
	\fill[white] (center)+ (0, -0.1) arc (270:450:0.1);
	\draw[black] (center)+ (0, -0.1) arc (270:450:0.1);
	\draw[black] (center)+ (0, 0.1) arc (90:270:0.1);
	\coordinate (center) at (-6,-4);
	\fill[white] (center) + (0, 0.1) arc (90:270:0.1);
	\fill[white] (center)+ (0, -0.1) arc (270:450:0.1);
	\draw[black] (center)+ (0, -0.1) arc (270:450:0.1);
	\draw[black] (center)+ (0, 0.1) arc (90:270:0.1);
	\coordinate (center) at (-6,-5);
	\fill[white] (center) + (0, 0.1) arc (90:270:0.1);
	\fill[white] (center)+ (0, -0.1) arc (270:450:0.1);
	\draw[black] (center)+ (0, -0.1) arc (270:450:0.1);
	\draw[black] (center)+ (0, 0.1) arc (90:270:0.1);
	
	\coordinate (center) at (-6,-6);
	\fill[white] (center) + (0, 0.1) arc (90:270:0.1);
	\fill[white] (center)+ (0, -0.1) arc (270:450:0.1);
	\draw[black] (center)+ (0, -0.1) arc (270:450:0.1);
	\draw[black] (center)+ (0, 0.1) arc (90:270:0.1);

	\draw[->,black] (0,-1) -- (6,-1) node[anchor=north west]{};
	\draw[->,black] (0,-1) -- (0,-7) node[anchor=south east]{};
	\node[black] at (0,-7.5) {$n$};
	\node[black] at (6.5,-1) {$m$};
	\node[black] at (3,-7) {$(m,n,1+\frac{n-m}{2},1+\frac{n+m}{2})$};
	\node[black] at (6,-7) {\textbf{D}};
	\draw[black] (1,-1) -- (0,-2);
	\draw[black] (2,-1) -- (0,-3);
	\draw[black] (3,-1) -- (0,-4);
	\draw[black] (4,-1) -- (0,-5);
	\draw[black] (5,-1) -- (0,-6);
	\draw[black] (5.5,-1.5) -- (0.5,-6.5);
	\draw[black] (5.5,-2.5) -- (1.5,-6.5);
	\draw[black] (5.5,-3.5) -- (2.5,-6.5);
	\draw[black] (5.5,-4.5) -- (3.5,-6.5);
	\draw[black] (5.5,-5.5) -- (4.5,-6.5);

	\draw[-latex,black] (1,-1) -- (0.1,-1.9);
	
	\draw[-latex,black] (2,-1) -- (1.1,-1.9);
	\draw[-latex,black] (1,-2) -- (0.1,-2.9);
	
	\draw[-latex,black] (3,-1) -- (2.1,-1.9);
	\draw[-latex,black] (2,-2) -- (1.1,-2.9);
	\draw[-latex,black] (1,-3) -- (0.1,-3.9);
	
	\draw[-latex,black] (4,-1) -- (3.1,-1.9);
	\draw[-latex,black] (3,-2) -- (2.1,-2.9);
	\draw[-latex,black] (2,-3) -- (1.1,-3.9);
	\draw[-latex,black] (1,-4) -- (0.1,-4.9);
	
	\draw[-latex,black] (5,-1) -- (4.1,-1.9);
	\draw[-latex,black] (4,-2) -- (3.1,-2.9);
	\draw[-latex,black] (3,-3) -- (2.1,-3.9);
	\draw[-latex,black] (2,-4) -- (1.1,-4.9);
	\draw[-latex,black] (1,-5) -- (0.1,-5.9);

	\draw[-latex,black] (5.5,-1.5) -- (5.1,-1.9);
	\draw[-latex,black] (5,-2) -- (4.1,-2.9);
	\draw[-latex,black] (4,-3) -- (3.1,-3.9);
	\draw[-latex,black] (3,-4) -- (2.1,-4.9);
	\draw[-latex,black] (2,-5) -- (1.1,-5.9);
	\draw[-latex,black] (1,-6) -- (0.5,-6.5);

	\draw[-latex,black] (5.5,-2.5) -- (5.1,-2.9);
	\draw[-latex,black] (5,-3) -- (4.1,-3.9);
	\draw[-latex,black] (4,-4) -- (3.1,-4.9);
	\draw[-latex,black] (3,-5) -- (2.1,-5.9);
	\draw[-latex,black] (2,-6) -- (1.5,-6.5);
	
	\draw[-latex,black] (5.5,-3.5) -- (5.1,-3.9);
	\draw[-latex,black] (5,-4) -- (4.1,-4.9);
	\draw[-latex,black] (4,-5) -- (3.1,-5.9);
	\draw[-latex,black] (3,-6) -- (2.5,-6.5);
	
	\draw[-latex,black] (5.5,-4.5) -- (5.1,-4.9);
	\draw[-latex,black] (5,-5) -- (4.1,-5.9);
	\draw[-latex,black] (4,-6) -- (3.5,-6.5);

	\draw[-latex,black] (5.5,-5.5) -- (5.1,-5.9);
	\draw[-latex,black] (5,-6) -- (4.5,-6.5);
	%\coordinate (center) at (0,-1);

	%\fill[white] (center) + (0, 0.1) arc (90:270:0.1);
	%\fill[white] (center)+ (0, -0.1) arc (270:450:0.1);
	%\draw[black] (center)+ (0, -0.1) arc (270:450:0.1);
	%\draw[black] (center)+ (0, 0.1) arc (90:270:0.1);
	\coordinate (center) at (0,-2);
	\fill[white] (center) + (0, 0.1) arc (90:270:0.1);
	\fill[white] (center)+ (0, -0.1) arc (270:450:0.1);
	\draw[black] (center)+ (0, -0.1) arc (270:450:0.1);
	\draw[black] (center)+ (0, 0.1) arc (90:270:0.1);
	\coordinate (center) at (0,-3);
	\fill[white] (center) + (0, 0.1) arc (90:270:0.1);
	\fill[white] (center)+ (0, -0.1) arc (270:450:0.1);
	\draw[black] (center)+ (0, -0.1) arc (270:450:0.1);
	\draw[black] (center)+ (0, 0.1) arc (90:270:0.1);
	\coordinate (center) at (0,-4);
	\fill[white] (center) + (0, 0.1) arc (90:270:0.1);
	\fill[white] (center)+ (0, -0.1) arc (270:450:0.1);
	\draw[black] (center)+ (0, -0.1) arc (270:450:0.1);
	\draw[black] (center)+ (0, 0.1) arc (90:270:0.1);
	\coordinate (center) at (0,-5);
	\fill[white] (center) + (0, 0.1) arc (90:270:0.1);
	\fill[white] (center)+ (0, -0.1) arc (270:450:0.1);
	\draw[black] (center)+ (0, -0.1) arc (270:450:0.1);
	\draw[black] (center)+ (0, 0.1) arc (90:270:0.1);
	\coordinate (center) at (0,-6);
	\fill[white] (center) + (0, 0.1) arc (90:270:0.1);
	\fill[white] (center)+ (0, -0.1) arc (270:450:0.1);
	\draw[black] (center)+ (0, -0.1) arc (270:450:0.1);
	\draw[black] (center)+ (0, 0.1) arc (90:270:0.1);
	\coordinate (center) at (1,-1);
	\fill[white] (center) + (0, 0.1) arc (90:270:0.1);
	\fill[white] (center)+ (0, -0.1) arc (270:450:0.1);
	\draw[black] (center)+ (0, -0.1) arc (270:450:0.1);
	\draw[black] (center)+ (0, 0.1) arc (90:270:0.1);
	\coordinate (center) at (2,-1);
	\fill[white] (center) + (0, 0.1) arc (90:270:0.1);
	\fill[white] (center)+ (0, -0.1) arc (270:450:0.1);
	\draw[black] (center)+ (0, -0.1) arc (270:450:0.1);
	\draw[black] (center)+ (0, 0.1) arc (90:270:0.1);
	\coordinate (center) at (3,-1);
	\fill[white] (center) + (0, 0.1) arc (90:270:0.1);
	\fill[white] (center)+ (0, -0.1) arc (270:450:0.1);
	\draw[black] (center)+ (0, -0.1) arc (270:450:0.1);
	\draw[black] (center)+ (0, 0.1) arc (90:270:0.1);
	\coordinate (center) at (4,-1);
	\fill[white] (center) + (0, 0.1) arc (90:270:0.1);
	\fill[white] (center)+ (0, -0.1) arc (270:450:0.1);
	\draw[black] (center)+ (0, -0.1) arc (270:450:0.1);
	\draw[black] (center)+ (0, 0.1) arc (90:270:0.1);
	\coordinate (center) at (5,-1);
	\fill[white] (center) + (0, 0.1) arc (90:270:0.1);
	\fill[white] (center)+ (0, -0.1) arc (270:450:0.1);
	\draw[black] (center)+ (0, -0.1) arc (270:450:0.1);
	\draw[black] (center)+ (0, 0.1) arc (90:270:0.1);

	\coordinate (center) at (1,-2);
	\fill[white] (center) + (0, 0.1) arc (90:270:0.1);
	\fill[white] (center)+ (0, -0.1) arc (270:450:0.1);
	\draw[black] (center)+ (0, -0.1) arc (270:450:0.1);
	\draw[black] (center)+ (0, 0.1) arc (90:270:0.1);
	\coordinate (center) at (1,-3);
	\fill[white] (center) + (0, 0.1) arc (90:270:0.1);
	\fill[white] (center)+ (0, -0.1) arc (270:450:0.1);
	\draw[black] (center)+ (0, -0.1) arc (270:450:0.1);
	\draw[black] (center)+ (0, 0.1) arc (90:270:0.1);
	
	\coordinate (center) at (1,-4);
	\fill[white] (center) + (0, 0.1) arc (90:270:0.1);
	\fill[white] (center)+ (0, -0.1) arc (270:450:0.1);
	\draw[black] (center)+ (0, -0.1) arc (270:450:0.1);
	\draw[black] (center)+ (0, 0.1) arc (90:270:0.1);
	\coordinate (center) at (1,-5);
	\fill[white] (center) + (0, 0.1) arc (90:270:0.1);
	\fill[white] (center)+ (0, -0.1) arc (270:450:0.1);
	\draw[black] (center)+ (0, -0.1) arc (270:450:0.1);
	\draw[black] (center)+ (0, 0.1) arc (90:270:0.1);
	\coordinate (center) at (1,-6);
	\fill[white] (center) + (0, 0.1) arc (90:270:0.1);
	\fill[white] (center)+ (0, -0.1) arc (270:450:0.1);
	\draw[black] (center)+ (0, -0.1) arc (270:450:0.1);
	\draw[black] (center)+ (0, 0.1) arc (90:270:0.1);
	\coordinate (center) at (2,-2);
	\fill[white] (center) + (0, 0.1) arc (90:270:0.1);
	\fill[white] (center)+ (0, -0.1) arc (270:450:0.1);
	\draw[black] (center)+ (0, -0.1) arc (270:450:0.1);
	\draw[black] (center)+ (0, 0.1) arc (90:270:0.1);
	\coordinate (center) at (2,-3);
	\fill[white] (center) + (0, 0.1) arc (90:270:0.1);
	\fill[white] (center)+ (0, -0.1) arc (270:450:0.1);
	\draw[black] (center)+ (0, -0.1) arc (270:450:0.1);
	\draw[black] (center)+ (0, 0.1) arc (90:270:0.1);
	\coordinate (center) at (2,-4);
	\fill[white] (center) + (0, 0.1) arc (90:270:0.1);
	\fill[white] (center)+ (0, -0.1) arc (270:450:0.1);
	\draw[black] (center)+ (0, -0.1) arc (270:450:0.1);
	\draw[black] (center)+ (0, 0.1) arc (90:270:0.1);
	\coordinate (center) at (2,-5);
	\fill[white] (center) + (0, 0.1) arc (90:270:0.1);
	\fill[white] (center)+ (0, -0.1) arc (270:450:0.1);
	\draw[black] (center)+ (0, -0.1) arc (270:450:0.1);
	\draw[black] (center)+ (0, 0.1) arc (90:270:0.1);
	\coordinate (center) at (2,-6);
	\fill[white] (center) + (0, 0.1) arc (90:270:0.1);
	\fill[white] (center)+ (0, -0.1) arc (270:450:0.1);
	\draw[black] (center)+ (0, -0.1) arc (270:450:0.1);
	\draw[black] (center)+ (0, 0.1) arc (90:270:0.1);
	\coordinate (center) at (3,-2);
	\fill[white] (center) + (0, 0.1) arc (90:270:0.1);
	\fill[white] (center)+ (0, -0.1) arc (270:450:0.1);
	\draw[black] (center)+ (0, -0.1) arc (270:450:0.1);
	\draw[black] (center)+ (0, 0.1) arc (90:270:0.1);
	\coordinate (center) at (3,-3);
	\fill[white] (center) + (0, 0.1) arc (90:270:0.1);
	\fill[white] (center)+ (0, -0.1) arc (270:450:0.1);
	\draw[black] (center)+ (0, -0.1) arc (270:450:0.1);
	\draw[black] (center)+ (0, 0.1) arc (90:270:0.1);
	\coordinate (center) at (3,-4);
	\fill[white] (center) + (0, 0.1) arc (90:270:0.1);
	\fill[white] (center)+ (0, -0.1) arc (270:450:0.1);
	\draw[black] (center)+ (0, -0.1) arc (270:450:0.1);
	\draw[black] (center)+ (0, 0.1) arc (90:270:0.1);
	\coordinate (center) at (3,-5);
	\fill[white] (center) + (0, 0.1) arc (90:270:0.1);
	\fill[white] (center)+ (0, -0.1) arc (270:450:0.1);
	\draw[black] (center)+ (0, -0.1) arc (270:450:0.1);
	\draw[black] (center)+ (0, 0.1) arc (90:270:0.1);
	\coordinate (center) at (3,-6);
	\fill[white] (center) + (0, 0.1) arc (90:270:0.1);
	\fill[white] (center)+ (0, -0.1) arc (270:450:0.1);
	\draw[black] (center)+ (0, -0.1) arc (270:450:0.1);
	\draw[black] (center)+ (0, 0.1) arc (90:270:0.1);
	\coordinate (center) at (4,-2);
	\fill[white] (center) + (0, 0.1) arc (90:270:0.1);
	\fill[white] (center)+ (0, -0.1) arc (270:450:0.1);
	\draw[black] (center)+ (0, -0.1) arc (270:450:0.1);
	\draw[black] (center)+ (0, 0.1) arc (90:270:0.1);
	\coordinate (center) at (4,-3);
	\fill[white] (center) + (0, 0.1) arc (90:270:0.1);
	\fill[white] (center)+ (0, -0.1) arc (270:450:0.1);
	\draw[black] (center)+ (0, -0.1) arc (270:450:0.1);
	\draw[black] (center)+ (0, 0.1) arc (90:270:0.1);
	\coordinate (center) at (4,-4);
	\fill[white] (center) + (0, 0.1) arc (90:270:0.1);
	\fill[white] (center)+ (0, -0.1) arc (270:450:0.1);
	\draw[black] (center)+ (0, -0.1) arc (270:450:0.1);
	\draw[black] (center)+ (0, 0.1) arc (90:270:0.1);
	\coordinate (center) at (4,-5);
	\fill[white] (center) + (0, 0.1) arc (90:270:0.1);
	\fill[white] (center)+ (0, -0.1) arc (270:450:0.1);
	\draw[black] (center)+ (0, -0.1) arc (270:450:0.1);
	\draw[black] (center)+ (0, 0.1) arc (90:270:0.1);
	
	\coordinate (center) at (4,-6);
	\fill[white] (center) + (0, 0.1) arc (90:270:0.1);
	\fill[white] (center)+ (0, -0.1) arc (270:450:0.1);
	\draw[black] (center)+ (0, -0.1) arc (270:450:0.1);
	\draw[black] (center)+ (0, 0.1) arc (90:270:0.1);
	\coordinate (center) at (5,-2);
	\fill[white] (center) + (0, 0.1) arc (90:270:0.1);
	\fill[white] (center)+ (0, -0.1) arc (270:450:0.1);
	\draw[black] (center)+ (0, -0.1) arc (270:450:0.1);
	\draw[black] (center)+ (0, 0.1) arc (90:270:0.1);
	\coordinate (center) at (5,-3);
	\fill[white] (center) + (0, 0.1) arc (90:270:0.1);
	\fill[white] (center)+ (0, -0.1) arc (270:450:0.1);
	\draw[black] (center)+ (0, -0.1) arc (270:450:0.1);
	\draw[black] (center)+ (0, 0.1) arc (90:270:0.1);
	\coordinate (center) at (5,-4);
	\fill[white] (center) + (0, 0.1) arc (90:270:0.1);
	\fill[white] (center)+ (0, -0.1) arc (270:450:0.1);
	\draw[black] (center)+ (0, -0.1) arc (270:450:0.1);
	\draw[black] (center)+ (0, 0.1) arc (90:270:0.1);
	\coordinate (center) at (5,-5);
	\fill[white] (center) + (0, 0.1) arc (90:270:0.1);
	\fill[white] (center)+ (0, -0.1) arc (270:450:0.1);
	\draw[black] (center)+ (0, -0.1) arc (270:450:0.1);
	\draw[black] (center)+ (0, 0.1) arc (90:270:0.1);
	
	\coordinate (center) at (5,-6);
	\fill[white] (center) + (0, 0.1) arc (90:270:0.1);
	\fill[white] (center)+ (0, -0.1) arc (270:450:0.1);
	\draw[black] (center)+ (0, -0.1) arc (270:450:0.1);
	\draw[black] (center)+ (0, 0.1) arc (90:270:0.1);

	\draw[->,black] (-1,0) -- (-1,6) node[anchor=north west]{};
	\draw[->,black] (-1,0) -- (-7,0) node[anchor=south east]{};
	\node[black] at (-7.5,0) {$m$};
	\node[black] at (-1,6.5) {$n$};
	\node[black] at (-4,6.5) {$(m,n,1+\frac{m-n}{2},-1-\frac{m+n}{2})$};
	\node[black] at (-7,6.5) {\textbf{B}};
	\draw[black] (-2,0) -- (-1,1);
	\draw[black] (-3,0) -- (-1,2);
	\draw[black] (-4,0) -- (-1,3);
	\draw[black] (-5,0) -- (-1,4);
	\draw[black] (-6,0) -- (-1,5);
	\draw[black] (-6.5,0.5) -- (-1.5,5.5);
	\draw[black] (-6.5,1.5) -- (-2.5,5.5);
	\draw[black] (-6.5,2.5) -- (-3.5,5.5);
	\draw[black] (-6.5,3.5) -- (-4.5,5.5);
	\draw[black] (-6.5,4.5) -- (-5.5,5.5);

	\draw[-latex,black] (-1,1) -- (-1.9,0.1);
	
	\draw[-latex,black] (-1,2) -- (-1.9,1.1);
	\draw[-latex,black] (-2,1) -- (-2.9,0.1);
	
	\draw[-latex,black] (-1,3) -- (-1.9,2.1);
	\draw[-latex,black] (-2,2) -- (-2.9,1.1);
	\draw[-latex,black] (-3,1) -- (-3.9,0.1);
	
	\draw[-latex,black] (-1,4) -- (-1.9,3.1);
	\draw[-latex,black] (-2,3) -- (-2.9,2.1);
	\draw[-latex,black] (-3,2) -- (-3.9,1.1);
	\draw[-latex,black] (-4,1) -- (-4.9,0.1);
	
	\draw[-latex,black] (-1,5) -- (-1.9,4.1);
	\draw[-latex,black] (-2,4) -- (-2.9,3.1);
	\draw[-latex,black] (-3,3) -- (-3.9,2.1);
	\draw[-latex,black] (-4,2) -- (-4.9,1.1);
	\draw[-latex,black] (-5,1) -- (-5.9,0.1);
	
	\draw[-latex,black] (-1.5,5.5) -- (-1.9,5.1);
	\draw[-latex,black] (-2,5) -- (-2.9,4.1);
	\draw[-latex,black] (-3,4) -- (-3.9,3.1);
	\draw[-latex,black] (-4,3) -- (-4.9,2.1);
	\draw[-latex,black] (-5,2) -- (-5.9,1.1);
	\draw[-latex,black] (-6,1) -- (-6.5,0.5);
	
	\draw[-latex,black] (-2.5,5.5) -- (-2.9,5.1);
	\draw[-latex,black] (-3,5) -- (-3.9,4.1);
	\draw[-latex,black] (-4,4) -- (-4.9,3.1);
	\draw[-latex,black] (-5,3) -- (-5.9,2.1);
	\draw[-latex,black] (-6,2) -- (-6.5,1.5);
	
	\draw[-latex,black] (-3.5,5.5) -- (-3.9,5.1);
	\draw[-latex,black] (-4,5) -- (-4.9,4.1);
	\draw[-latex,black] (-5,4) -- (-5.9,3.1);
	\draw[-latex,black] (-6,3) -- (-6.5,2.5);
	
	\draw[-latex,black] (-4.5,5.5) -- (-4.9,5.1);
	\draw[-latex,black] (-5,5) -- (-5.9,4.1);
	\draw[-latex,black] (-6,4) -- (-6.5,3.5);
	
	\draw[-latex,black] (-5.5,5.5) -- (-5.9,5.1);
	\draw[-latex,black] (-6,5) -- (-6.5,4.5);
	%
	%\draw[-latex,black] (2,-1) -- (1.1,-1.9);
	%\draw[-latex,black] (1,-2) -- (0.1,-2.9);
	
	%\coordinate (center) at (-1,0);
	%\fill[white] (center) + (0, 0.1) arc (90:270:0.1);
	%\fill[white] (center)+ (0, -0.1) arc (270:450:0.1);
	%\draw[black] (center)+ (0, -0.1) arc (270:450:0.1);
	%\draw[black] (center)+ (0, 0.1) arc (90:270:0.1);
	\coordinate (center) at (-1,1);
	\fill[white] (center) + (0, 0.1) arc (90:270:0.1);
	\fill[white] (center)+ (0, -0.1) arc (270:450:0.1);
	\draw[black] (center)+ (0, -0.1) arc (270:450:0.1);
	\draw[black] (center)+ (0, 0.1) arc (90:270:0.1);
	\coordinate (center) at (-1,2);
	\fill[white] (center) + (0, 0.1) arc (90:270:0.1);
	\fill[white] (center)+ (0, -0.1) arc (270:450:0.1);
	\draw[black] (center)+ (0, -0.1) arc (270:450:0.1);
	\draw[black] (center)+ (0, 0.1) arc (90:270:0.1);
	\coordinate (center) at (-1,3);
	\fill[white] (center) + (0, 0.1) arc (90:270:0.1);
	\fill[white] (center)+ (0, -0.1) arc (270:450:0.1);
	\draw[black] (center)+ (0, -0.1) arc (270:450:0.1);
	\draw[black] (center)+ (0, 0.1) arc (90:270:0.1);
	\coordinate (center) at (-1,4);
	\fill[white] (center) + (0, 0.1) arc (90:270:0.1);
	\fill[white] (center)+ (0, -0.1) arc (270:450:0.1);
	\draw[black] (center)+ (0, -0.1) arc (270:450:0.1);
	\draw[black] (center)+ (0, 0.1) arc (90:270:0.1);
	\coordinate (center) at (-1,5);
	\fill[white] (center) + (0, 0.1) arc (90:270:0.1);
	\fill[white] (center)+ (0, -0.1) arc (270:450:0.1);
	\draw[black] (center)+ (0, -0.1) arc (270:450:0.1);
	\draw[black] (center)+ (0, 0.1) arc (90:270:0.1);
	\coordinate (center) at (-2,0);
	\fill[white] (center) + (0, 0.1) arc (90:270:0.1);
	\fill[white] (center)+ (0, -0.1) arc (270:450:0.1);
	\draw[black] (center)+ (0, -0.1) arc (270:450:0.1);
	\draw[black] (center)+ (0, 0.1) arc (90:270:0.1);
	\coordinate (center) at (-3,0);
	\fill[white] (center) + (0, 0.1) arc (90:270:0.1);
	\fill[white] (center)+ (0, -0.1) arc (270:450:0.1);
	\draw[black] (center)+ (0, -0.1) arc (270:450:0.1);
	\draw[black] (center)+ (0, 0.1) arc (90:270:0.1);
	\coordinate (center) at (-4,0);
	\fill[white] (center) + (0, 0.1) arc (90:270:0.1);
	\fill[white] (center)+ (0, -0.1) arc (270:450:0.1);
	\draw[black] (center)+ (0, -0.1) arc (270:450:0.1);
	\draw[black] (center)+ (0, 0.1) arc (90:270:0.1);
	\coordinate (center) at (-5,0);
	\fill[white] (center) + (0, 0.1) arc (90:270:0.1);
	\fill[white] (center)+ (0, -0.1) arc (270:450:0.1);
	\draw[black] (center)+ (0, -0.1) arc (270:450:0.1);
	\draw[black] (center)+ (0, 0.1) arc (90:270:0.1);
	\coordinate (center) at (-6,0);
	\fill[white] (center) + (0, 0.1) arc (90:270:0.1);
	\fill[white] (center)+ (0, -0.1) arc (270:450:0.1);
	\draw[black] (center)+ (0, -0.1) arc (270:450:0.1);
	\draw[black] (center)+ (0, 0.1) arc (90:270:0.1);

	\coordinate (center) at (-2,1);
	\fill[white] (center) + (0, 0.1) arc (90:270:0.1);
	\fill[white] (center)+ (0, -0.1) arc (270:450:0.1);
	\draw[black] (center)+ (0, -0.1) arc (270:450:0.1);
	\draw[black] (center)+ (0, 0.1) arc (90:270:0.1);
	\coordinate (center) at (-2,2);
	\fill[white] (center) + (0, 0.1) arc (90:270:0.1);
	\fill[white] (center)+ (0, -0.1) arc (270:450:0.1);
	\draw[black] (center)+ (0, -0.1) arc (270:450:0.1);
	\draw[black] (center)+ (0, 0.1) arc (90:270:0.1);
	\coordinate (center) at (-2,3);
	\fill[white] (center) + (0, 0.1) arc (90:270:0.1);
	\fill[white] (center)+ (0, -0.1) arc (270:450:0.1);
	\draw[black] (center)+ (0, -0.1) arc (270:450:0.1);
	\draw[black] (center)+ (0, 0.1) arc (90:270:0.1);
	\coordinate (center) at (-2,4);
	\fill[white] (center) + (0, 0.1) arc (90:270:0.1);
	\fill[white] (center)+ (0, -0.1) arc (270:450:0.1);
	\draw[black] (center)+ (0, -0.1) arc (270:450:0.1);
	\draw[black] (center)+ (0, 0.1) arc (90:270:0.1);
	\coordinate (center) at (-2,5);
	\fill[white] (center) + (0, 0.1) arc (90:270:0.1);
	\fill[white] (center)+ (0, -0.1) arc (270:450:0.1);
	\draw[black] (center)+ (0, -0.1) arc (270:450:0.1);
	\draw[black] (center)+ (0, 0.1) arc (90:270:0.1);
	\coordinate (center) at (-3,1);
	\fill[white] (center) + (0, 0.1) arc (90:270:0.1);
	\fill[white] (center)+ (0, -0.1) arc (270:450:0.1);
	\draw[black] (center)+ (0, -0.1) arc (270:450:0.1);
	\draw[black] (center)+ (0, 0.1) arc (90:270:0.1);
	\coordinate (center) at (-3,2);
	\fill[white] (center) + (0, 0.1) arc (90:270:0.1);
	\fill[white] (center)+ (0, -0.1) arc (270:450:0.1);
	\draw[black] (center)+ (0, -0.1) arc (270:450:0.1);
	\draw[black] (center)+ (0, 0.1) arc (90:270:0.1);
	\coordinate (center) at (-3,3);
	\fill[white] (center) + (0, 0.1) arc (90:270:0.1);
	\fill[white] (center)+ (0, -0.1) arc (270:450:0.1);
	\draw[black] (center)+ (0, -0.1) arc (270:450:0.1);
	\draw[black] (center)+ (0, 0.1) arc (90:270:0.1);
	\coordinate (center) at (-3,4);
	\fill[white] (center) + (0, 0.1) arc (90:270:0.1);
	\fill[white] (center)+ (0, -0.1) arc (270:450:0.1);
	\draw[black] (center)+ (0, -0.1) arc (270:450:0.1);
	\draw[black] (center)+ (0, 0.1) arc (90:270:0.1);
	\coordinate (center) at (-3,5);
	\fill[white] (center) + (0, 0.1) arc (90:270:0.1);
	\fill[white] (center)+ (0, -0.1) arc (270:450:0.1);
	\draw[black] (center)+ (0, -0.1) arc (270:450:0.1);
	\draw[black] (center)+ (0, 0.1) arc (90:270:0.1);
	\coordinate (center) at (-4,1);
	\fill[white] (center) + (0, 0.1) arc (90:270:0.1);
	\fill[white] (center)+ (0, -0.1) arc (270:450:0.1);
	\draw[black] (center)+ (0, -0.1) arc (270:450:0.1);
	\draw[black] (center)+ (0, 0.1) arc (90:270:0.1);
	\coordinate (center) at (-4,2);
	\fill[white] (center) + (0, 0.1) arc (90:270:0.1);
	\fill[white] (center)+ (0, -0.1) arc (270:450:0.1);
	\draw[black] (center)+ (0, -0.1) arc (270:450:0.1);
	\draw[black] (center)+ (0, 0.1) arc (90:270:0.1);
	\coordinate (center) at (-4,3);
	\fill[white] (center) + (0, 0.1) arc (90:270:0.1);
	\fill[white] (center)+ (0, -0.1) arc (270:450:0.1);
	\draw[black] (center)+ (0, -0.1) arc (270:450:0.1);
	\draw[black] (center)+ (0, 0.1) arc (90:270:0.1);
	\coordinate (center) at (-4,4);
	\fill[white] (center) + (0, 0.1) arc (90:270:0.1);
	\fill[white] (center)+ (0, -0.1) arc (270:450:0.1);
	\draw[black] (center)+ (0, -0.1) arc (270:450:0.1);
	\draw[black] (center)+ (0, 0.1) arc (90:270:0.1);
	\coordinate (center) at (-4,5);
	\fill[white] (center) + (0, 0.1) arc (90:270:0.1);
	\fill[white] (center)+ (0, -0.1) arc (270:450:0.1);
	\draw[black] (center)+ (0, -0.1) arc (270:450:0.1);
	\draw[black] (center)+ (0, 0.1) arc (90:270:0.1);
	\coordinate (center) at (-5,1);
	\fill[white] (center) + (0, 0.1) arc (90:270:0.1);
	\fill[white] (center)+ (0, -0.1) arc (270:450:0.1);
	\draw[black] (center)+ (0, -0.1) arc (270:450:0.1);
	\draw[black] (center)+ (0, 0.1) arc (90:270:0.1);
	\coordinate (center) at (-5,2);
	\fill[white] (center) + (0, 0.1) arc (90:270:0.1);
	\fill[white] (center)+ (0, -0.1) arc (270:450:0.1);
	\draw[black] (center)+ (0, -0.1) arc (270:450:0.1);
	\draw[black] (center)+ (0, 0.1) arc (90:270:0.1);
	\coordinate (center) at (-5,3);
	\fill[white] (center) + (0, 0.1) arc (90:270:0.1);
	\fill[white] (center)+ (0, -0.1) arc (270:450:0.1);
	\draw[black] (center)+ (0, -0.1) arc (270:450:0.1);
	\draw[black] (center)+ (0, 0.1) arc (90:270:0.1);
	\coordinate (center) at (-5,4);
	\fill[white] (center) + (0, 0.1) arc (90:270:0.1);
	\fill[white] (center)+ (0, -0.1) arc (270:450:0.1);
	\draw[black] (center)+ (0, -0.1) arc (270:450:0.1);
	\draw[black] (center)+ (0, 0.1) arc (90:270:0.1);
	\coordinate (center) at (-5,5);
	\fill[white] (center) + (0, 0.1) arc (90:270:0.1);
	\fill[white] (center)+ (0, -0.1) arc (270:450:0.1);
	\draw[black] (center)+ (0, -0.1) arc (270:450:0.1);
	\draw[black] (center)+ (0, 0.1) arc (90:270:0.1);
	\coordinate (center) at (-6,1);
	\fill[white] (center) + (0, 0.1) arc (90:270:0.1);
	\fill[white] (center)+ (0, -0.1) arc (270:450:0.1);
	\draw[black] (center)+ (0, -0.1) arc (270:450:0.1);
	\draw[black] (center)+ (0, 0.1) arc (90:270:0.1);
	\coordinate (center) at (-6,2);
	\fill[white] (center) + (0, 0.1) arc (90:270:0.1);
	\fill[white] (center)+ (0, -0.1) arc (270:450:0.1);
	\draw[black] (center)+ (0, -0.1) arc (270:450:0.1);
	\draw[black] (center)+ (0, 0.1) arc (90:270:0.1);
	\coordinate (center) at (-6,3);
	\fill[white] (center) + (0, 0.1) arc (90:270:0.1);
	\fill[white] (center)+ (0, -0.1) arc (270:450:0.1);
	\draw[black] (center)+ (0, -0.1) arc (270:450:0.1);
	\draw[black] (center)+ (0, 0.1) arc (90:270:0.1);
	\coordinate (center) at (-6,4);
	\fill[white] (center) + (0, 0.1) arc (90:270:0.1);
	\fill[white] (center)+ (0, -0.1) arc (270:450:0.1);
	\draw[black] (center)+ (0, -0.1) arc (270:450:0.1);
	\draw[black] (center)+ (0, 0.1) arc (90:270:0.1);
	\coordinate (center) at (-6,5);
	\fill[white] (center) + (0, 0.1) arc (90:270:0.1);
	\fill[white] (center)+ (0, -0.1) arc (270:450:0.1);
	\draw[black] (center)+ (0, -0.1) arc (270:450:0.1);
	\draw[black] (center)+ (0, 0.1) arc (90:270:0.1);

	\coordinate (center) at (0,-1);
	\fill[white] (center) + (0, 0.1) arc (90:270:0.1);
	\fill[white] (center)+ (0, -0.1) arc (270:450:0.1);
	\draw[black] (center)+ (0, -0.1) arc (270:450:0.1);
	\draw[black] (center)+ (0, 0.1) arc (90:270:0.1);
	\coordinate (center) at (-1,0);
	\fill[white] (center) + (0, 0.1) arc (90:270:0.1);
	\fill[white] (center)+ (0, -0.1) arc (270:450:0.1);
	\draw[black] (center)+ (0, -0.1) arc (270:450:0.1);
	\draw[black] (center)+ (0, 0.1) arc (90:270:0.1);
	
	\end{tikzpicture}
\end{figure}
%\end{center}
From Theorems \ref{sing1}, \ref{sing2} and \ref{sing3} it follows that the module $M(0,0,2,0)$ does not contain non trivial singular vectors, hence it is irreducible due to Theorem \ref{keythmsingular}. This is also confirmed by the following result which can be skipped by the reader who is interested in the classification of singular vectors only.
%\begin{prop}
	%The module $M(0,0,2,0)$ is irreducible and it is isomorphic to the coadjoint representation of $K(1,4)_{+}$ on the restricted dual, i.e. $K(1,4)_{+}^{*}=\bigoplus_{j \in \Z}({K(1,4)_{+}}_{j})^{*}$.
%\end{prop}
\begin{prop}
     The module $M(0,0,2,0)$ is irreducible and it is isomorphic to the
coadjoint representation of $K(1,4)_{+}$ on the restricted dual, i.e.
$K(1,4)_{+}^{*}=\bigoplus_{j \in \Z}({K(1,4)_{+}}_{j})^{*}$.
\end{prop}
\begin{proof}
Observe that we consider $M(0,0,2,0)$ as a $K(1,4)_+$-module since the
action of $C$ is trivial.

     We first  show that $K(1,4)_{+}^{*}$ is an irreducible
$K(1,4)_{+}-$module.\\
     We recall that the action on the restricted dual is given, for
every $x,y \in K(1,4)_{+}$ and $f \in K(1,4)_{+}^{*}$, by:
     \begin{align}\label{resdu}
     (x.f)(y)=-(-1)^{p(x)p(f)}f([x,y]),
     \end{align}
     where $p(x)$ (resp. $p(f)$) denotes the parity of $x$ (resp. $f$)
and the bracket is given by \eqref{bracketlie}.\\
     Since we are considering the restricted dual, a basis of
$K(1,4)_{+}^{*}$ is given by the dual basis elements $(t^{n}
\xi_{I})^{*}$ with $n\geq 0$ and $I\in \mathcal I_<$. We will also
denote $\Theta^*=-2\xi_{\emptyset}^*$, so that $\Theta^*(\Theta)=1$.

     We first show by induction on $s+p$ that
     \begin{equation} \label{irre}
     (\underbrace{\Theta.\cdots(\Theta}_{s-times}.(\xi_{i_{1}}.( \cdots
(\xi_{i_{p}}.\Theta^{*}))))\cdots)=\gamma_{s,p} (t^s\xi_{i_1\cdots i_p})^*,
     \end{equation}
     for some $\gamma_{s,p}\in \C\setminus \{0\}$. If $s+p=0$ the result
is trivial, so we assume $s+p>0$.\\
     If $s>0$, by induction hypothesis, for every $n\geq 0$ and $J\in
\mathcal I_<$ we have
     \begin{align*}
     (\underbrace{\Theta.\cdots(\Theta}_{s-times}.(\xi_{i_{1}}.( \cdots
(\xi_{i_{p}}.\Theta^{*}))))(t^n\xi_J)&=-\gamma_{s-1,p}
(t^{s-1}\xi_{i_{1}\cdots i_{p}})^{*}([\Theta,t^n \xi_J])\\
     &=
     \begin{cases}\gamma_{s-1,p} s&\text{if $n=s$ and $J=i_1\cdots i_p$},\\
     0 & \text{otherwise.}
     \end{cases}
     \end{align*}
     and the claim follows in this case.\\
     If $s=0$ (and $p>0$)  for every $n\geq 0$ and $J\in \mathcal I_<$
we have

     \begin{align*}
     (\xi_{i_{1}}. \cdots
(\xi_{i_{p}}.(\Theta^{*})))(t^n\xi_J)&=\gamma_{0,p-1}(\xi_{i_{1}}.
(\xi_{i_{2} \cdots i_{p}})^{*})(t^n\xi_J)=(-1)^{p}\gamma_{0,p-1}
(\xi_{i_{2}\cdots i_{p}})^{*}([\xi_{i_{1}},t^n\xi_J])\\
     &=
     \begin{cases}
     (-1)^{p+1}\gamma_{0,p-1} &\text{if $n=0$ and $J=i_{1}\cdots i_{p}$},\\
     0 &\text{otherwise.}\\
     \end{cases}
     \end{align*}
     and the proof of \eqref{irre} is complete.

     Now we need the following observation: let $m,s\geq 0$ and $I,K\in
\mathcal I_<$ be such that $\deg(t^m\xi_I)\geq \deg(t^s \xi_K)$, i.e.
$2m+|I|-2\geq 2s+|K|-2$. Then
     \begin{equation}\label{irredue}
     t^{m+1}\xi_I.(t^s\xi_K)^*=\begin{cases} \beta_{m,I}\Theta^*
&\text{if $K=I$ and $s=m$}\\0&\text{otherwise.}\end{cases},
     \end{equation}
     for suitable $\beta_{m,I}\in \C\setminus \{0\}$. By \eqref{irredue} we deduce that if $f\in K(1,4)_+$ with $f=\sum
\alpha_{I,m}(t^m\xi_I)^*\neq 0$ and we choose the pair $I_0, m_0$ among
all pairs $(I,m)$ with $\alpha_{I,m}\neq 0$ such that $2m+|I|-2$ is
maximum, then
     $t^{m_0+1}\xi_{I_0}.f$ is a nonzero scalar multiple of $\Theta^*$.
 From this observation and \eqref{irre} we deduce that $K(1,4)_+^*$ is
irreducible.

     Now consider $M(0,0,2,0)=\Ind(F)$, where $F=\langle v\rangle$ is
the 1-dimensional $\g_0$-module of highest weight $(0,0,2,0)$. Since
$\Theta^*$ is a highest weight singular vector in $K(1,4)^*_+$ of weight
$(0,0,2,0)$ we deduce that there exists a (unique) morphism
     \[
     \varphi:M(0,0,2,0)\rightarrow K(1,4)^*_+
     \]
     such that $\varphi(v)=\Theta^*$.
     The morphism $\varphi$ is surjective by the irreducibility of
$K(1,4)^*_+$. The morphism $\varphi$ is also injective since it
preserves the degree, and homogeneous components of the same degree of
$M(0,0,2,0)$ and $K(1,4)^*_+$ have also the same dimension.
\end{proof}

In order to prove Theorems \ref{sing1}, \ref{sing2} and \ref{sing3}, we need some lemmas.		
\begin{rem}	
\label{appoggiotecnicos1s2s3}
We point out that, by Remark \ref{lemmatecnicos1s2s3}, a vector $\vec{m} \in \Ind (F)$ is a highest weight singular vector if and only if it satisfies \textbf{S0}--\textbf{S3}. Since $T$, defined as in Proposition \ref{actiondual}, is an isomorphism and $\vec{m}=T^{-1}T(\vec{m})$, the fact that $\vec{m} \in \Ind (F)$ satisfies \textbf{S0}--\textbf{S3} is equivalent to impose conditions \textbf{S0}--\textbf{S3} for $(T \circ f_{\lambda} \circ T^{-1})T(\vec{m})$, using the expression given by Proposition \ref{actiondual}.\\
Therefore in the following lemmas we will consider a vector $T(\vec{m}) \in \Ind (F)$ and we will impose that the expression for $(T \circ f_{\lambda} \circ T^{-1})T(\vec{m})=T( f_{\lambda} \vec{m})$, given by Proposition \ref{actiondual}, satisfies conditions \textbf{S0}--\textbf{S3}. We will have that $\vec{m}$ is a highest weight singular vector.
\end{rem}	
Motivated by Remark \ref{appoggiotecnicos1s2s3}, we look for a singular vector $\vec{m}$ such that
\begin{align}
\label{vettsingeta}
T(\vec{m})= \sum ^{N}_{k=0} \sum_{L\in \mathcal I_<} \Theta^{k} \eta_{L} \otimes v_{L,k},
\end{align}
with $v_{L,k} \in F$ for all $k$.
For all $k$, we will denote $v_{*,k}=v_{1234,k}$.  

In order to make clearer how the $\lambda-$action of Proposition \ref{actiondual} works for a vector as in \eqref{vettsingeta}, let us see the following example.
\begin{example}
	Let $T(\vec{m} )=  \Theta^{2} \eta_{13} \otimes v_{13,2}+  \eta_{2}  \otimes v_{2,0}$. Using Proposition \ref{actiondual} and Lemma \ref{actiontheta}, we have:
	\begin{align*}
	&T({\xi_2}_{\lambda}\vec{m})=\\
	&=-(\lambda+\Theta)^{2} \big\{-\Theta (\xi_{2} \star \eta_{13}) \otimes  v_{13,2} + \sum^{4}_{i=1}(\partial_{i}\xi_{2}\star \partial_{i}\eta_{13}) \otimes v_{13,2} \\
	& \quad +\lambda \big[ (\xi_{2} \star \eta_{13})   \otimes t.v_{13,2}+\sum^{4}_{i=1} \partial_{i}(\xi_{2\,i} \star \eta_{13} ) \otimes v_{13,2}- \sum _{i \neq j} (\partial_{i} \xi_{2j} \star \eta_{13})  \otimes \xi_{j,i}. v_{13,2}  \big] \\
	&\quad + \lambda^{2} \big[- \sum _{i < j} (\xi_{2\,ij} \star  \eta_{13} ) \otimes   \xi_{j,i}. v_{13,2}- \varepsilon_{2}  \, (\xi_{(2)^{c}} \star \eta_{13} )  \otimes Cv_{13,2} \big]\big\}\\
	&\quad- \Theta (\xi_{2} \star \eta_{2}) \otimes  v_{2,0}+ \sum^{4}_{i=1}(\partial_{i}\xi_{2}\star \partial_{i}\eta_{2}) \otimes  v_{2,0}+\lambda \big[ (\xi_{2} \star \eta_{2} ) \otimes t. v_{2,0}+\sum^{4}_{i=1} \partial_{i}(\xi_{2\,i} \star \eta_{2} ) \otimes  v_{2,0}\\
	&\quad - \sum _{i \neq j} (\partial_{i}\xi_{2j} \star \eta_{2}) \otimes \xi_{j,i}.  v_{2,0} \big] + \lambda^{2} \big[- \sum _{i < j} (\xi_{2ij} \star \eta_{2} ) \otimes   \xi_{j,i} . v_{2,0} -  \varepsilon_{(2)}  \, (\xi_{(2)^{c}} \star \eta_{2})  \otimes C v_{2,0} \big] \\
	=&-(\lambda+\Theta)^{2} \big\{\Theta  \eta_{123} \otimes  v_{13,2} +  \lambda \big[ - \eta_{123}   \otimes t.v_{13,2}+\partial_{4}(\xi_{24} \star \eta_{13} ) \otimes v_{13,2}\\
	&-  (\partial_{2} \xi_{24} \star \eta_{13})  \otimes \xi_{4,2} .v_{13,2} \big]\big\} + 1\otimes  v_{2,0}+\lambda \big[- \sum _{j\neq 2}( \xi_{j}\star \eta_{2} )\otimes \xi_{j,2} . v_{2,0} \big] + \lambda^{2} \eta_{1234}  \otimes C v_{2,0}\\
	=&-(\lambda+\Theta)^{2} \big\{\Theta  \eta_{123} \otimes  v_{13,2} +  \lambda \big[ - \eta_{123}   \otimes t.v_{13,2}+    \eta_{123}  \otimes v_{13,2}+   \eta_{134} \otimes \xi_{2,4} .v_{13,2} \big]\big\}\\
	&\quad + 1\otimes  v_{2,0}+\lambda \big[ -\eta_{12} \otimes \xi_{1,2} . v_{2,0}+\eta_{23} \otimes \xi_{2,3} . v_{2,0}+\eta_{24} \otimes \xi_{24}.  v_{2,0}\big] + \lambda^{2} \eta_{1234}  \otimes C v_{2,0}.
	\end{align*}
\end{example}
\begin{lem}\label{411}
	Let $\vec{m} \in \Ind (F)$ be a singular vector, such that $T(\vec{m}) $ is written as in \eqref{vettsingeta}. The degree of $T(\vec{m})$ in $\Theta$ is at most 3.
\end{lem}
\begin{proof}
	Using Proposition \ref{actiondual}, Lemma \ref{actiontheta} and Remark \ref{appoggiotecnicos1s2s3}, condition \textbf{S1}  for $I=\emptyset$ reduces to:
	%\begin{align*}
	%&0= \frac{d^{3}}{d \lambda^{3}}(1 _{\lambda} \vec{m})=\\
	%&\sum^{N}_{k=3} \sum_{I} k(k-1)(k-2) (\lambda +\Theta)^{k-3} [ (-2) \Theta \xi_{I} \otimes v_{I,k}+ \lambda( \xi_{I} \otimes E_{00} v_{I,k}-(n-|I|) \xi_{I} \otimes v_{I,k}) + \\
	%& + \lambda^{2}(-\sum_{i<j}\xi_{i} \xi_{j} \xi_{I}  \otimes F_{i,j} v_{I,k}) + \lambda^{3}(\delta_{deg(\xi_{I})=0}6 sgn \xi_{1}\xi_{2}\xi_{3}\xi_{4} \otimes v_{I,k}) ]+\\
	%&+3 \sum^{N}_{k=2} \sum_{I} k(k-1)(\lambda +\Theta)^{k-2} [ \xi_{I} \otimes E_{00} v_{I,k} -(n-|I|) \xi_{I} \otimes v_{I,k}) + \\
	%&-2 \lambda \sum_{i<j}\xi_{i} \xi_{j} \xi_{I}  \otimes F_{i,j} v_{I,k}+3 \lambda^{2}(\delta_{deg(\xi_{I})=0} sgn 6 \xi_{1}\xi_{2}\xi_{3}\xi_{4} \otimes v_{I,k}) ]+\\
	%& 3 \sum^{N}_{k=1} \sum_{I} k(\lambda +\Theta)^{k-1}[-2 \sum_{i<j}\xi_{i} \xi_{j} \xi_{I}  \otimes F_{i,j} v_{I,k}+ \\
	%&+6 \lambda(\delta_{deg(\xi_{I})=0} sgn 6 \xi_{1}\xi_{2}\xi_{3}\xi_{4} \otimes v_{I,k}) ]+\\
	%&\sum^{N}_{k=0} \sum_{I}(\lambda +\Theta)^{k}[6\delta_{deg(\xi_{I})=0} sgn 6 \xi_{1}\xi_{2}\xi_{3}\xi_{4} \otimes v_{I,k}) ]
	%\end{align*}
	\begin{align*}
	0=& \frac{d^{2}}{d \lambda^{2}}(T( 1 _{\lambda} \vec{m}))=\sum^{N}_{k=2} \sum_{L} k(k-1) (\lambda +\Theta)^{k-2} \Big[ -2 \Theta \eta_{L} \otimes v_{L,k}+ \lambda \left( \eta_{L} \otimes t. v_{L,k}-(4-|L|) \eta_{L} \otimes v_{L,k}\right) \\
	& + \lambda^{2}\sum_{i<j}(\xi_{ij} \star \eta_{L} ) \otimes \xi_{ij}. v_{L,k} - \lambda^{3}\bigchi_{|L|=0 } \eta_{1234} \otimes Cv_{L,k}\Big]\\
	&+2 \sum^{N}_{k=1} \sum_{L} k(\lambda +\Theta)^{k-1} \Big[ \eta_{L} \otimes t. v_{L,k} -(4-|L|) \eta_{L} \otimes v_{L,k}  \\
	&+2 \lambda  \sum_{i<j}(\xi_{ij} \star \eta_{L} ) \otimes \xi_{ij}. v_{L,k}  -3 \lambda^{2}\bigchi_{|L|=0 }   \eta_{1234} \otimes C v_{L,k} \Big]\\
	%& 3 \sum^{N}_{k=1} \sum_{I} k(\lambda +\Theta)^{k-1}[-2 \sum_{i<j}\xi_{i} \xi_{j} \xi_{I}  \otimes F_{i,j} v_{I,k}+ \\
	%&+6 \lambda(\delta_{deg(\xi_{I})=0} sgn 6 \xi_{1}\xi_{2}\xi_{3}\xi_{4} \otimes v_{I,k}) ]+\\
	&+\sum^{N}_{k=0} \sum_{L}(\lambda +\Theta)^{k}\Big[2 \sum_{i<j}(\xi_{ij} \star \eta_{L} ) \otimes \xi_{ij}. v_{I,k}-6 \lambda \bigchi_{|L|=0 } \eta_{1234} \otimes Cv_{L,k} \Big].
	\end{align*}
	If we expand this expression with respect to the variables $\lambda$ and $\lambda+\Theta$, the coefficients of $(\lambda +\Theta)^{s} \lambda^{3}$, with $s \geq 0$, are:
	\[
	(s+2)(s+1)   \eta_{1234} \otimes Cv_{\emptyset,s+2} =0.
	\]
	and therefore 
	\begin{equation}\label{primopezzo}
	v_{\emptyset,k}=0
	\end{equation} for all $k\geq 2$. If we consider the coefficients of $(\lambda +\Theta)^{s} \lambda^{2}$ with $s \geq 1$ we obtain:
	\begin{align*}
	\sum _{L}  \sum_{i<j}(s+2)(s+1)(\xi_{ij} \star \eta_{L}  )\otimes \xi_{ij}. v_{L,s+2}-6 (s+1)  \eta_{1234} \otimes Cv_{\emptyset,s+1}=0.
	\end{align*}
	Therefore we obtain that for $s \geq 1$:
	\begin{align}
	\label{secpezzo}
	\sum _{L}  \sum_{i<j}(\xi_{ij} \star \eta_{L} ) \otimes \xi_{ij}. v_{L,s+2}=0.
	\end{align}
	Now we look at the coefficients of $(\lambda +\Theta)^{s} \lambda$ with $s \geq 2$ and obtain:
	\begin{align*}
	\sum _{L} (s+2)(s+1) (2 \eta_{L} &\otimes v_{L,s+2}+ \eta_{L} \otimes t. v_{L,s+2}-(4-|L|) \eta_{L} \otimes v_{L,s+2})\\
	&+4\sum _{L}  \sum_{i<j}(s+1)(\xi_{ij} \star \eta_{L}  )\otimes \xi_{ij}. v_{L,s+1}-6  \eta_{1234} \otimes Cv_{\emptyset,s}=0.
	\end{align*}
	Therefore, using \eqref{primopezzo} and \eqref{secpezzo}, we obtain that for $s \geq 2$:
	\begin{align}
	\label{terpezzo} 
	\sum _{L}  ((|L|-2) \eta_{L} \otimes v_{L,s+2}+ \eta_{L} \otimes t. v_{L,s+2})=0.
	\end{align}
	Finally we look at the coefficients of $(\lambda +\Theta)^{s}$ with $s \geq 3$ and obtain:
	\begin{align*}
	&\sum _{L}  (s+1)s(-2 \eta_{L} \otimes v_{L,s+1})+2(s+1)( \eta_{L} \otimes t. v_{L,s+1}-(4-|L|) \eta_{L} \otimes v_{L,s+1}) \\
	&+2\sum _{L}  \sum_{i<j}(\xi_{ij} \star \eta_{L})  \otimes \xi_{ij}. v_{L,s}=0.
	\end{align*}
	This equation together with \eqref{secpezzo} and \eqref{terpezzo} immediately implies $v_{L,k}=0$ for all $k \geq 4$.
\end{proof}
By Lemma \ref{411}, for a singular vector $\vec{m} \in \Ind(F)$, $T(\vec{m}) $ has the following form:
\begin{align}
\label{vettsindeta2}
T(\vec{m}) =  \Theta^{3} \sum_{L}\eta_{L} \otimes v_{L,3}+\Theta^{2} \sum_{L}\eta_{L} \otimes v_{L,2}+\Theta \sum_{L}\eta_{L} \otimes v_{L,1}+ \sum_{L}\eta_{L} \otimes v_{L,0}.
\end{align}
We write the $\lambda-$action in the following way, using Proposition \ref{actiondual} and Lemma \ref{actiontheta}
\begin{align}
\label{lambdaactionabcd}
&T ({\xi_{I}}_{\lambda}\vec{m})\\ \nonumber
&=b_{0}(I)+G_{1}(I)+\lambda \big[B_{0}(I)-a_{0}(I)-G_{2}(I)\big]+\lambda^{2}\big[C_{0}(I)+G_{3}(I) \big]+\lambda^{3}D_{0}(I)\\ \nonumber
&+(\lambda+\Theta)\big( \big[a_{0}(I)+b_{1}(I)+2G_{2}(I)\big]+\lambda[ B_{1}(I)-a_{1}(I)-3G_{3}(I)\big]+\lambda^2 C_1(I)+\lambda^3D_1(I)\big)\\ \nonumber
&+(\lambda+\Theta)^{2}\big(\big[a_{1}(I)+b_{2}(I)+3G_{3}(I)\big]+\lambda \big[ B_{2}(I)-a_{2}(I)\big]+\lambda^2 C_2(I)+\lambda^3 D_2(C)\big)\\ \nonumber
&+(\lambda+\Theta)^{3}\big(\big[a_{2}(I)+b_{3}(I)\big]+ \lambda \big[B_{3}(I)-a_{3}(I)\big]+\lambda^{2} C_{3}(I) +\lambda^{3} D_{3}(I) \big) \\
& \nonumber  +   (\lambda+\Theta)^{4} a_{3}(I)
\end{align}
where the coefficients $a_{p}(I), \, b_{p}(I), \, B_{p}(I),\,C_{p}(I),\,D_{p}(I), \, G_{p}(I)$ depend on $I$ for every $0 \leq p \leq 3$. Here is their explicit expression:
\begin{align*}
a_{p}(I)&=\sum_{L} (-1)^{(|I|(|I|+1)/2)+|I||L|}  (|I|-2)  (\xi_I \star \eta_{L}) \otimes v_{L,p};\\
b_{p}(I)&=\sum_{L} (-1)^{(|I|(|I|+1)/2)+|I||L|} \bigg[ -(-1)^{|I|} \sum^{4}_{i=1}(\partial_{i}\xi_I\star \partial_{i}\eta_{L}) \otimes  v_{L,p} +\sum_{r<s}  (\partial_{rs}\xi_I\star \eta_{L}) \otimes \xi_{r,s}. v_{L,p})   \\
&\quad  +\bigchi_{|I|=3} \, \varepsilon_{I} (\xi_{I^{c}} \star \eta_{L} ) \otimes C  v_{L,p}  \bigg];\\
B_{p}(I)&=\sum_{L} (-1)^{(|I|(|I|+1)/2)+|I||L|} \bigg[(\xi_I \star \eta_{L} ) \otimes t. v_{L,p}-(-1)^{|I|}\sum^{4}_{i=1} \partial_{i}(\xi_{Ii} \star \eta_{L} ) \otimes  v_{L,p}\\
&\quad + (-1)^{|I|} \sum _{i \neq j} (\partial_{i}\xi_{Ij} \star \eta_{L}) \otimes \xi_{ji}.  v_{L,p}   +\bigchi_{|I|=2} \varepsilon_{I}   (\xi_{I^{c}} \star \eta_{L} )\otimes C  v_{L,p}) \bigg];\\
C_{p}(I)&=\sum_{L} (-1)^{(|I|(|I|+1)/2)+|I||L|} \bigg[ \sum _{i < j} (\xi_{Iij} \star \eta_{L} ) \otimes   \xi_{ij}.  v_{L,p}  -\bigchi_{|I|=1}  \varepsilon_{I}  (\xi_{I^{c}} \star \eta_{L} ) \otimes C v_{L,p}) \bigg];\\
D_{p}(I)&= \sum_{L} (-1)^{(|I|(|I|+1)/2)+|I||L|} \bigg[-\bigchi_{|I|=0}     \,( \xi_{1234} \star \eta_{L} )\otimes C v_{L,p}  \bigg];\\
G_{p}(I)&=-\sum_{L}\bigchi_{|I|=4} \varepsilon_{I} \eta_{L} \otimes C v_{L,p}.
\end{align*}
We will write $a_{p}$ instead of $a_{p}(I)$ if there is no risk of confusion, and similarly for the others.

\begin{prop}
	\label{azioneabcd}
	Let $\vec{m} \in \Ind(F)$ be such that $T(\vec{m}) $ is written  as in formula \eqref{vettsindeta2}. Using notation \eqref{lambdaactionabcd}, we have that:
	\begin{enumerate}
		\item condition \textbf{S1} implies that for all $I\in \mathcal I_{<}$ we have
		\begin{align*}
		D_{3}&=D_{2}=C_{3}=D_{1}+a_{3}=C_{2}-3a_{3}=B_{3} +2a_{3}\\&=C_{1}+2B_{2}+a_{2}+3b_{3}=D_{0} +C_{1}+B_{2}+b_{3} =C_{0}+B_{1}+b_{2}+G_{3}=0;
		\end{align*}
		\item condition \textbf{S2} implies that for all $I\in \mathcal I_<$ such that $|I|\geq 1$ we have
		\begin{align*}
		B_{0}+b_{1}+G_{2}=B_{1}+a_{1}+2b_{2}+3 G_{3}=2a_{2}+B_{2}+3b_{3}=3a_{3}+B_{3}=0;
		\end{align*}
		\item condition \textbf{S3} implies that for all $I \in \mathcal I_<$ such that $|I|\geq 3$ we have
		\begin{align*}
		b_{0} +G_{1}=a_{0}+b_{1}+2G_{2}= a_{1}+b_{2}+3G_{3}= a_{2}+b_{3}=a_{3}=0.
		\end{align*}		
	\end{enumerate}
\end{prop}
\begin{proof}
	We compute $\frac{d^{2}}{d \lambda^{2}}(T( {\xi_I}_\lambda \vec{m}))$ and $\frac{d}{d \lambda}(T ({\xi_I}_\lambda \vec{m}))$ using \eqref{lambdaactionabcd}. We have
	\begin{align*}
	&\frac{d}{d \lambda}(T({\xi_I}_\lambda \vec{m}))=B_{0}+b_{1}+G_{2}+\lambda \big[2 C_{0}+B_{1}-a_{1}-G_{3}\big]+\lambda^{2} \big[ 3D_{0}+C_{1}\big]+ \lambda^{3}D_{1}\\
	&+(\lambda+\Theta)\big(\big[B_{1}+a_{1}+2b_{2}+3G_{3}\big]+ \lambda\big[ 2C_{1}+2B_{2}-2a_{2}\big]+\lambda^{2}\big[3D_{1}+2C_{2}\big]+2\lambda^{3}D_{2} \big)\\
	&+(\lambda+\Theta)^{2}\big(\big[2a_{2}+B_{2}+3b_{3}\big] +\lambda\big[3 B_{3}-3a_{3} +2C_{2}\big]+\lambda^{2}\big[3D_{2}+3C_{3}\big]+3\lambda^{3}D_{3}\big)\\
	&+(\lambda+\Theta)^{3}\big(\big[3a_{3}+B_{3}\big]+2\lambda C_{3}+3\lambda^{2} D_{3}\big),
	\end{align*}
	and
	\begin{align*}
	&\frac{d^{2}}{d \lambda^{2}}(T({\xi_I}_\lambda \vec{m}))\\
	&=2  C_{0}+2B_{1}+2b_{2}+2G_{3}+ \lambda \big[6D_{0}+4C_{1}+2B_{2}-2a_{2} \big] +\lambda^{2} \big[6D_{1}+2C_{2}\big] +2\lambda^{3} D_{2}\\
	&\quad +(\lambda+\Theta)\big(\big[2C_{1}+4B_{2}+2a_{2}+6b_{3}\big]+\lambda\big[6D_{1}+8C_{2}+6B_{3}-6a_{3}\big]+\lambda^{2}\big[12D_{2}+6C_{3}\big]+6\lambda^{3}D_{3}\big)\\
	&\quad +(\lambda+\Theta)^{2}\big(\big[2C_{2}+6a_{3}+6B_{3}\big]+\lambda\big[12C_{3}+6D_{2}\big]+18\lambda^{2}D_{3}\big)\\
	&\quad +(\lambda+\Theta)^{3}\big(2C_{3}+6\lambda D_{3}\big).
	\end{align*}
	
The result follows.	
\end{proof}
Let us show some other reductions on singular vectors.

\begin{lem}
	\label{lemmav3}
	Let $\vec{m} \in \Ind(F)$ be a singular vector, such that $T(\vec{m}) $ is written  as in formula \eqref{vettsindeta2}.
	For all $I $ we have that $v_{I,3}=0$.
\end{lem}
\begin{proof}
	By Proposition \ref{azioneabcd}, we have $2a_{3}(i)+B_{3}(i)=0$ and $3a_{3}(i)+B_{3}(i)=0$ for all $i\in \{1,2,3,4\}$. Therefore $a_{3}(i)=0$ which immediately implies $v_{L,3}=0$ for every $L$ such that $|L|<4$. 
	
	Proposition \ref{azioneabcd} also provides  $D_{0}(1) +C_{1}(1)+B_{2}(1)+b_{3}(1) =0 $, $C_{1}(1)+2B_{2}(1)+  a_{2}(1)+3b_{3}(1)=0$ and $2a_{2}(1)+B_{2}(1)+3b_{3}(1)=0 $. 
	A linear combination of these equations gives us $D_{0}(1) +a_{2}(1)+b_{3}(1)=0$. Since $D_{0}(1)=0$, we have
	\begin{align*}
	0=a_{2}(1)+b_{3}(1)=-\sum_{L}(-1)^{1+|L|}(\xi_{1} \star \eta_{L} )\otimes v_{L,2}- \eta_{234} \otimes v_{1234,3}=0,
	\end{align*}
	which implies $v_{1234,3}=0$.
\end{proof}
Lemma \ref{lemmav3} implies $a_3=b_3=B_3=C_3=D_3=G_3=0$ and so all equations in Proposition \ref{azioneabcd} can be significantly simplified. Next result provides a further semplification.
\begin{lem}
	\label{lemmav2}
	Let $\vec{m} \in \Ind(F)$ be a singular vector such that $T(\vec{m}) $ is written as in formula \eqref{vettsindeta2}. For all $I $ we have that $v_{I,2}=0$.
\end{lem}
\begin{proof}
	By Proposition \ref{azioneabcd} and Lemma \ref{lemmav3} we know that $D_{0}(i) +C_{1}(i)+B_{2}(i) =0$, $C_{1}(i)+2B_{2}(i)+  a_{2}(i)=0$ and $2a_{2}(i)+B_{2}(i)=0 $  for all $i\in \{1,2,3,4\}$. Moreover $D_{0}(i)=0$ by definition and from these equations we can deduce $a_2(i)=0$ for all $i\in \{1,2,3,4\}$ and, as in the proof of Lemma \ref{lemmav3} we can immediately conclude that $v_{L,2}=0$ for every $L$ such that $|L|<4$.
	
	We now show that $v_{1234,2}=0$. By Proposition \ref{azioneabcd} we know that $b_{0}(123)+G_1(123)=0$ and since $G_1(123)=0$ by definition we have 
	\begin{align*}
	0&=b_{0}(123)\\
	&= \sum_{L}(-1)^{|L|}\Big( \sum^{4}_{i=1}(\partial_{i}\xi_{123}\star \partial_{i}\eta_{L} )\otimes v_{L,0} + \sum_{r<s} (\partial_{rs}\xi_{123} \star \eta_{L}) \otimes \xi_{rs}.v_{L,0}+ ( \xi_{4} \star \eta_{L})\otimes C v _{L,0}\Big).
	\end{align*}
	In this equation the unique term in $\eta_4$ is
	\begin{align*}
	\eta_{4} \otimes C v _{\emptyset,0} 
	\end{align*}
	and so $Cv _{\emptyset,0} =0$.
	%and so $ v _{\emptyset,0} =0$ if $C \neq 0$. Nevertheless, if $C=0$ the $\lambda-$action in Proposition \ref{actiondual} reduces to the $\lambda-$action found in Theorem 4.3 of \cite{kac1} and one can obtain $v_{\emptyset,0}=0$ as in Lemma B.4 of \cite{kac1}. 
	
	By Proposition \ref{azioneabcd} we have that $C_{0}(1)+B_{1}(1)+b_{2}(1)=0$, $B_{1}(1)+a_{1}(1)+2b_{2}(1)=0$  and so $C_{0}(1)-a_{1}(1)-b_{2}(1)=0$. We have:
	\begin{align*}
	0&=C_{0}(1)-a_{1}(1)-b_{2}(1)\\
	&=\sum_{L}\sum _{i < j} (-1)^{1+|L|}(\xi_{1ij}   \star \eta_{L}) \otimes   \xi_{ij}. v_{L,0}  - \sum_{L} (-1)^{1+|L|}(\xi_{234} \star \eta_{L} )\otimes Cv_{L,0} \\
	&\quad+\sum_{L} (-1)^{1+|L|} (\xi_{1} \star \eta_{L} )\otimes v_{L,1} +\eta_{234}\otimes v_{1234,2}.
	\end{align*}
	The terms in $ \eta_{234}$ in this expression are
	\begin{align*}
	\eta_{234} \otimes Cv_{\emptyset,0}+ \eta_{234}\otimes v_{1234,2} =0.
	\end{align*}
	Since $Cv_{\emptyset,0}=0$, we conclude that $v_{1234,2}=0$.
\end{proof}
By Lemma \ref{lemmav2} we can deduce that $a_2=b_2=B_2=C_2=D_2=G_2=0$ for all $I$. 
\begin{lem}
	\label{lemmav1}
	Let $\vec{m} \in \Ind(F)$ be a singular vector such that $T(\vec{m}) $ is written as in formula \eqref{vettsindeta2}. For all $L$ such that $|L| \leq 2$, we have that $v_{L,1}=0$.
\end{lem}
\begin{proof}
	By Proposition \ref{azioneabcd} and Lemmas \ref{lemmav3} and \ref{lemmav2}, we have $a_{1}(I)=0$ for all $|I| \geq 3$ which immediately implies $v_{L,1}=0$ for all $L$ such that $|L| \leq 1$.
	
	Let's show the result for $|L| =2$.
	By Proposition \ref{azioneabcd} and Lemmas \ref{lemmav3} and \ref{lemmav2}, we know that $B_{0}(a)+b_{1}(a)=0$ for all $a\in \{1,2,3,4\}$. Letting $(a)^{c}=(b,c,d)$ we have
	\begin{align*}
	0&= \sum_{L }(-1)^{1+|L|}\Big((\xi_{a} \star \eta_{L} ) \otimes t.v_{L,0}+\sum_{i\neq a} \partial_{i}(\xi_{ai} \star \eta_{L} ) \otimes v_{L,0} - \sum _{j\neq a}(\xi_{j} \star \eta_{L}  )\otimes \xi_{j,a}. v_{L,0}\Big)\\
	&\quad + \sum _{|L| \geq 2 }(-1)^{1+|L|}\partial_{a}\eta_{L} \otimes v_{L,1}.
	\end{align*}
	
	The terms in $\eta_{d}$ of $B_{0}(a)$ are:
	\begin{align*}
	\eta_{d}  \otimes \xi_{a,d}. v_{\emptyset,0}. 
	\end{align*}
	We have shown in Lemma \ref{lemmav2} that $Cv_{\emptyset,0}=0$ and so $ v _{\emptyset,0} =0$ if $C \neq 0$.
	 Nevertheless, if $C=0$ the $\lambda-$action in Proposition \ref{actiondual} reduces to the $\lambda-$action found in Theorem 4.3 of \cite{kac1} and one can obtain even in this case $v_{\emptyset,0}=0$ proceeding as in Lemma B.4 of \cite{kac1}. 
	 
	 Therefore, the unique term in $\eta_d$ in the equation above is 
	\begin{align*}
	\partial_{a}\eta_{L_0} \otimes v_{L_0,1} =0,
	\end{align*}
	where $L_0=ad$ if $a<d$ and $L_0=da$ if $a>d$, and this implies $v_{L_0,1}=0$.
\end{proof}
By Lemmas \ref{lemmav3}, \ref{lemmav2} and \ref{lemmav1} and the fact $v_{\emptyset,0}=0$, for a singular vector $\vec{m}$ the expression in \eqref{vettsingeta} can be simplified as
\begin{align}
\label{vettsingeta3}
T(\vec{m}) =  \Theta\sum_{|L|\geq 3}\eta_{L} \otimes v_{L,1}+ \sum_{|L| \geq 1}\eta_{L} \otimes v_{L,0}.
\end{align}

Therefore, from \eqref{vettsingeta3}, we have that there can only be singular vectors of degree $3, \, 2 $ and $1$. Hence we have showed Theorem \ref{greaterthan3}. Following the notation used in \cite{kac1}, we rewrite \eqref{vettsingeta3} in the following way: for $|L|=3$, $\eta_{L}$ will be written as $\eta_{(i)^{c}}$, where $(i)^{c}=L$, $v_{L,1}$ will be renamed as $v_{i,1}$ and $v_{L,0}$ will be renamed as $v_{i}$, so that they depend on one index; for $|L|=2$, $\eta_{L}$ will be written as $\eta_{(i,j)^{c}}$, where $(i,j)^{c}=L$, and $v_{L,0}$ will be renamed as $v_{i,j}$.
In particular, by \eqref{vettsingeta3}, the singular vectors of degree $3, \, 2 $ and $1$ are such that
\begin{description}
	\item[degree 3] $T(\vec{m})=\Theta \sum_{i}\eta_{(i)^{c}} \otimes v_{i,1}+ \sum_{i}\eta_{i} \otimes v_{i,0}$,
	\item[degree 2] $T(\vec{m})=\Theta \eta_{*} \otimes v_{*}+ \sum_{i<j}\eta_{(i,j)^{c}} \otimes v_{i,j} $,
	\item[degree 1] $T(\vec{m})=\sum_{i}\eta_{(i)^{c}} \otimes v_{i} $.
\end{description}
By Proposition \ref{azioneabcd} and Lemmas \ref{lemmav3}, \ref{lemmav2}, \ref{lemmav1} we obtain the following result.
\begin{prop}
	\label{azioneabcdfinal}
	Let $\vec{m} \in \Ind(F)$ be such that $T(\vec{m})$ is as in formula \eqref{vettsingeta3}. Using notation \eqref{lambdaactionabcd}, we have that:
	\begin{enumerate}
		\item condition \textbf{S1} implies that for all $ {I}\in \mathcal I_{\neq} $
		\begin{align*} 
		C_{1}=D_{1}=D_{0} =C_{0}+B_{1}=0;
		\end{align*} 
		\item condition \textbf{S2} implies that for all ${I} \in \mathcal I_{\neq} $ with $ |I| \geq 1$
		\begin{align*}
		B_{0}+b_{1}=B_{1}+a_{1}=0;
		\end{align*}
		\item condition \textbf{S3} implies that for all ${I} \in \mathcal I_{\neq} $ with $ |I| \geq 3$
		\begin{align*}
		 b_{0}+G_{1}= a_{0}+b_{1}=a_{1}=0.
		\end{align*}
	\end{enumerate}
\end{prop}
\section{Singular vectors of degree $2$}
The aim of this section is to classify all singular vectors of degree $2$.
We have that a singular vector of degree $2$ is such that:
\begin{align}
\label{vettsingetagrado2}
T(\vec{m}) =\Theta \eta_{*} \otimes v_{*}+ \sum_{i<j}\eta_{(i,j)^{c}} \otimes v_{i,j}.  
\end{align} 
We will assume for our convenience that $v_{i,i}=0$ and $ v_{i,j}=- v_{j,i}$ for all $i,j$. We write the vector $\vec{m}$ also in the following way:
\begin{align} 
\label{vettsingetagrado2originale}
\vec{m}=&(\eta_{2}+i\eta_{1})(\eta_{4}+i\eta_{3}) \otimes w_{1}+(\eta_{2}+i\eta_{1})(\eta_{4}-i\eta_{3}) \otimes w_{2}+(\eta_{2}-i\eta_{1})(\eta_{4}+i\eta_{3}) \otimes w_{3}\\ \nonumber
&+(\eta_{2}-i\eta_{1})(\eta_{4}-i\eta_{3})\otimes w_{4}+(\eta_{2}+i\eta_{1})(\eta_{2}-i\eta_{1}) \otimes w_{5}+(\eta_{4}+i\eta_{3})(\eta_{4}-i\eta_{3}) \otimes w_{6}+\Theta \otimes w_{7}\\ \nonumber
=&(-\eta_{13}+i\eta_{14}+i\eta_{23}+\eta_{24})\otimes w_{1}+(\eta_{13}+i\eta_{14}-i\eta_{23}+\eta_{24})\otimes w_{2}+(\eta_{13}-i\eta_{14}+i\eta_{23}+\eta_{24})\otimes w_{3}\\ \nonumber
&+(-\eta_{13}-i\eta_{14}-i\eta_{23}+\eta_{24})\otimes w_{4}+(2 \Theta +2i\eta_{12})\otimes w_{5}+(2 \Theta +2i\eta_{34})\otimes w_{6}+\Theta \otimes w_{7}.
\end{align}
From these two expressions it follows that
\begin{align}
\label{notazionevijw}
&v_{1,2}=2i w_{5},\\ \nonumber
&v_{1,3}= w_{1}- w_{2}- w_{3}+ w_{4},\\ \nonumber
&v_{1,4}=iw_{1}+iw_{2}-iw_{3}-iw_{4},\\ \nonumber
&v_{2,3}=iw_{1}-iw_{2}+iw_{3}-iw_{4},\\ \nonumber
&v_{2,4}=-w_{1}-w_{2}-w_{3}-w_{4},\\ \nonumber
&v_{3,4}=2i w_{6},\\ \nonumber
&v_{*}=2w_{5}+2w_{6}+w_{7}.
\end{align}
Indeed, let us show for example one of the previous equations.
In \eqref{vettsingetagrado2}, let us consider $\eta_{(1,3)^{c}}\otimes  v_{1,3}=\eta_{24} \otimes v_{1,3}$ . We have that $T(\eta_{13})=-\eta_{24}$ . In \eqref{vettsingetagrado2originale}, the terms in $\eta_{13}$ are:
\begin{align*}
-\eta_{13}\otimes w_{1}+\eta_{13}\otimes w_{2}+\eta_{13}\otimes w_{3}-\eta_{13}\otimes w_{4},
\end{align*}
therefore $v_{1,3}=w_{1}- w_{2}- w_{3}+ w_{4}$.\\
In the following lemma we write explicitly the relations of Proposition \ref{azioneabcdfinal} for a vector as in formula \eqref{vettsingetagrado2}.
\begin{lem}
	\label{s1s2s3grado2}
	Let $\vec{m} \in \Ind(F)$ be such that $T(\vec{m})$ is as in formula \eqref{vettsingetagrado2}. We have that:\\
	\textbf{1)} condition \textbf{S1} implies (for $I=\emptyset$)
	\begin{align}
	\label{C0+B1} 
	 \sum _{i < j} (  \xi_{ij} \star \eta_{(i,j)^{c}} ) \otimes   \xi_{ij}. v_{i,j}+\eta_{*}  \otimes t. v_{*}=0;
	\end{align}
	\textbf{2)} condition \textbf{S2} implies that for all $I$ such that $|I| = 1,2$
	\begin{align}
	\label{B0+b1}
	&\sum_{i<j}  \bigg[(\xi_I \star \eta_{(i,j)^{c}} ) \otimes t.v_{i,j}-(-1)^{|I|}\sum^{4}_{l=1} \partial_{l}(\xi_{Il}  \star \eta_{(i,j)^{c}} ) \otimes  v_{i,j}+ (-1)^{|I|} \sum _{k \neq l} (\partial_{k}\xi_{Il} \star \eta_{(i,j)^{c}}) \otimes \xi_{lk} . v_{i,j} \\ \nonumber 
	&  +\bigchi_{|I|=2}    \, \varepsilon_{I} ( \xi_{I^{c}} \star \eta_{(i,j)^{c}}) \otimes C  v_{i,j} \bigg] -(-1)^{|I|} \sum^{4}_{i=1}(\partial_{i}\xi_I\star \partial_{i}\eta_{*}) \otimes  v_{*}+\sum_{r<s}  (\partial_{rs}f \star \eta_{*} )\otimes \xi_{rs} .v_{*} =0;
	\end{align}
	\textbf{3)} condition \textbf{S3} implies that for all $I$ such that $|I|\geq 3$ 
		%For $f \in B_{\so(4)}$:
	%\begin{align}
	%\label{b1=0}
	%\sum_{r<s}  (\partial_{rs}f) \star \eta_{*} \otimes \xi_{rs}. v_{*}=0.
	%\end{align}
	%For all $I$ such that $|I|\geq 3$% or $f \in B_{\so(4)}$:
	\begin{align}
	\label{b0+G1=0}
	&\sum_{i<j} (-1)^{(|I|(|I|+1)/2)} \bigg[ -(-1)^{|I|} \sum^{4}_{l=1}(\partial_{l}\xi_I\star \partial_{l}\eta_{(i,j)^{c}}) \otimes  v_{i,j} +\sum_{r<s}  (\partial_{rs}\xi_I \star \eta_{(i,j)^{c}}) \otimes \xi_{rs}. v_{i,j}   \\ \nonumber
	&  +\bigchi_{|I|=3} \, \varepsilon_{I} (\xi_{I^{c}} \star \eta_{(i,j)^{c}})  \otimes C  v_{i,j}  \bigg]-\bigchi_{|I|=4} \varepsilon_{I} \eta_{*} \otimes C v_{*}=0.
	\end{align}
\end{lem}

The following result collects the crucial equations that we will use in the classification of singular vectors of degree 2.
\begin{lem}
	\label{s1s2s3grado2vstardiversodazero}
	Let $\vec{m} \in \Ind(F)$ be as in \eqref{vettsingetagrado2}. Then for any permutation $(a,b,c,d)$ of $\{1,2,3,4\}$ we have

	\begin{align}
	\label{B.22B}  &\sum_{j \neq a}(-1)^{a+j}\xi_{ja}.v_{j,a}=v_*;\\
	\label{B.23B}  &t.v_{a,b}-v_{a,b}+\sum_{j\neq a,b}(-1)^{a+j}\xi_{aj}.v_{j,b}=0;\\
	\label{B.25B}  &\xi_{ab}.v_{*}+(-1)^{a+b}t.v_{a,b}+\sum_{ j\neq a,b}(-1)^{b+j}\xi_{aj}.v_{j,b}-\sum_{j \neq a,b}(-1)^{a+j}\xi_{bj}.v_{j,a}-\varepsilon_{(a,b)}\varepsilon_{(c,d)}Cv_{c,d}=0;\\
	\label{B.26B}  &\sum_{i<j}(-1)^{i+j}\xi_{ij}.v_{i,j}=t.v_*;\\
	\label{B.27B}  &(-1)^{b+c}v_{b,c}-(-1)^{a+c}\xi_{ab}.v_{a,c}+(-1)^{a+b}\xi_{ac}.v_{a,b}
	                 +\varepsilon_{(a,b,c)}(-1)^{a+d}Cv_{a,d}=0;\\
	\label{B.28B}  &\sum_{i<j} \xi_{(ij)^{c}} . v_{i,j}+  C v_{*}=0.
	\end{align}
	
	Finally:
	\begin{align}
	\label{s1s2s3grado2vstarzeroalfabetaconv*}
	&e_1.v_*=0&& e_2.v_*=0\\ \nonumber
	&e_1.v_{1,2}=-iv_{1,3}+v_{2,3},    &&e_2.v_{1,2}=-v_{1,4}-iv_{2,4},\\ \nonumber
	&e_1.v_{1,3}=iv_{1,2},    &&e_2.v_{1,3}=-iv_{3,4},\\ \nonumber
	&e_1.v_{1,4}= -v_{3,4},   &&e_2.v_{1,4}=v_{1,2},\\ \nonumber
	&e_1.v_{2,3}=-v_{1,2},    &&e_2.v_{2,3}=v_{3,4},\\ \nonumber
	&e_1.v_{2,4}=-iv_{3,4},   &&e_2.v_{2,4}=iv_{1,2},\\ \nonumber
	&e_1.v_{3,4}= v_{1,4} +iv_{2,4},  &&e_2.v_{3,4}=iv_{1,3}-v_{2,3},
	\end{align}
	where $e_1$ and $e_2$ are defined by \eqref{notazionealfa} and \eqref{notazionebeta}.

\end{lem}
\begin{proof}
	We will repeatedly use Lemma \ref{s1s2s3grado2}.
	
	$\bullet$ Equation \eqref{B.22B}.
	%$(B_{0}+b_{1})(\xi_{a})=0$
	We consider Equation \eqref{B0+b1} with $I=a$:
	\begin{align}
	\label{B0+b1appoggio}
	&\sum_{i<j}  \big[(\xi_{a} \star \eta_{(i,j)^{c}} ) \otimes t.v_{i,j}+\sum^{4}_{l=1} \partial_{l}(\xi_{al}\star \eta_{(i,j)^{c}} ) \otimes  v_{i,j}+ \sum _{l\neq a} (\xi_{l} \star \eta_{(i,j)^{c}}) \otimes \xi_{al}.  v_{i,j}\big]  +  \partial_{a}\eta_{*} \otimes  v_{*} =0,
	\end{align}
	%Equation \eqref{B0+b1} for all $f=\xi_{a}$ with $a \in \left\{1,2,3,4\right\}$ does not involve terms in $C$, so we obtain the following equations that are the same as in \cite{kac1}.
	and, considering the terms in $\eta_{(a)^{c}}$, we obtain:
	\begin{align*}
	0=& \sum _{l<a}( \xi_{l} \star \eta_{(l,a)^{c}} )\otimes \xi_{al}.  v_{l,a}+ \sum _{l>a} (\xi_{l} \star \eta_{(a,l)^{c}}) \otimes \xi_{al} . v_{a,l}  +  \partial_{a}\eta_{*} \otimes  v_{*} \\
	=&\sum _{l \neq a} (-1)^{l} \eta_{(a)^{c}} \otimes \xi_{l,a} .v_{l,a}+   (-1)^{a-1}\eta_{(a)^{c}} \otimes  v_{*}.
	\end{align*}
	and Equation \eqref{B.22B} follows.
	
	$\bullet$ Equation \eqref{B.23B}. Consider the terms in $\eta_{(b)^{c}}$ in Equation \eqref{B0+b1appoggio}:
	\begin{align*}
	0&=\eta_{a} \eta_{(a,b)^{c}}  \otimes t. v_{a,b}-\eta_{a} \eta_{(a,b)^{c}}  \otimes  v_{a,b} - \sum _{a \neq l,\, l<b}( \xi_{l} \star \eta_{(l,b)^{c}} )\otimes \xi_{la} . v_{l,b}- \sum _{a \neq l,\, l>b} (\xi_{l} \star \eta_{(b,l)^{c}}) \otimes \xi_{la} . v_{b,l}\\
	&=(-1)^{a-1}\eta_{(b)^{c}}  \otimes t. v_{a,b}-(-1)^{a-1} \eta_{(b)^{c}}  \otimes  v_{a,b} - \sum _{a \neq l,\, l<b} (-1)^{l-1} \eta_{(b)^{c}} \otimes \xi_{la}. v_{l,b}\\
	&\quad - \sum _{a \neq l,\, l>b} (-1)^{l} \eta_{(b)^{c}} \otimes \xi_{la}.  v_{b,l}\\
	&=(-1)^{a-1}\eta_{(b)^{c}}  \otimes t.v_{a,b}-(-1)^{a-1} \eta_{(b)^{c}}  \otimes  v_{a,b} + \sum _{l \neq a,b} (-1)^{l} \eta_{(b)^{c}} \otimes \xi_{la} . v_{l,b},
	\end{align*}
	and Equation \eqref{B.23B} follows.
	
	$\bullet$ Equation \eqref{B.25B}.
	We consider Equation \eqref{B0+b1} with $I=ab$ and we assume $c<d$ with no loss of generality. We have
	\begin{align*}
	0=&-\eta_{*} \otimes \xi_{ab} .v_{*}-(-1)^{a+b}\eta_{*} \otimes t.v_{a,b}\\
	&+\sum_{i<j} \sum_{l \neq k}(\partial_{l} \xi_{abk} \star \eta_{(i,j)^{c}})\otimes \xi_{kl}.v_{i,j}+\varepsilon_{(a,b)}(\xi_{cd}\star \eta_{(c,d)^{c}}) \otimes Cv_{c,d}.
	\end{align*}
	The coefficient of $-\eta_{*}$ in this expression is:
	\begin{align*}
	0=\xi_{ab}. v_{*}+(-1)^{a+b}t.v_{a,b}-\sum_{j\neq a,b}(-1)^{b+j}\xi_{ja}.v_{j,b}+\sum_{j \neq a,b}(-1)^{a+j}\xi_{jb}.v_{j,a}-\varepsilon_{(a,b)}\varepsilon_{(c,d)}Cv_{c,d}.
	\end{align*}
	
	$\bullet$ Equation \eqref{B.26B}. This follows immediately by Equation \eqref{C0+B1}.
	
	$\bullet$ Equation \eqref{B.27B}. 
	Equation \eqref{b0+G1=0} for $I=abc$ provides
	\begin{align*}
	0=&\sum_{i<j}\big(\sum^{4}_{l=1}(\partial_{l}\xi_{abc}\star \partial_{l}\eta_{(i,j)^{c}}) \otimes v_{i,j}+\sum_{r<s}(\partial_{rs}\xi_{abc} \star \eta_{(i,j)^{c}} )\otimes \xi_{rs}.v_{i,j}+\varepsilon_{(a,b,c)}(\xi_{d}\star \eta_{(i,j)^{c}})\otimes Cv_{i,j}\big).
	\end{align*}
	Considering the coefficients of $(-1)^{a}\eta_{(a)^{c}}$ we have:
	\begin{align*}
	&(-1)^{b+c}v_{b,c}-(-1)^{a+c}\xi_{ab}.v_{a,c}+(-1)^{a+b}\xi_{ac}.v_{a,b}+\varepsilon_{(a,b,c)}(-1)^{a+d}Cv_{a,d}=0.
	\end{align*}
	
	$\bullet$ Equation \eqref{B.28B}.
	Equation \eqref{b0+G1=0} for $I=1234$ is:
	\begin{align*}
	%0=&b_{0}(f)+G_{1}(f)\\
	0=&-\sum_{i<j}\sum_{l=1}^{4}(\partial_{l} \xi_{1234} \star \partial_{l}\eta_{(i,j)^{c}} )\otimes v_{i,j}+\sum_{i<j}\sum_{r<s} (\partial_{rs} \xi_{1234}\star \eta_{(i,j)^{c}})\otimes \xi_{rs}. v_{i,j}- \eta_{*} \otimes C v_{*}\\
	=&-\eta_{*} \otimes \sum_{i<j} \xi_{(i,j)^{c}} \otimes v_{i,j}-  \eta_{*} \otimes C v_{*}.
	\end{align*}
	
	%$b_{1}(f)=0$
	$\bullet$ Equations \eqref{s1s2s3grado2vstarzeroalfabetaconv*}. These equations are a consequence of {\bf S0}, i.e. $e_1. \vec{m}=e_2.\vec{m}=0$. Recall that $T(e_1.\vec {m})=T((-\xi_{13}+i\xi_{23}).\vec{m})=-(T({\xi_{13}}_\lambda \vec{m}))_{|\lambda=0}+i(T({\xi_{23}}_\lambda \vec{m}))_{|\lambda=0}$ and so it can be easily computed by means of Proposition \ref{actiondual}. We obtain
	
	%Equation \eqref{b1=0} for $f \in B_{\so(4)}$ is equivalent to Equation . Indeed $B_{\so(4)}=\langle %\alpha_{1,2}=F_{1,3}-iF_{2,3}, \beta_{1,2}=F_{2,4}+iF_{1,4} \rangle$ and we obtain:
	%\begin{align*}
	%b_{1}(-\xi_{1}\xi_{3}+i\xi_{2}\xi_{3})=\eta_{*} \otimes (-F_{1,3}+iF_{2,3})v_{*}, \\
	%b_{1}(-\xi_{2}\xi_{4}-i\xi_{1}\xi_{4})=\eta_{*} \otimes (-F_{2,4}-iF_{1,4})v_{*}. 
	%0=&\eta_{*} \otimes (-F_{1,3}+iF_{2,3})v_{*}, \\
	%0=&\eta_{*} \otimes (-F_{2,4}-iF_{1,4})v_{*}. 
	%\end{align*}
	%Thus this implies that $v_{*}$ must be a highest weight vector, when it is nonzero. \\
	
	%Equation \eqref{b0+G1=0} for $f =\alpha_{1,2} \in B_{\so(4)}$ is:
	\begin{align*}
	0&=T(e_1.\vec{m})\\
	&=-\eta_*\otimes e_1.v_*-(\eta_{1}-i\eta_{2})\eta_{4} \otimes v_{1,2}-\eta_{3}\eta_{4} \otimes e_1. v_{1,2}-i\eta_{3}\eta_{4}\otimes v_{1,3}-\eta_{2}\eta_{4} \otimes e_1. v_{1,3}\\
	&\quad+\eta_{1}\eta_{2}\otimes v_{1,4}-\eta_{2}\eta_{3} \otimes e_1. v_{1,4}+\eta_{3}\eta_{4} \otimes v_{2,3}-\eta_{1}\eta_{4} \otimes e_1. v_{2,3}\\
	&\quad+i\eta_{1}\eta_{2}\otimes v_{2,4}-\eta_{1}\eta_{3} \otimes e_1. v_{2,4}+(-\eta_{2}\eta_{3}-i\eta_{1}\eta_{3})\otimes v_{3,4}-\eta_{1}\eta_{2} \otimes e_1. v_{3,4}\\
	&=-\eta_*\otimes e_1.v_*+\eta_{1}\eta_{2}\otimes (v_{1,4}+iv_{2,4}-e_1.v_{3,4})+ \eta_{1}\eta_{3}\otimes (-e_1.v_{2,4}-iv_{3,4})\\
	&\quad+\eta_{1}\eta_{4}\otimes (-v_{1,2}-e_1.v_{2,3}) + \eta_{2}\eta_{3} \otimes (-e_1.v_{1,4}-v_{3,4})\\
	&\quad+\eta_{2}\eta_{4}\otimes (-e_1.v_{1,3}+iv_{1,2})+ \eta_{3}\eta_{4}\otimes (-e_1.v_{1,2}-iv_{1,3}+v_{2,3}).
	\end{align*}
	From the previous equation we obtain relations \eqref{s1s2s3grado2vstarzeroalfabetaconv*} for $e_1$.\\
	Equations \eqref{s1s2s3grado2vstarzeroalfabetaconv*} for $e_2$ are obtained similarly.
	
\end{proof}

\begin{lem}
	\label{lemmavettorigrado2senzav*}
	If $\vec{m}$ is a singular vector of degree 2 such that $T(\vec{m})$ is as in \eqref{vettsingetagrado2} then  with $v_{*}= 0$.
\end{lem}
\begin{proof}
	Let $T(\vec{m}) \in \Ind(F)$ be as in formula \eqref{vettsingetagrado2}. We show that relations of Lemma \ref{s1s2s3grado2vstardiversodazero} lead to $v_{*}=0$.
	Let $a,b \in \left\{1,2,3,4\right\}$ with $a<b$ and $(a,b)^{c}=(c,d)$.\\
	Considering Equation \eqref{B.23B} and the same equation for reversed role of $a$ and $b$ we can deduce
	\begin{align*}
	0=&-(-1)^{a+b}2t.v_{a,b}+2(-1)^{a+b}v_{a,b}-\sum_{j\neq a,b}(-1)^{b+j}\xi_{aj}.v_{j,b}+\sum_{j\neq a,b}(-1)^{a+j}\xi_{bj}.v_{j,a}. 
	%&+\sum_{a<j, j\neq b}(-1)^{a+j}F_{b,j}v_{a,j}.
	\end{align*}
	We compare this with Equation \eqref{B.25B} and obtain
	\begin{align}
	\label{B.30B}
	\xi_{ab}.v_{*}=(-1)^{a+b}t.v_{a,b}+2(-1)^{a+b+1}v_{a,b}+Cv_{c,d},
	\end{align}
	since for $a<b$ we have that $\varepsilon_{(a,b)}\varepsilon_{(c,d)}=1$.\\
	Now consider Equation  \eqref{B.25B}:
	\begin{align}
	\label{B.31B}
	0=\xi_{ab}. v_{*}+(-1)^{a+b}t.v_{a,b}+\sum_{ j\neq a,b}[(-1)^{b+j}\xi_{aj}.v_{jb}-(-1)^{a+j}\xi_{bj}.v_{j,a}]-Cv_{c,d}.
	\end{align}
	We also consider \eqref{B.27B} with $a=j$, $b=a$, $c=b$ and $d=h$ and we substitute it into \eqref{B.31B}; we obtain %\eqref{B.28B},\eqref{B.29B} 
	\begin{align}
	\label{B.32B}
	\xi_{ab}. v_{*}=&(-1)^{a+b+1}t.v_{a,b}+2(-1)^{a+b}v_{a,b}+\sum_{j \neq a,b}\varepsilon_{(j,a,b)}(-1)^{h+j}Cv_{j,h}+Cv_{c,d}\\ \nonumber  %-\sum_{a<j<b}((-1)^{j+1}Cv_{j,d})-Cv_{c,d}.
	=&(-1)^{a+b+1}t.v_{a,b}+2(-1)^{a+b}v_{a,b}+\sum_{j<a \ or \ j>b}(-1)^{j}Cv_{j,h}+\sum_{a<j<b}(-1)^{j+1}Cv_{j,h}+Cv_{c,d}.
	\end{align}
	%where every time we write $d$ we mean $d(j)=(a,b,j)^{c}$.
	Combining \eqref{B.32B} and \eqref{B.30B}, we get:
	\begin{align*}
	(-1)^{a+b}2t.v_{a,b}=&4(-1)^{a+b}v_{a,b}+\sum_{j<a\ or \ j>b}(-1)^{j}Cv_{j,h}+\sum_{a<j<b}(-1)^{j+1}Cv_{j,h}.
	\end{align*}
	Comparing this equation with \eqref{B.32B} we obtain
	\begin{align}
	2\xi_{ab}. v_{*}=&\sum_{j<a \ or \ j>b}(-1)^{j}Cv_{j,h}+\sum_{a<j<b}(-1)^{j+1}Cv_{j,h}+2Cv_{c,d}, \nonumber
	\end{align}
	which simplifies to
	\begin{align}
	\xi_{a,b}. v_{*}=0
	\end{align}
	for every $a<b$.
	This implies that, if $v_*\neq 0$, then $F=\langle v_{*} \rangle$ has dimension 1 and $\frak{so}(4)$ acts trivially on it. Moreover all the $v_{a,b}$'s are scalar multiple of $v_{*}$ since $F=\langle v_{*} \rangle$.\\
	By \eqref{B.22B} we also have $v_{*}=\sum_{j \neq a}(-1)^{a+j}\xi_{ja}.v_{j,a}$ for every $1 \leq a \leq 4$; then, since all the $v_{a,b}$'s are multiple of $v_{*}$, we have a contradiction.
\end{proof}
By Lemma \ref{lemmavettorigrado2senzav*} we know that if $\vec(m)$ is a singular vector of degree 2 $T(\vec m)$ has the following form
\begin{align}
\label{vettsingetagrado2senzav*}
T(\vec{m}) = \sum_{i<j}\eta_{(i,j)^{c}} \otimes v_{i,j} .
\end{align} 
\begin{rem}
	\label{alfabetadisegni}
	Relations \eqref{s1s2s3grado2vstarzeroalfabetaconv*}, by  Lemma \ref{lemmavettorigrado2senzav*} and notation \eqref{notazionevijw} are equivalent to the following:
	\begin{align}
	\label{s1s2s3grado2vstarzeroalfabetaw}
	&e_1.w_{1}=-w_{5}-w_{6},&&e_2.w_{1}=w_{5}+w_{6},\\ \nonumber
	&e_1.w_{2}=w_{5}-w_{6}, &&e_2.w_{2}=w_{5}-w_{6},\\  \nonumber
	&e_1.w_{3}=0, &&e_2.w_{3}=0,\\  \nonumber
	&e_1.w_{4}=0, &&e_2.w_{4}=0,\\  \nonumber
	&e_1.w_{5}=w_{3}-w_{4}, &&e_2.w_{5}=w_{3}+w_{4},\\  \nonumber
	&e_1.w_{6}=-w_{3}-w_{4}, &&e_2.w_{6}=-w_{3}+w_{4}.
	\end{align}
	We represent these relations in the following picture
		
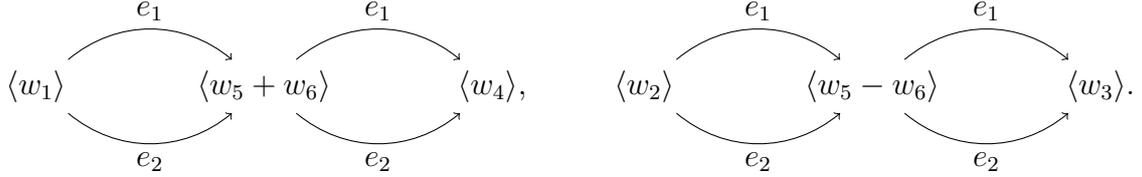
\begin{center}
		\begin{tikzpicture}
		\node at (0,0) (w1) {$\langle w_{1} \rangle$};
		\node at (3,0) (w5+w6) {$\langle w_{5}+w_{6} \rangle $};
		\node at (6,0) (w4) {$\langle w_{4} \rangle $,};
		\draw (w1) edge[out=40,in=140,->] (w5+w6);
		\node at (1.5,1) {$e_1$};
		\node at (1.5,-1) {$e_2$};
		\draw (w1) edge[out=-40,in=-140,->] (w5+w6);
		\draw (w5+w6) edge[out=40,in=140,->] (w4);
		\node at (4.5,1) {$e_1$};
		\node at (4.5,-1) {$e_2$};
		\draw (w5+w6) edge[out=-40,in=-140,->] (w4);
		\node at (8,0) (w2) {$\langle w_{2} \rangle$};
		\node at (11,0) (w5-w6) {$\langle w_{5}-w_{6} \rangle $};
		\node at (14,0) (w3) {$\langle w_{3} \rangle $.};
		\draw (w2) edge[out=40,in=140,->] (w5-w6);
		\node at (9.5,1) {$e_1$};
		\node at (9.5,-1) {$e_2$};
		\draw (w2) edge[out=-40,in=-140,->] (w5-w6);
		\draw (w5-w6) edge[out=40,in=140,->] (w3);
		\node at (12.5,1) {$e_1$};
		\node at (12.5,-1) {$e_2$};
		\draw (w5-w6) edge[out=-40,in=-140,->] (w3);
		\end{tikzpicture}
	\end{center}
	%\begin{align*}
	%\langle w_{1} \rangle   \langle w_{5}+w_{6} \rangle   \langle w_{4} \rangle
	%\end{align*}
	%\begin{align*}
	%\langle w_{2} \rangle   \langle w_{5}-w_{6} \rangle   \langle w_{3} \rangle
	%\end{align*}
	
\end{rem}
\begin{proof}[Proof of Theorem \ref{sing2}]
	Throughout this proof we let $\mu=(m,n,\mu_0,\mu_1)$ where $m,n,\mu_0,\mu_1$  denote the highest weights of $F$ with respect to $h_x,h_y,t,C$ respectively.
	We split the proof in four cases that we number by 1), 2), 3), 4).% $E_{00}$, $\mu_{1}$ (resp. $\mu_{2}$) will denote the highest weight of $F$ with respect to $H_{1}$(resp. $H_{2}$) and $m=\mu_{1}-\mu_{2}$ (resp.$\, n=\mu_{1}+\mu_{2}$) will denote the highest weight of $F$ with respect to $h_{x}$ (resp.$\, h_{y}$).
	\begin{itemize}%[leftmargin=*]
		\item[\textbf{1)}] Let $w_{5}=w_{6}=0$. 
		
		We immediately have also $w_{3}=w_{4}=0$ by \eqref{s1s2s3grado2vstarzeroalfabetaw}.		
		\begin{itemize}%[leftmargin=*]
			\item[\textbf{1a)}] Let $w_{1} \neq 0$ and $w_{2} =0$. 
			
			By \eqref{s1s2s3grado2vstarzeroalfabetaw}, we have that $w_{1}$ is a highest weight vector and, by \eqref{notazionevijw},
			\begin{align*}
			&v_{1,2}=v_{3,4}=0,\,v_{1,3}=w_{1},\,v_{1,4}=iw_{1},\,v_{2,3}=iw_{1},\,v_{2,4}=-w_{1}.
			\end{align*}
			Equation \eqref{B.23B} for $a=1,b=3$ gives $(t-i\xi_{12}-1).w_1=0$,
			Equation \eqref{B.27B} for $a=1,b=2,c=3$ gives $(C-i\xi_{12}+1).w_1=0$, and
			Equation \eqref{B.27B} for $a=3,b=1,c=4$ gives $(C-i\xi_{34}+1).w_1=0$.

			Recalling that $h_x=-i\xi_{12}+i\xi_{34} $ and $h_y=-i\xi_{12}-i\xi_{34}  $ we deduce that $\mu=(0,n,1-\frac{n}{2},-1-\frac{n}{2})$ for some $n \in \Z_{\geq 0}$.
			
			A simple verification shows that these conditions lead to the vector
						\begin{align*}
			\vec{m}_{2a}=w_{11}w_{21} \otimes y^{n}_{1},
			\end{align*}
			in $M(0,n,1-\frac{n}{2},-1-\frac{n}{2})$ for $n \in \Z_{\geq 0}$ which is indeed a singular vector.
			\item[\textbf{1b)}] Let $w_{1} = 0$ and $w_{2}\neq 0$. 
			
			By \eqref{s1s2s3grado2vstarzeroalfabetaw} we have that $w_{2}$ is a highest weight vector and, by \eqref{notazionevijw},
			\begin{align*}
			&v_{1,2}=v_{3,4}=0, \,v_{1,3}=-w_{2} ,\,v_{1,4}=iw_{2},\,v_{2,3}=-iw_{2},\,v_{2,4}=-w_{2},\,v_{3,4}=0.
			\end{align*}	
			Equation \eqref{B.23B} for $a=1,b=3$ gives $(t-i\xi_{12}-1).w_{2}=0$,
			Equation \eqref{B.27B} for $a=1,b=2,c=3$ gives $(C+i\xi_{12}-1).w_2=0$, and
			Equation \eqref{B.27B} for $a=3,b=1,c=4$ gives $(-C+i\xi_{34}+1).w_2=0$.
			
			From these conditions, recalling that $h_x=-i\xi_{12}+i\xi_{34} $ and $h_y=-i\xi_{12}-i\xi_{34}  $, we deduce that $\mu=(m,0,1-\frac{m}{2},1+\frac{m}{2})$ with $m \in \Z_{\geq 0}$ and we obtain the singular vector
			\begin{align*}
			\vec{m}_{2b}=w_{11}w_{12} \otimes x^{m}_{1},
			\end{align*}
			in $M(m,0,1-\frac{m}{2},1+\frac{m}{2})$ with $m \in \Z_{\geq 0}$.
			\item[\textbf{1c)}] Let $w_{1} \neq 0$ and $w_{2}\neq 0$. 
			
			By \eqref{s1s2s3grado2vstarzeroalfabetaw}, we have that both $w_1$ and $w_2$ are highest weight vectors of $F$, so that $w_1=\alpha w_2$ for some $\alpha\neq 0$.
			By Equation \eqref{B.27B} for $a=3,b=2,c=4$  and for $a=4, b=1, c=3$ we obtain respectively
			\[(-\alpha -1-i\xi_{34}(-\alpha +1)+C(-\alpha +1)).w_2=0\]
			and
			\[(\alpha -1-i\xi_{34}(\alpha+1)+C(\alpha +1)).w_2=0.\] The sum and the difference of these two equations show an evident contradiction.
		\end{itemize}
		\item[\textbf{2)}] Let $w_5\neq 0$ and $w_5+w_6=0$. %Let  $w_{7}=0$.
		
		  Since $e_1.w_{5}=w_{3}-w_{4}$ and $e_1.w_{6}=-w_{3}-w_{4}$ we deduce that $w_{4}=0$; we also know that $w_{2}\neq 0$ since $e_1.w_{2}=2w_{5}$. 		   		\begin{itemize}%[leftmargin=*]
			\item [\textbf{2a)}] Let $w_{1}=0$ and $w_{3} \neq 0$. By Remark \ref{alfabetadisegni} $w_3$ is a highest weight vector
			%\begin{align*}
			%&\alpha_{1,2}(w_{2})=2w_{5} \,	&\beta_{1,2}(w_{2})&=2w_{5}, \\
			%&\alpha_{1,2}(w_{3})=0 \, &\beta_{1,2}(w_{3})&=0, \\
			%&\alpha_{1,2}(w_{5})=w_{3}  \,  &\beta_{1,2}(w_{5})&=w_{3}, \\ 
			%&\alpha_{1,2}(w_{6})=-w_{3}  \,  &\beta_{1,2}(w_{6})&=-w_{3}.
			%\end{align*}
			%Therefore $w_{3}$ is a highest weight vector.
			 and Equations \eqref{notazionevijw} provide
			\begin{align*}
			&v_{1,2}=2iw_{5}, \,v_{1,3}=-w_{2}-w_{3}, v_{1,4}=iw_{2}-iw_{3},\\
			&v_{2,3}=-iw_{2}+iw_{3},\,v_{2,4}=-w_{2}-w_{3},\,v_{3,4}=-2iw_{5}.
			\end{align*}
			Let us compute the weight of $w_2$ and $w_{3}$.\\
			Equation \eqref{B.23B} for $a=1,b=3$ gives $t.(-w_{2}-w_{3})+w_{2}+w_{3}-\xi_{1,2}.(-iw_{2}+iw_{3})+\xi_{14}.(-2iw_{5})=0$, and for $a=2, b=3$ gives $t.(-iw_{2}+iw_{3})+iw_{2}-iw_{3}+\xi_{12}.(-w_{2}-w_{3})-\xi_{24}.(-2iw_{5})=0$.\\
			Recalling the definition of $e_2$ in \eqref{notazionebeta}, we deduce from these equations that $2it.w_{3}-2iw_{3}-2\xi_{12}.w_{3}-2ie_2.w_{5}=0$
			that is equivalent to
			\[(t+i\xi_{12}-2).w_{3}=0.\]
			Equation	\eqref{B.27B} for $a=1,b=2,c=4$ gives $-w_{2}-w_{3}+\xi_{12}.(iw_{2}-iw_{3})-\xi_{14}.2iw_{5}-C(-w_{2}-w_{3})=0$, and for $a=2,b=1,c=4$ gives $-iw_{2}+iw_{3}+\xi_{12}.(-w_{2}-w_{3})+\xi_{24}.(2iw_{5})-C(-iw_{2}+iw_{3})=0$.\\
			By these equations we obtain $-2w_{3}-2i\xi_{12}.w_{3}+2e_2.w_{5}+2Cw_{3}=0$ that is equivalent to
			\[(-i\xi_{1,2}+C).w_{3}=0.\]		
			Equation \eqref{B.27B} for $a=3,b=1,c=4$ gives $-i(w_{2}-w_{3})+2i\xi_{13}.w_{5}+\xi_{34}.(w_{2}+w_{3})+iC(w_{2}-w_{3})=0$, and for $a=3,b=2,c=4$ gives $-w_{2}-w_{3}-\xi_{2,3}.(-2iw_{5})+\xi_{3,4}.(-iw_{2}+iw_{3})-C(-w_{2}-w_{3})=0$.\\
			By these equations we obtain
			\[(i\xi_{3,4}+C).w_{3}=0.\]
			
			Hence, we conclude that $\mu=(m,0,\frac{m}{2}+2,-\frac{m}{2})$ for some $m\geq 0$ and since $\dim  F\geq 3$ (since, e.g., $w_2$, $w_3$ and $w_5$ are linearly independent) we also have $m\geq 2$.
			All these conditions lead to
						\begin{align*}
			\vec{m}_{2c}=w_{22}w_{21} \otimes x^{m}_{1}+(w_{11}w_{22}+w_{21}w_{12}) \otimes x^{m-1}_{1}x_{2}-w_{11}w_{12}\otimes x^{m-2}_{1}x^{2}_{2},
			\end{align*}
			in $M(m,0,\frac{m}{2}+2,-\frac{m}{2})$ with $m \geq 2$, which is indeed a singular vector.

			\item [\textbf{2b)}] Let $w_1\neq 0$ or $w_3=0$. 
			We show that in this case necessarily $C=0$ so the $\lambda-$action of Proposition \ref{actiondual} reduces to the action found in Theorem 4.3 of \cite{kac1}; in that case it was shown that there are no singular vectors of degree 2. 
						
			Equation \eqref{B.27B} for $a=1,b=3,c=4$ gives $2iw_{5}+\xi_{13}.(iw_{1}+iw_{2}-iw_{3})+\xi_{14}.(w_{1}-w_{2}-w_{3})-2iCw_{5}=0$.\\
			Equation \eqref{B.27B} for $a=2,b=3,c=4$ gives $2iw_{5}+\xi_{23}.(w_{1}+w_{2}+w_{3})-\xi_{24}.(iw_{1}-iw_{2}+iw_{3})-2iCw_{5}=0$.\\
			We take the sum of these equations and get:
			\begin{align*}
			0&=4iw_5-ie_1.w_1-ie_1.w_2+ie_2.w_1-ie_2.w_2-2if_x.w_3-4iCw_5\\
			&=-2if_x.w_3-4iCw_5.
			\end{align*}
			%\begin{align*}
			%0=&4iw_{5}-F_{1,3}(iw_{1}+iw_{2}-iw_{3})+F_{2,3}(-w_{1}-w_{2}-w_{3})\\
			%&-F_{1,4}(w_{1}-w_{2}-w_{3})+F_{2,4}(iw_{1}-iw_{2}+iw_{3})-4iCw_{5}\\
			%=&4iw_{5}-i(F_{1,3}-iF_{2,3})w_{1}-i(F_{1,3}-iF_{2,3})w_{2}+i(F_{2,4}+iF_{1,4})w_{1}-i(F_{2,4}+iF_{1,4})w_{2}\\
			%&+i(F_{1,3}+F_{2,4}-iF_{1,4}+iF_{2,3})w_{3}-4iCw_{5}\\
			%=&4iw_{5}-4iw_{5}+iE_{-(\varepsilon_{1}-\varepsilon_{2})}w_{3}-4iCw_{5}=-4iCw_{5}+iE_{-(\varepsilon_{1}-\varepsilon_{2})}w_{3}.
			%\end{align*}
			If $w_3=0$ we can conclude $C=0$ so we can assume $w_3\neq 0$ and $w_1\neq 0$. We observe that $w_3$ and $w_1$ are highest weight vectors so that they are scalar multiples of each other. If we take the difference of the two equations above we get $f_x.w_1=0$ and so  we have $C=0$ also in this case.
		\end{itemize}
		\item[\textbf{3)}] Let $w_5=w_6\neq 0$.\\
		This condition implies $w_1\neq 0$ since $e_1.w_1=-w_5-w_6$, and $w_3=0$ since $e_1(w_5-w_6)=2w_3$. 
				\begin{itemize}%[leftmargin=*]
			\item[\textbf{3a)}] Let $w_{2}=0$ and $w_4\neq 0$. By Remark \ref{alfabetadisegni} %we have:
			%\begin{align*}
			%&e_1.w_{1}=-2w_{5} \quad 	&\beta_{1,2}(w_{1})&=2w_{5}, \\ 
			%&e_1.w_{4}=0  \quad  &\beta_{1,2}(w_{4})&=0,\\
			%&e_1.w_{5}=-w_{4} \quad  &\beta_{1,2}(w_{5})&=w_{4}, \\
			%&e_1.w_{6}=-w_{4}  \quad  &\beta_{1,2}(w_{6})&=w_{4}.
			%\end{align*}
			%We have that 
			$w_{4}$ is a highest weight vector and Equations \eqref{notazionevijw} reduce to:
			\begin{align*}
			& v_{1,2}=2iw_{5},\,v_{1,3}=w_{1}+w_{4},\,v_{1,4}=iw_{1}-iw_{4},\,\\
			&v_{2,3}=iw_{1}-iw_{4},\,v_{2,4}=-w_{1}-w_{4},\,v_{3,4}=2iw_{5}.
			\end{align*}
			Let us compute the weight of $w_{4}$.\\
			Equation \eqref{B.23B} for $a=1$, $b=3$ gives $t.(w_1+w_4)-w_1-w_4-i\xi_{12}.(w_1-w_4)+2i\xi_{14}.w_5=0$, and for $a=2,b=3$ gives $it.(w_1-w_4)-i(w_1-w_4)+\xi_{12}.(w_1+w_4)-2i\xi_{24}.w_5=0$.\\
			These two equations provide
			\[
			(t+i\xi_{12}-2).w_4=0.
			\]
			Equation\eqref{B.27B} for $a=1,b=2,c=4$ gives $-w_1-w_4+i\xi_{12}.(w_1-w_4)-2i\xi_{14}.w_5-C(w_1+w_4)=0$, and for $a=2,b=1,c=4$ gives $i(-w_1+w_4)-\xi_{12}.(w_1+w_4)+2i\xi_{24}.w_5-iC(w_1-w_4)=0$.\\
			These two equations provide
			\[
			(i\xi_{12}+C).w_4=0.
			\]						
			Equation \eqref{B.27B} for $a=3,b=1,c=4$ gives $-i(w_1-w_4)-2i\xi_{13}.w_5-\xi_{34}.(w_1+w_4)-iC(w_1-w_4)=0$ and for for $a=3,b=2,c=4$ gives $-w_1-w_4-2i\xi_{23}.w_5+i\xi_{34}.(w_1-w_4)-C(w_1+w_4)=0$.\\
			These two equations provide
			\[
			(i\xi_{34}+C).w_4=0.
			\]
			We conclude that $\mu=(0,n,\frac{n}{2}+2,\frac{n}{2})$ for some $n\geq 0$. Moreover, since $w_1$, $w_5$ and $w_4$ are linearly independent we have $\dim F\geq 3$ and so $n\geq 2$. All this conditions lead to the vector
			\begin{align*}
			\vec{m}_{2d}=w_{22}w_{12} \otimes y^{n}_{1}-(w_{22}w_{11}+w_{21}w_{12}) \otimes y^{n-1}_{1}y_{2}-w_{11}w_{21}\otimes y^{n-2}_{1}x^{2}_{2},
			\end{align*}
			in $M(0,n,\frac{n}{2}+2,\frac{n}{2})$ with $n \geq 2$ which is indeed a singular vector.
			
			\item[\textbf{3b)}] Let $w_2\neq 0$ or $w_4=0$.\\
			We show that in this case necessarily $C=0$, so the $\lambda-$action of Proposition \ref{actiondual} reduces to the action found in Theorem 4.3 of \cite{kac1} and we already know that there are no singular vectors of degree 2.
			
			Equation \eqref{B.27B} for $a=1,b=3,c=4$ gives $-2iw_5+i\xi_{13}.(w_1+w_2-w_4)+\xi_{14}.(w_1-w_2+w_4)-2iCw_5=0$, and for $a=2,b=3,c=4$ gives $-2iw_5+\xi_{23}.(w_1+w_2+w_4)-i\xi_{24}.(w_1-w_2-w_4)-2iCw_5=0$.\\
			Taking the sum of these equations we obtain
			\begin{align*}
			0&=-4iw_{5}-i e_1.w_{1}+ie_2.w_{1}-ie_1.w_{2}\\
			&\quad-ie_2.w_{2}-i(\xi_{13}-\xi_{24}+i\xi_{14}+i\xi_{23}).w_{4}-4iCw_{5}\\
			&=-2if_y.w_4-4iCw_5.
			\end{align*}
			Therefore if $w_4=0$ we can conclude $C=0$. If $w_4\neq 0$ and so $w_2\neq 0$ and both $w_2$ and $w_4$ are highest weight vectors.
			We take the difference of the previous equations and we obtain:
			\begin{align*}
			0&=i(\xi_{13}+i\xi_{23}+\xi_{24}-i\xi_{14}).w_{1}+i(\xi_{13}+i\xi_{23}-\xi_{24}+i\xi_{14}).w_{2}\\
			&\quad-i(\xi_{13}+\xi_{24}+i\xi_{14}-i\xi_{2,3}).w_{4}\\
			&=2if_x.w_1+2if_y.w_2+i(e_1+e_2).w_4\\
			&=2if_x.w_1+2if_y.w_2.	\end{align*}
			Since $w_{2}$ is a highest weight vector and $w_{1}$ is not, these two terms are both 0. In particular, since $w_4$ is a scalar multiple of $w_2$ we have that $f_y.w_4=0$ and we conclude $C=0$ by a previous equation. 
		\end{itemize}
		\item[\textbf{4)}] Let $w_5\neq \pm w_6$.\\
		We show that in this case necessarily $C=0$, so the $\lambda-$action of Proposition \ref{actiondual} reduces to the action found in Theorem 4.3 of \cite{kac1} and we already know that there are no singular vectors of degree 2.\\
		Equation \eqref{B.27B} for $a=1,b=3,c=4$ gives
			\[-2iw_{6}+i\xi_{13}.(w_{1}+w_{2}-w_{3}-w_{4})+\xi_{14}.(w_{1}-w_{2}-w_{3}+w_{4})-2iCw_{5}=0,\]
		and for $a=2,b=3,c=4$ gives
			\[-2iw_{6}+\xi_{23}.(w_{1}+w_{2}+w_{3}+w_{4})-i\xi_{24}.(w_{1}-w_{2}+w_{3}-w_{4})-2iCw_{5}=0.\]
		These equations provide
		\begin{equation}\label{somm1}
		2Cw_5+f_x.w_3+f_y.w_4=0.
		\end{equation}
		Equation \eqref{B.27B} for $a=4,b=1,c=2$ gives
			\[-2iw_{5}+\xi_{14}.(w_{1}+w_{2}+w_{3}+w_4)-i\xi_{24}.(w_{1}+w_{2}-w_{3}-w_4)-2iCw_{6}=0,\]
		and for $a=3,b=1,c=2$ gives
			\[-2iw_{5}+i\xi_{13}.(w_{1}-w_{2}+w_{3}-w_4)+\xi_{23}.(w_{1}-w_{2}-w_{3}+w_4)-2iCw_{6}=0.\]
		These equations provide
		\begin{equation}\label{somm2}
		2Cw_6+f_x.w_3-f_y.w_4=0
		\end{equation}
		If $w_3=0$ or $w_4=0$ we immediate deduce by \eqref{somm1} and \eqref{somm2} that $C=0$, since $w_5\neq \pm w_6$. If $w_3\neq 0$ and $w_4\neq 0$ they are both highest weight vectors. Applying $e_1=e_x+e_y$ to \eqref{somm1} we obtain
		\[
		2C(w_3-w_4)+h_x.w_3+h_y.w_4=0
		\]
and applying $e_2=e_x-e_y$ to \eqref{somm1} we obtain
\[
2C(w_3+w_4)+h_x.w_3-h_y.w_4=0.
\]
From these equations we deduce $2C+m=0$ and so $C\leq 0$, and $-2C	+n=0$ and so $C\geq 0$.	
\end{itemize}
\end{proof}

\section{Singular vectors of degree $3$}
The aim of this section is to classify all singular vectors of degree $3$.\\
We have that a singular vector $\vec{m}$ of degree $3$ is such that:
\begin{align}
\label{notazionevett3eta}
T(\vec{m})=\Theta \sum_{i}\eta_{(i)^{c}} \otimes v_{i,1}+ \sum_{i}\eta_{i} \otimes v_{i,0}.
\end{align}

We write the vector $\vec{m}$ also in the following way
\begin{align}
\label{notazionevett3etanonduale}
\vec{m}=&(\eta_{2}+i\eta_{1})(\eta_{2}-i\eta_{1})(\eta_{4}+i\eta_{3}) \otimes w_{1}+(\eta_{2}+i\eta_{1})(\eta_{2}-i\eta_{1})(\eta_{4}-i\eta_{3}) \otimes w_{2}+\\ \nonumber
&(\eta_{4}+i\eta_{3})(\eta_{4}-i\eta_{3})(\eta_{2}+i\eta_{1}) \otimes w_{3}+(\eta_{4}+i\eta_{3})(\eta_{4}-i\eta_{3})(\eta_{2}-i\eta_{1}) \otimes w_{4}+\\ \nonumber
&\Theta (\eta_{2}+i\eta_{1})\otimes w_{5}+\Theta (\eta_{2}-i\eta_{1})\otimes w_{6}+\Theta (\eta_{4}+i\eta_{3})\otimes w_{7}+\Theta (\eta_{4}-i\eta_{3})\otimes w_{8}\\ \nonumber
=&(2 \Theta i\eta_{3}+2 \Theta  \eta_{4}- 2\eta_{1}\eta_{2}\eta_{3}+2 i\eta_{1}\eta_{2}\eta_{4})\otimes w_{1}+(-2i \Theta \eta_{3}+2 \Theta  \eta_{4}+2 \eta_{1}\eta_{2}\eta_{3}+2i \eta_{1}\eta_{2}\eta_{4})\otimes w_{2}+\\ \nonumber
&(2i \Theta \eta_{1}+2 \Theta  \eta_{2}-2 \eta_{1}\eta_{3}\eta_{4}+2 i\eta_{2}\eta_{3}\eta_{4})\otimes w_{3}+(-2i \Theta \eta_{1}+2 \Theta \eta_{2}+2 \eta_{1}\eta_{3}\eta_{4}+2i \eta_{2}\eta_{3}\eta_{4})\otimes w_{4}+\\ \nonumber
&\Theta (\eta_{2}+i\eta_{1})\otimes w_{5}+\Theta (\eta_{2}-i\eta_{1})\otimes w_{6}+\Theta (\eta_{4}+i\eta_{3})\otimes w_{7}+\Theta (\eta_{4}-i\eta_{3})\otimes w_{8}.
\end{align}
From these two expressions it follows that
\begin{align}
\label{notazionevijwgrado3}
&v_{1,0}=2iw_{3}+2iw_{4},\\ \nonumber
&v_{2,0}=2w_{3}-2w_{4},\\  \nonumber
&v_{3,0}=2iw_{1}+2iw_{2},\\  \nonumber
&v_{4,0}=2w_{1}-2w_{2},\\  \nonumber
&v_{1,1}=-2iw_{3}+2iw_{4}-i w_{5}+iw_{6},\\  \nonumber
&v_{2,1}=2w_{3}+2w_{4}+w_{5}+w_{6},\\  \nonumber
&v_{3,1}=-2iw_{1}+2iw_{2}-iw_{7}+iw_{8},\\  \nonumber
&v_{4,1}=2w_{1}+2w_{2}+w_{7}+w_{8}.  
\end{align}
Indeed, let us show for example one of the previous equations. In \eqref{notazionevett3eta}, let us consider $\eta_{2} \otimes v_{2,0}$. We have that $\eta_{2} $ is the Hodge dual of $-\eta_{134} $. In \eqref{notazionevett3etanonduale}, the terms in $\eta_{134} $ are:
\begin{align*}
-2 \eta_{134}\otimes w_{3}+2 \eta_{134}\otimes w_{4},
\end{align*}
hence $v_{2,0}=2w_{3}-2w_{4}$. Analogously for $v_{1,0},v_{3,0}$ and $v_{4,0}$. Moreover in \eqref{notazionevett3eta}, let us consider, for example, $\Theta\eta_{(1)^{c}} \otimes v_{1,1}=\Theta \eta_{234} \otimes v_{1,1}$. We have that $\Theta \eta_{234} $ is the Hodge dual of $-\Theta \eta_{1}  $. In \eqref{notazionevett3etanonduale}, the terms in $\Theta \eta_{1} $ are:
\begin{align*}
2i \Theta \eta_{1}\otimes w_{3}-2i \Theta \eta_{1}\otimes w_{4}+i\Theta \eta_{1}\otimes w_{5}-i\Theta \eta_{1}\otimes w_{6},
\end{align*}
hence $v_{1,1}=-2iw_{3}+2iw_{4}-i w_{5}+i w_{6}$. Analogously for $v_{2,1},v_{3,1}$ and $v_{4,1}$.\\
In the following lemma we write explicitly some of the relations of Proposition \ref{azioneabcdfinal} for a vector as in formula \eqref{notazionevett3eta}.
\begin{lem}
	\label{azioneabcdfinalgrado3}
	Let $\vec{m} \in \Ind F$ be a singular vector such that $T(\vec m)$ is as in formula \eqref{notazionevett3eta}. Then\\
	\textbf{1)} For all $b\in \{1,2,3,4\}$ we have
	\begin{align}
	\label{C0+B1grado3} 
	0=&\sum_{i}  \bigg[ \sum _{l < k} (  \xi_{blk} \star \eta_{i} ) \otimes   \xi_{lk}.  v_{i,0}  - \varepsilon_{b}( \xi_{(b)^{c}} \star \eta_{i})  \otimes C v_{i,0} +(\xi_b \star \eta_{(i)^{c}})  \otimes t.v_{i,1}\\ \nonumber
	&+ \sum _{l \neq k} (\partial_{l}\xi_{bk} \star \eta_{(i)^{c}} )\otimes \xi_{lk}.  v_{i,1} \bigg].
	\end{align}
	\textbf{2)} For all $s\in \{1,2,3,4\}$ we have

	\begin{align}
	\label{B0+b1grado3}
	0=&\sum_{i}\bigg[(\xi_s \star \eta_{i} ) \otimes t.v_{i,0}+\sum^{4}_{l=1} \partial_{l}(  \xi_{sl} \star \eta_{i} ) \otimes  v_{i,0} + \sum _{l \neq k} (\partial_{l}\xi_{sk} \star \eta_{i}) \otimes \xi_{lk}.  v_{i,0} \\ \nonumber
	&+ \sum^{4}_{l=1}(\partial_{l}\xi_s\star \partial_{l}\eta_{(i)^{c}}) \otimes  v_{i,1}  \bigg],
	\end{align}
	For all $b\in \{1,2,3,4\}$ we have
	\begin{align}
	\label{B1+a1grado3}
	0=&\sum_{i} \bigg[(\xi_b \star \eta_{(i)^{c}} ) \otimes t.v_{i,1} + \sum _{l \neq k}( \partial_{l}\xi_{bk} \star \eta_{(i)^{c}} )\otimes \xi_{lk}.  v_{i,1}- (\xi_b \star \eta_{(i)^{c}}) \otimes v_{i,1}\bigg].
	\end{align}
	\textbf{3)} For all $I$ such that $|I|=3$ we have
	
	\begin{align}
	\label{b0+G1grado3}
	0=&\sum_{i} \bigg[ \partial_{i}\eta_I \otimes  v_{i,0} +\sum_{r<s}  (\partial_{rs}\xi_I \star \eta_{i}) \otimes \xi_{rs}. v_{i,0} + \varepsilon_{I}( \xi_{I^{c}} \star \eta_{i})  \otimes C  v_{i,0}  \bigg].
	\end{align}
	
	For all $I$ such that $|I|=4$ we have
	\begin{align}
	\label{b0+G1grado3bis}
	0=&\sum_{i} \bigg[ - \partial_{i}\eta_I \otimes  v_{i,0} +\sum_{r<s}  (\partial_{rs}f \star \eta_{i}) \otimes \xi_{rs} .v_{i,0}  - \varepsilon_{I} \eta_{(i)^{c}} \otimes C v_{i,1}\bigg].
	\end{align}

\end{lem}
\begin{proof}
These are particular cases of Proposition \ref{azioneabcdfinal}. In particular we have \eqref{C0+B1grado3} is $C_0(b)+B_1(b)=0$, \eqref{B0+b1grado3} is $B_0(s)+b_1(s)=0$, \eqref{B1+a1grado3} is $B_1(b)+a_1(b)=0$, \eqref{b0+G1grado3} is $b_0(I)+G_1(I)=0$ for $|I|=3$ and \eqref{b0+G1grado3bis} is $b_0(I)+G_1(I)=0$ for $|I|=4$.
\end{proof}
\begin{lem}
	\label{azioneabcdfinalgrado3finali}
	Let $\vec{m} \in \Ind F$ be a highest weight singular vector such that $T(\vec{m})$ is as in formula \eqref{notazionevett3eta}. Then
	for every $(a,b,c,d)$ permutation of $\{1,2,3,4\}$ we have
	\begin{align}
	\label{eqC} &v_{a,1}=(-1)^{a+1}2 Cv_{a,0}.\\
	\label{E00} &t.v_{a,0}-2 v_{a,0}+\xi_{ab}.v_{b,0}=0\\
	\label{57}  &v_{c,0}+\xi_{ca}.v_{a,0}+\xi_{cb}.v_{b,0}=0,\\
	\label{60}  &\xi_{bc}.v_{d,0}+\varepsilon_{(a,b,c)}Cv_{a,0}=0.
	\end{align}
	Moreover $C$ (resp. $t$) acts as multiplication by $\pm \frac{1}{2}$ (resp. $\frac{5}{2}$) on $F$.\\

	Finally:
	\begin{align}
	\label{alfabetagrado3}
	&e_1.v_{1,0}=-v_{3,0},            &&e_2.v_{1,0}=- iv_{4,0},\\ \nonumber
	&e_1. v_{2,0}=iv_{3,0},           &&e_2. v_{2,0}=- v_{4,0},\\ \nonumber
	&e_1. v_{3,0}=v_{1,0}-iv_{2,0},   &&e_2. v_{3,0}=0,\\ \nonumber
	&e_1. v_{4,0}=0                   &&e_2.v_{4,0}=iv_{1,0}+v_{2,0},
	\end{align}
	where $e_1$ and $e_2$ are defined by \eqref{notazionealfa} and \eqref{notazionebeta}.
\end{lem}
\begin{proof}
	$\bullet$ Equation \eqref{eqC}.
	We consider the difference between \eqref{C0+B1grado3} and \eqref{B1+a1grado3}. We assume $(a,c,d)=(b)^{c}$. We have that:
	\begin{align*}
	&-(\xi_{b}\star \eta_{(b)^{c}}) \otimes v_{b,1}=\sum^{4}_{i=1}\sum_{l<k}(\xi_{blk} \star\eta_{i})\otimes \xi_{lk}. v_{i,0} -\varepsilon_{b} (\xi_{acd}\star \eta_{b}) \otimes  Cv_{b,0}.
	\end{align*}
	It is equivalent to:
	\begin{align}
	\label{a1=C0}
	&(\xi_{b}\star \eta_{(b)^{c}} )\otimes v_{b,1}=-\sum_{l<k}(\xi_{blk}\star \eta_{(b,l,k)^{c}})\otimes \xi_{lk}. v_{(b,l,k)^{c},0} -\varepsilon_{b}\eta_{bacd}\otimes  Cv_{b,0}.
	\end{align}
	%Let us focus on equation $b_{1}(\xi_{s})+B_{0}(\xi_{s})=0$, with $s \neq b$. We have:
	Let us focus on Equation \eqref{B0+b1grado3} for $s \neq b$. We have:
	\begin{align}
	\label{B0+b1grado3appoggio}
	0=\sum^{4}_{i=1} \partial_{s}\eta_{(i)^{c}}\otimes v_{i,1}+ \sum^{4}_{i=1} (\xi_{s}\star \eta_{i})\otimes t.v_{i,0}+  \sum^{4}_{i=1} \sum^{4}_{l=1}\partial_{l}(\xi_{sl}\star\eta_{i})  \otimes v_{i,0}+\sum^{4}_{i=1}\sum_{l \neq s}(\xi_{l}\star \eta_{i})\otimes \xi_{sl}.v_{i,0} .
	\end{align}
	The terms in $\eta_{(s,b)^{c}} $ of this equation are:%$b_{1}(\xi_{s})+B_{0}(\xi_{s})=0$ are:
	\begin{align*}
	\partial_{s}\eta_{(b)^{c}}\otimes v_{b,1}+\sum_{l \neq s,b}(\xi_{l}\star \eta_{(s,b,l)^{c}})\otimes \xi_{sl}.v_{(s,b,l)^{c},0}=0.
	\end{align*}
	We take the sum over $s \neq b$ and, as in \cite{kac1}, using \eqref{a1=C0} we obtain:
	\begin{align*}
	0=&\sum_{s \neq b} (\xi_{s} \star  \partial_{s}\eta_{(b)^{c}} )\otimes v_{b,1}+\sum_{s \neq b}\sum_{l \neq s,b}(\xi_{sl}\star \eta_{(s,b,l)^{c}})\otimes \xi_{sl}.v_{(s,b,l)^{c},0}\\
	=&3\eta_{(b)^{c}}\otimes v_{b,1}+2\sum_{s<l}(\xi_{sl}\star \eta_{(s,b,l)^{c}})\otimes \xi_{sl}.v_{(s,b,l)^{c},0}\\
	%=&(3-2)\eta_{(b)^{c}}\otimes v_{b,1}-2\varepsilon_{(b)}  \eta_{(b)^{c}}\otimes  Cv_{b,0}\\
	=&\eta_{(b)^{c}}\otimes (v_{b,1}-2\varepsilon_{(b)} Cv_{b,0}).
	\end{align*}
	Equation \eqref{eqC} follows.
	
	%%%%%%%%%%%%%%%%%We write these conditions in terms of $w$'s:
	%%%%%%%%%%%%%%%%%\begin{align}
	%%%%%%%%%%%%%%%%%\label{1}
	%%%%%%%%%%%%%%%%%&2w_{3}+2w_{4}+w_{5}+w_{6}=2C(-2w_{3}+2w_{4}), \\ \label{2}
	%%%%%%%%%%%%%%%%%&2iw_{3}-2iw_{4}+iw_{5}-iw_{6}=-2C(2iw_{3}+2iw_{4}),\\ \label{3}
	%%%%%%%%%%%%%%%%%&2w_{1}+2w_{2}+w_{7}+w_{8}=2C(-2w_{1}+2w_{2}),\\ \label{4}
	%%%%%%%%%%%%%%%%%&2iw_{1}-2iw_{2}+iw_{7}-iw_{8}=-2C(2iw_{1}+2iw_{2}).
	%%%%%%%%%%%%%%%%%\end{align}
	%%%%%%%%%%%%%%%%%We take i\eqref{1}+\eqref{2} and we obtain:
	%%%%%%%%%%%%%%%%%\begin{align*}
	%%%%%%%%%%%%%%%%%w_{5}=-(2+4C)w_{3}.
	%%%%%%%%%%%%%%%%%\end{align*}
	%%%%%%%%%%%%%%%%%We take i\eqref{1}-\eqref{2} and we obtain:
	%%%%%%%%%%%%%%%%%\begin{align*}
	%%%%%%%%%%%%%%%%%w_{6}=-(2-4C)w_{4}.
	%%%%%%%%%%%%%%%%%\end{align*}
	%%%%%%%%%%%%%%%%%We take i\eqref{3}+\eqref{4} and we obtain:
	%%%%%%%%%%%%%%%%%\begin{align*}
	%%%%%%%%%%%%%%%%%w_{7}=-(2+4C)w_{1}.
	%%%%%%%%%%%%%%%%%\end{align*}
	%%%%%%%%%%%%%%%%%We take i\eqref{3}-\eqref{4} and we obtain:
	%%%%%%%%%%%%%%%%%\begin{align*}
	%%%%%%%%%%%%%%%%%w_{8}=-(2-4C)w_{2}.
	%%%%%%%%%%%%%%%%%\end{align*}
	%$b_{1}(\xi_{s})+B_{0}(\xi_{s})=0$
	$\bullet$ Equation \eqref{E00}.
	Given $r\neq s \in \left\{1,2,3,4\right\}$, the terms in $\eta_{sr} $ of \eqref{B0+b1grado3appoggio} are:
	\begin{align*}
	\eta_{sr} \otimes t.v_{r,0}+\sum_{l \neq s,r} \partial_{l}(\xi_{sl}\star\eta_{r}) \otimes v_{r,0}-\eta_{sr} \otimes \xi_{sr}.v_{s,0}=0.
	\end{align*}
	This condition is equivalent to:
	\begin{align*}
	%\label{E00}
	t.v_{r,0}-2 v_{r,0}-\xi _{sr}.v_{s,0}=0
	\end{align*}
	which is Equation \eqref{E00}.
	
	$\bullet$ Equations \eqref{57} and \eqref{60}.

	Let us analyze Equation \eqref{b0+G1grado3} for $I=abc$. We obtain:
	\begin{align*}
	\sum^{4}_{i=1}\sum^{4}_{l=1} (\partial_{l}\xi_{abc}\star \partial_{l}\eta_{i}) \otimes v_{i,0}+\sum_{r<s}(\partial_{rs}\xi_{abc}	\star \eta_{i} )\otimes \xi_{rs}.v_{i,0}+\sum_{i}\varepsilon_{(a,b,c)}(\xi_{d}\star \eta_{i}) \otimes C v_{i,0}=0.
	\end{align*}
	Looking at the coefficients of $\eta_{ab}$ and $\eta_{ad}$  we obtain Equations \eqref{57} and \eqref{60}: 
	\begin{align*}
	%\label{57}
	v_{c,0}+\xi_{ca}.v_{a,0}+\xi_{cb}.v_{b,0}=0,\\
	%\label{60}
	\xi_{bc}.v_{d,0}+\varepsilon_{(a,b,c)}Cv_{a,0}=0.
	\end{align*}
	$\bullet$ $C=\pm 1/2$ and $t=5/2$.\\
	Using \eqref{eqC}, Equation \eqref{b0+G1grado3bis} for $I=1234$ is:
	\begin{align*}
	0=&- \sum^{4}_{i=1}\sum^{4}_{l=1} (\partial_{l}\xi_{1234}\star \partial_{l}\eta_{i}) \otimes  v_{i,0}+\sum_{r<s}\sum^{4}_{i=1}  (\partial_{rs}\xi_{1234}\star \eta_{i} )\otimes \xi_{rs}. v_{i,0} -C \sum_{i} \eta_{(i)^{c}} \otimes v_{i,1}\\
	=&\eta_{123}\otimes ((1+2C^2)v_{4,0}+\xi_{41}.v_{1,0}+\xi_{42}.v_{2,0}+\xi_{43}.v_{3,0})\\
	&-\eta_{124}\otimes ((1+2C^2)v_{3,0}+\xi_{31}.v_{1,0}+\xi_{32}.v_{2,0}+\xi_{34}.v_{4,0})\\
	&+\eta_{134}\otimes ((1+2C^2)v_{2,0}+\xi_{21}.v_{1,0}+\xi_{23}.v_{3,0}+\xi_{24}.v_{4,0})\\
	&-\eta_{234}\otimes ((1+2C^2)v_{1,0}+\xi_{12}.v_{2,0}+\xi_{13}.v_{3,0}+\xi_{14}.v_{4,0}).
	\end{align*}
	Therefore for every $a=1,2,3,4$ we have $(1+2C^2)v_{a,0}+\sum_{b\neq a}\xi_{ab}.v_{b,0}=0$ and by Equation \eqref{E00} we deduce $(7+2C^2-3t).v_{a,0}=0$. This implies that $t$ acts as $\frac{1}{3}(7+2C^2)$ on $F$.

	Equation \eqref{57}, for $a=2,b=3,c=1$, is:
	%\begin{align*}
	%\xi_{2}\xi_{3} \otimes v_{1,0}-\xi_{2}\xi_{3} \otimes F_{1,3}v_{3,0}+\xi_{3}\xi_{2} \otimes F_{1,2}v_{2,0}=0.
	%\end{align*}
	%Therefore we have:
	\begin{align*}
	v_{1,0}+\xi_{13}.v_{3,0}+\xi_{12}.v_{2,0}=0.
	%&v_{1,0}-\frac{1}{3}v_{1,0}-\frac{1}{3}v_{1,0}=\frac{1}{3}v_{1,0}
	\end{align*}
	Using \eqref{E00} and the fact that $t$ acts as $\frac{1}{3}(7+2C^2)$, we get
	\begin{align*}
	0=&v_{1,0}+\xi_{13}.v_{3,0}+\xi_{12}.v_{2,0}\\
	=&v_{1,0}-2\frac{1+2C^{2}}{3}v_{1,0}=\frac{1-4C^{2}}{3}v_{1,0}.
	\end{align*}
	From this we deduce that $C=\pm \frac{1}{2}$ and so $t$ acts as $\frac{5}{2}$.
	
	$\bullet$ Equations \eqref{alfabetagrado3}.	The fact that $\vec{m}$ is annihilated by $e_1=-\xi_{13}+i\xi_{23}$ provides:
	\begin{align*}
	0=&-\sum_{i} \sum^{4}_{l=1}(\partial_{l}(-\xi_{13}+i\xi_{23})\star \partial_{l}\eta_{i}) \otimes v_{i,0}+\sum_{i}\sum_{r<s}(\partial_{rs}(-\xi_{13}+i\xi_{23})\star \eta_{i}) \otimes \xi_{rs}. v_{i,0}\\
	=&\eta_{3} \otimes v_{1,0}-\eta_{1} \otimes e_1.v_{1,0}-i\eta_{3} \otimes v_{2,0}-\eta_{2} \otimes e_1. v_{2,0}+(-\eta_{1} +i\eta_{2} )\otimes v_{3,0}-\eta_{3}\otimes  e_1.v_{3,0}\\
	&-\eta_{4} \otimes e_1. v_{4,0}\\
	=&-\eta_{1}\otimes (e_1. v_{1,0}+ v_{3,0})+\eta_{2}\otimes(i v_{3,0}-e_1.v_{2,0})+\eta_{3}\otimes ( v_{1,0}-i v_{2,0}-e_1. v_{3,0})-\eta_{4} \otimes  e_1. v_{4,0}.
	\end{align*}
	
	Equations \eqref{alfabetagrado3} for $e_1$ follow. Equations for $e_2$ are obtained similarly.
\end{proof}
\begin{rem}
	Let us point out that relations \eqref{eqC} are equivalent to the following, using notation \eqref{notazionevijwgrado3}:
	\begin{align*}
	%\label{1}
	-2iw_{3}+2iw_{4}-iw_{5}+iw_{6}&=2C(2iw_{3}+2iw_{4}), \\ %\label{2}
	2w_{3}+2w_{4}+w_{5}+w_{6}&=-2C(2w_{3}-2w_{4}),\\ %\label{3}
	-2iw_{1}+2iw_{2}-iw_{7}+iw_{8}&=2C(2iw_{1}+2iw_{2}),\\ %\label{4}
	2w_{1}+2w_{2}+w_{7}+w_{8}&=-2C(2w_{1}-2w_{2}).
	\end{align*}
	Thus, we obtain:
	%We take i\eqref{1}+\eqref{2} and we obtain:
	\begin{align}
	\label{eqCconw}
	w_{5}=-(2+4C)w_{3},\\ \nonumber
	%\end{align*}
	%We take i\eqref{1}-\eqref{2} and we obtain:
	%\begin{align*}
	w_{6}=-(2-4C)w_{4},\\ \nonumber
	%\end{align*}
	%We take i\eqref{3}+\eqref{4} and we obtain:
	%\begin{align*}
	w_{7}=-(2+4C)w_{1},\\ \nonumber
	%\end{align*}
	%We take i\eqref{3}-\eqref{4} and we obtain:
	%\begin{align*}
	w_{8}=-(2-4C)w_{2}.
	\end{align}
	Equations \eqref{alfabetagrado3} are therefore equivalent to the following, using notation \eqref{notazionevijwgrado3}:
	\begin{align}
	\label{alfabetagrado3conw}
	&e_1. w_{1}=w_{4}, &&e_2.w_{1}=-w_{4},\\ \nonumber
	&e_1. w_{2}=w_{4}, &&e_2. w_{2}=w_{4},\\ \nonumber
	&e_1. w_{3}=-w_{1}-w_{2}, &&e_2. w_{3}=-w_{1}+w_{2},\\ \nonumber
	&e_1.w_{4}=0, &&e_2.w_{4}=0.
	\end{align}
	We represent these relations in the following picture
	
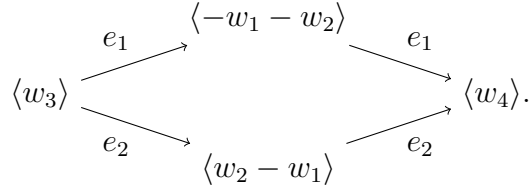
\begin{center}
	\begin{tikzpicture}
		\node at (0,0) (w3) {$\langle w_{3} \rangle$};
	\node at (3,1) (w2+w1) {$\langle -w_{1}-w_{2} \rangle $};
		\node at (3,-1) (w2-w1) {$\langle w_{2}-w_{1} \rangle $};
		\node at (6,0) (w4) {$\langle w_{4} \rangle $.};
		\draw (w3) edge[->] (w2+w1);
		\draw (w3) edge[->] (w2-w1);
	\node at (1,0.7) {$e_1$};
		\node at (1,-0.7) {$e_2$};
		\draw (w2-w1) edge[->] (w4);
		\draw (w2+w1) edge[->] (w4);
		\node at (5,0.7) {$e_1$};
		\node at (5,-0.7) {$e_2$};
		\end{tikzpicture}
	\end{center}
\end{rem}
We are now ready to prove the stated classification of singular vectors of degree 3.
\begin{proof}[Proof of Theorem \ref{sing3}]
	Let $\mu=(m,n,5/2,C)$, with $C=\pm 1/2$, be the highest weight of $F$ with respect with $(h_x,h_y,t,C)$. We observe that $w_3\neq 0$ otherwise $\vec{m}=0$.
	\begin{itemize}%[leftmargin=*]
		 
		\item[\textbf{1)}] Let $w_{4}=0$.
		\begin{itemize}%[leftmargin=*]
			\item [\textbf{1a)}] Let $w_{2}=0$ and $w_{1} \neq 0$.  By Equations \eqref{alfabetagrado3conw}, we have that $w_{1} $ is a highest weight vector.
			By \eqref{notazionevijwgrado3} we have:
			\begin{align*}
			&v_{1,0}=2iw_{3},\,v_{2,0}=2w_{3},\,v_{3,0}=2iw_{1},\,v_{4,0}=2w_{1},\\
			&v_{1,1}=4iCw_{3},\,v_{2,1}=-4 Cw_{3},\,v_{3,1}=4i Cw_{1},\,v_{4,1}=-4 Cw_{1}. 
			\end{align*}
			Equation \eqref{E00} for $a=3,b=4$ gives $-\xi_{34}.v_{3,0}=(2-t).v_{4,0}$ which is equivalent to
			\[
			(i\xi_{34}-1/2 ).w_{1}=0.
			\]
			Equation \eqref{60} for $a=3,b=1,c=2$ gives $-\xi_{12}.v_{4,0}-Cv_{3,0}=0$ which is equivalent to
			\begin{align*}
			&(-i\xi_{12}+C).w_{1}=0.
			\end{align*}
			These equations imply $(h_y+C+1/2).w_1=0$ and so $n+C+1/2=0$ which implies $C=-1/2$ (since $n\geq 0$) and $n=0$.
			Similarly the same equations imply $m+C-1/2=0$ and so $m=1$ and hence $\mu=(1,0,5/2,-1/2)$.
			 
			By Equations \eqref{alfabetagrado3conw} we know that $2e_x.w_3=e_1.w_{3}+e_2. w_{3}=-2w_{1}$. Hence $w_{3}=-f_x. w_{1}$. All these conditions lead to the vector
			\begin{align*}
			\vec{m}_{3a}=w_{11}w_{22}w_{21}\otimes x_{1}+w_{21}w_{12}w_{11}\otimes x_{2}\in M(1,0,5/2,-1/2),
			\end{align*}
			which is indeed a singular vector.
			\item [\textbf{1b)}] Let $w_{1}=0$ and $w_{2} \neq 0$.
			By Equations \eqref{alfabetagrado3conw}, we have that $w_{2} $ is a highest weight vector.
			From \eqref{notazionevijwgrado3} we have:
			\begin{align*}
			&v_{1,0}=2iw_{3},\,v_{2,0}=2w_{3},\,v_{3,0}=2iw_{2},\,v_{4,0}=-2w_{2},\\
			&v_{1,1}=4i Cw_{3},\,v_{2,1}=-4 Cw_{3},\,v_{3,1}=4iCw_{2},\,v_{4,1}=4 Cw_{2}.
			\end{align*}
			Equation \eqref{E00} for $a=3,b=4$ gives $-\xi_{34}.v_{3,0}=2v_{4,0}-t.v_{4,0}$ which is equivalent to
			\begin{align*}
			(i\xi_{34}+1/2).w_2=0.
			\end{align*}
			Equation \eqref{60} for $a=3,b=1,c=2$ gives $-\xi_{12}.v_{4,0}-Cv_{3,0}=0$ which is equivalent to
			\begin{align*}
			(i\xi_{12}+C).w_2=0.
			\end{align*}
These equations imply $(h_x-C+1/2).w_2=0$ and so $(m-C+1/2)=0$. Since $m\geq 0$ and $C=\pm 1/2$ we necessarily have $C=1/2$ and $m=0$. The same equations also imply $n-C-1/2=0$ and hence we have $n=1$.
			
			By Equations \eqref{alfabetagrado3conw} we know that $2e_y.w_3=e_1.w_3-e_2. w_{3}=-2w_{2}$. Hence $w_{3}=-f_y. w_{2}$. %Moreover, by Equations \eqref{eqCconw} and $C=\frac{1}{2}$, we have that $w_{5}=-2E_{-(\varepsilon_{1}+\varepsilon_{2})} w_{2}$ and $w_{6}=w_{7}=w_{8}=0$.  All the other equations of Lemma \ref{azioneabcdfinalgrado3finali} are verified by this choice of $v_{1,0},v_{2,0},v_{3,0},v_{4,0},v_{1,1},v_{2,1},v_{3,1},v_{4,1}$.\\
			
			These conditions lead to the vector
			\begin{align*}
			\vec{m}_{3b}=w_{11}w_{22}w_{12}\otimes y_{1}+w_{12}w_{21}w_{11}\otimes y_{2}\in M(0,1,5/2,1/2),
			\end{align*}
			which is indeed a singular vector.
			\item [\textbf{1c)}] Let $w_{1} \neq 0$, $w_{2} \neq 0$. By Equations \eqref{alfabetagrado3conw}, we know that $w_{1}$ and $w_{2}$ are highest weight vectors.\\
			Equations \eqref{60}  for $a=3,b=1,c=2$, using \eqref{notazionevijwgrado3} gives
			\begin{align*}
			-i\xi_{12}.(w_1-w_2)+C(w_1+w_2)=0
			\end{align*}
			Equation \eqref{60} for $a=4,b=1,c=2$, using \eqref{notazionevijwgrado3}, gives
			\begin{align*}
			-i\xi_{12}.(w_1+w_2)+C(w_1-w_2)=0
			\end{align*}
			These two equations lead to $C=0$ which is a contradiction since $C=\pm 1/2$.
			
			\item [\textbf{1d)}] We suppose $w_{1}=w_{2} =0$. By Equations \eqref{alfabetagrado3conw}, we know that $w_{3}$ is a highest weight vector.\\
			Equation \eqref{57} for $a=2,b=3,c=1$, using \eqref{notazionevijwgrado3}, gives 
			\[0=(-2i\xi_{12}+2).w_3=(h_x+h_y+2).w_3\]
			which implies $m+n+2=0$, a contradiction.
		\end{itemize}
		\item[\textbf{2)}] Let $w_{4} \neq 0$. By Equations \eqref{alfabetagrado3conw}, we have that $w_{1} \neq 0$, $w_{2} \neq 0$, $w_{3} \neq 0$ and that $w_{4}$ is a highest weight vector. \\
		Equation \eqref{57} for $a=1,b=3,c=2$ gives
		\begin{align*}
		-w_{3}+w_{4}-i\xi_{23}.(w_{1}+w_{2})+i\xi_{12}.(w_{3}+w_{4})=0,
		\end{align*}
		and for $a=2,b=3,c=1$ gives 
		\begin{align*}
		-w_{3}-w_{4}+i\xi_{12}.(w_{3}-w_{4})-\xi_{13}.(w_{1}+w_{2})=0.
		\end{align*}
		These equations imply
		\begin{align*}
		0=&2w_4-e_1.w_1-e_1.w_2+2i\xi_{12}.w_4=2i\xi_{12}.w_4.
		\end{align*}
		Equation \eqref{60} for $a=1,b=3,c=4$ gives
		\begin{align*}
		-i\xi_{34}.(w_{3}-w_{4})+C(w_{3}+w_{4})=0,
		\end{align*}
		and for $a=2,b=3,c=4$ gives
		\begin{align*}
		-i\xi_{3,4}.(w_{3}+w_{4})+C(w_{3}-w_{4})=0.
		\end{align*}
		These equations imply
		\[
		(i\xi_{34}+C).w_4=0
		\]
		We deduce that $m+C=0$ and $-n+C=0$, hence $C=0$, a contradiction.
	\end{itemize}
\end{proof}
\section{Singular vectors of degree $1$}
The aim of this section is to classify singular vectors of degree $1$.
Let us consider a vector $\vec{m} \in \Ind(F)$ of degree $1$ such that $T(\vec{m})$ is of the form:
\begin{align}
\label{notazionevett1eta}
T(\vec{m})=\sum_{i}\eta_{(i)^{c}} \otimes v_{i} .
\end{align}
We write $\vec{m}$ as:
\begin{align}
\label{notazionevett1etanonduale}
\vec{m}=(\eta_2-i\eta_1)\otimes w_1+(\eta_4-i\eta_3)\otimes w_2+ (\eta_2+i\eta_1)\otimes \widetilde w_1+(\eta_4+i\eta_3)\otimes \widetilde w_2%\sum_{l=1}^{2}\left((\eta_{2l}-i\eta_{2l-1})\otimes w_{l}+(\eta_{2l}+i\eta_{2l-1})\otimes \widetilde{w}_{l}\right).
\end{align}
Hence :
\begin{align}
\label{notazionevwgrado1}
&v_{1}=i(w_{1}-\widetilde{w}_{1}),\,\,v_{2}=w_{1}+\widetilde{w}_{1},\,\, v_{3}=i(w_{2}-\widetilde{w}_{2}),\,\, v_{4}=w_{2}+\widetilde{w}_{2}.
\end{align}
Indeed, let us show one of these relations. In \eqref{notazionevett1eta}, let us consider $\eta_{(1)^{c}} \otimes v_{1}$. We have that $\eta_{(1)^{c}}=\eta_{234} $ is the Hodge dual of $-\eta_{1}  $. In \eqref{notazionevett1etanonduale}, the terms in $\eta_{1}  $ are $-i\eta_{1}\otimes w_{1}+i\eta_{1}\otimes \widetilde{w}_{1}$. Hence $v_{1}=i(w_{1}-\widetilde{w}_{1})$. The other relations in \eqref{notazionevwgrado1} are obtained analogously. \\
In the following lemma we write explicitly some of the relations of Proposition \ref{azioneabcdfinal} for a vector as in formula \eqref{notazionevett1eta} that we need.
\begin{lem}
	\label{azioneabcdfinalgrado1}
	Let $\vec{m} \in \Ind (F)$ be a highest weight singular vector such that $T(\vec{m})$ is as in \eqref{notazionevett1eta}. Then for all $a\in \{1,2,3,4\}$ we have
	\begin{align}
	\label{B0grado1} 
	0=&\sum_{i} \Big[(\xi_a \star \eta_{(i)^{c}})  \otimes t.v_{i}+\sum^{4}_{l=1} (\partial_{l}\xi_{al} \star \eta_{(i)^{c}} ) \otimes  v_{i} + \sum _{l \neq k} (\partial_{l}\xi_{1k} \star \eta_{(i)^{c}} )\otimes \xi_{lk}, v_{i} \Big].
	\end{align}
	and for every permutation $(a,b,c,d)$ of $\{1,2,3,4\}$ we have
	
	\begin{align}
	\label{b0grado1} 
	0=&\sum_{i} \Big[ \sum_{r<s}  (\partial_{rs}\xi_{abc} \star \eta_{(i)^{c}} )\otimes \xi_{rs}. v_{i}+ \varepsilon_{abc}  \,(\xi_{d} \star \eta_{(i)^{c}} ) \otimes C  v_{i}  \Big].
	\end{align}
\end{lem}
\begin{proof}
	Equation \eqref{B0grado1} is obtained by $B_0(a)+b_1(a)=0$ and Equation \eqref{b0grado1} is obtained by $b_0(abc)+G_1(abc)=0$ in Proposition \ref{azioneabcdfinal}. Note that $b_1(a)=0$ and $G_1(abc)=0$ since $\vec{m}$ has degree 1 and so the previous equations reduce to $B_0(a)=0$ and $b_0(abc)=0$. 
\end{proof}
\begin{lem}
	\label{azioneabcdfinalgrado1finali}
	Let $\vec{m} \in \Ind (F)$ be a highest weight sungular vector such that $T(\vec{m})$ is as in  \eqref{notazionevett1eta}. Then for every permutation $(a,b,c,d)$ of $\left\{1,2,3,4\right\}$ we have
	\begin{align}
	\label{B.44} &(-1)^{a}t.v_{a}+\sum_{k \neq a}(-1)^{k}\xi_{ak}.v_{k}=0;\\
	\label{B.45} &(-1)^{c} \xi_{ab}.v_{c}+(-1)^{b}\xi_{ca}.v_{b}+(-1)^{a} \xi_{bc}.v_{a} -\varepsilon_{(a,b,c)}  (-1)^{d}C v_{d}=0.
	\end{align}
	Moreover
	\begin{align}
	\label{alfabetagrado1}
	&e_1.v_{1}=-v_{3},       &&e_2.v_{1}=iv_{4},\\ \nonumber
	&e_1.v_{2}=-iv_{3},      &&e_2.v_{2}=-v_{4},\\ \nonumber
	&e_1.v_{3}=v_{1}+iv_{2}, &&e_2.v_{3}=0,\\ \nonumber
	&e_1.v_{4}=0,            &&e_2.v_{4}=-iv_{1}+v_{2},
	\end{align}
	where $e_1$ and $e_2$ are defined in \eqref{notazionealfa} and \eqref{notazionebeta}.
\end{lem}
\begin{proof}
	Equation \eqref{B.44} follows by considering the terms in $\eta_{1234}$ in \eqref{B0grado1}.\\
	For equation \eqref{B.45} we can assume $a<b<c$ with no loss of generality. Equation \eqref{B0grado1} becomes
	\begin{align*}
	0=\eta_{c}\eta_{(c)^{c}} \otimes \xi_{ab}.v_{c}+\eta_{b}\eta_{(b)^{c}} \otimes \xi_{ca}.v_{b}+\eta_{a}\eta_{(a)^{c}} \otimes \xi_{bc}.v_{a}-\varepsilon_{(a,b,c)} \eta_{d} \eta_{(d)^{c}}\otimes C v_{d},
	\end{align*}
	which is equivalent to \eqref{B.45}.
	
	The fact that $\vec{m}$ is annihilated by $e_1$ implies
	\begin{align*}
	0=&-\sum_{i} \sum^{4}_{l=1}(\partial_{l}(-\xi_{13}+i\xi_{23})\star \partial_{l}\eta_{(i)^{c}} )\otimes v_{i}+\sum_{i}\sum_{r<s}(\partial_{rs}(-\xi_{13}+i\xi_{23})\star \eta_{(i)^{c}}) \otimes \xi_{rs}. v_{i}\\
	%=&-(\xi_{1}-i\xi_{2})\star (-\eta_{2}\eta_{4})  \otimes v_{1}-\eta_{2}\eta_{3}\eta_{4}\otimes \alpha_{1,2} v_{1}-(\xi_{1}-i\xi_{2})\star (-\eta_{1}\eta_{4})  \otimes v_{2}-\eta_{1}\eta_{3}\eta_{4}\otimes \alpha_{1,2} v_{2}\\
	%&-(-\xi_{3})\star (\eta_{2}\eta_{4})  \otimes v_{3}-(i\xi_{3})\star(-\eta_{1}\eta_{4})  \otimes v_{3}-\eta_{1}\eta_{2}\eta_{4}\otimes \alpha_{1,2} v_{3}-\eta_{1}\eta_{2}\eta_{3}\otimes \alpha_{1,2} v_{4}\\
	=&\eta_{124}\otimes (v_{1}+iv_{2}- e_1. v_{3})+\eta_{234}\otimes (-v_{3}- e_1. v_{1})+\eta_{134}\otimes (-iv_{3}- e_1. v_{2})+\eta_{123}\otimes (- e_1.v_{4}).
	\end{align*}
	Therefore:
	\begin{align*}
	&e_1.v_{1}=-v_{3},\\
	&e_1.v_{2}=-iv_{3},\\
	&e_1.v_{3}=v_{1}+iv_{2},\\
	&e_1.v_{4}=0.
	\end{align*}
	Equations \eqref{alfabetagrado1} for $e_2$ follow similarly.
	%Let us consider Equation \eqref{b0grado1} for $f=\beta_{1,2} $. We have:
	%\begin{align*}
	%0=&-b_{0}(-\xi_{2}\xi_{4}-i\xi_{1}\xi_{4})\\
	%0=&-\sum_{i} \sum^{4}_{l=1}\partial_{l}(-\xi_{2}\xi_{4}-i\xi_{1}\xi_{4})\star \partial_{l}(\eta_{(i)^{c}}) \otimes v_{i}-\sum_{i}\sum_{r<s}\partial_{r}\partial_{s}(-\xi_{2}\xi_{4}-i\xi_{1}\xi_{4})\star \eta_{(i)^{c}} \otimes F_{r,s} v_{i}\\
	%=&-(\xi_{2}+i\xi_{1})\star (\eta_{2}\eta_{3}) \otimes v_{1}-\eta_{2}\eta_{3}\eta_{4} \otimes \beta_{1,2} v_{1}-(\xi_{2})\star (\eta_{1}\eta_{3}) \otimes v_{2}-\eta_{1}\eta_{3}\eta_{4} \otimes \beta_{1,2} v_{2}-\eta_{1}\eta_{2}\eta_{4} \otimes \beta_{1,2} v_{3}\\
	%&-(-\xi_{4})\star (-\eta_{1}\eta_{3}) \otimes v_{4}-(-i\xi_{4})\star (\eta_{2}\eta_{3}) \otimes v_{4}-\xi_{1}\xi_{2}\xi_{3} \otimes \beta_{1,2} v_{4}\\
	%=&\eta_{1}\eta_{2}\eta_{3} \otimes (-iv_{1}+v_{2}-\beta_{1,2} v_{4})+\eta_{1}\eta_{2}\eta_{4} \otimes (-\beta_{1,2} v_{3})+\eta_{1}\eta_{3}\eta_{4} \otimes (-v_{4}-\beta_{1,2} v_{2})+\eta_{2}\eta_{3}\eta_{4} \otimes (iv_{4}-\beta_{1,2} v_{1}).
	%\end{align*}
	%Therefore:
	%\begin{align*}
	%&\beta_{1,2}(v_{1})=iv_{4},\\
	%&\beta_{1,2}(v_{2})=-v_{4},\\
	%&\beta_{1,2}(v_{3})=0,\\
	%&\beta_{1,2}(v_{4})=-iv_{1}+v_{2}.
	%\end{align*}
\end{proof}	
\begin{rem}		
	By \eqref{notazionevwgrado1}, Equations \eqref{alfabetagrado1} are equivalent to:
	
	\begin{align}
	&e_1.w_{1}=0                          &&e_2.w_1=0\label{alfa1},\\ 
	&e_1.w_2=w_1                          &&e_2.w_2=w_1 \nonumber\\
	&e_1.\widetilde w_1=w_2-\widetilde w_2&&e_2.\widetilde w_1=-w_2-\widetilde w_2 \nonumber\\
	&e_1.\widetilde w_2=-w_1              &&e_2.\widetilde w_2=w_1 \nonumber
	\end{align}
	We represent these relations in the following picture
	\begin{center}
		\begin{tikzpicture}
		\node at (0,0) (tildew_{1}) {$\langle \widetilde{w}_{1} \rangle$};
		\node at (3,1) (w2-tildew_{2}) {$\langle w_{2}-\widetilde{w}_{2} \rangle $};
		\node at (3,-1) (w2+tildew_{2}) {$\langle w_{2}+\widetilde{w}_{2} \rangle $};
		\node at (6,0) (w1) {$\langle w_{1} \rangle $.};
		\draw (tildew_{1}) edge[->] (w2-tildew_{2});
		\draw (tildew_{1}) edge[->] (w2+tildew_{2});
		\node at (1,0.7) {$e_1$};
		\node at (1,-0.7) {$e_2$};
		\draw (w2-tildew_{2}) edge[->] (w1);
		\draw (w2+tildew_{2}) edge[->] (w1);
		\node at (5,0.7) {$e_1$};
		\node at (5,-0.7) {$e_2$};
		\end{tikzpicture}
	\end{center}
	
	By \eqref{notazionevwgrado1}, Equations \eqref{B.44} can be rewritten in the following equivalent way %, as in \cite{kac1} (see Lemma B.6): 
	\begin{align}
	\label{B.58}
	&(t-i\xi_{12}).\widetilde{w}_{1}=-f_x.\widetilde{w}_{2}+f_y.w_{2},\\ \nonumber
	&(t-i\xi_{34}).\widetilde{w}_{2}=-f_y.w_{1}-e_x.\widetilde{w}_{1},\\ 
	\label{B.59}
	&(t+i\xi_{12}).{w_{1}}=e_x.w_{2}-e_y.\widetilde{w}_{2},\\ \nonumber
	&(t+i\xi_{34}).{w_{2}}=e_y.\widetilde{w}_{1}+f_x.w_{1}. 
	\end{align}
\end{rem}	
\begin{proof}[Proof of Theorem \ref{sing1}]
	As usual we denote by $\mu=(m,n,\mu_0,C)$ the highest weight of the Verma module containing the singular vector $\vec{m}$. 
	Let us first observe that, by Equations \eqref{alfa1}, we have that if $w_{1} \neq 0$ then $w_{2} \neq 0$. %since $\alpha_{1,2}(w_{2})=w_{1}$.
	
	\textbf{1)} Let $w_{1}=w_{2}=0$.\\
	By Equations \eqref{alfa1}, we obtain that if $\widetilde{w}_{2} \neq 0$, then $\widetilde{w}_{1} \neq 0$. Hence, there are two subcases.
	\begin{itemize}%[leftmargin=0.8cm]
		\item[\textbf{1a)}] Let $\widetilde{w}_{1} \neq 0$ and $\widetilde{w}_{2} =0$.\\
		By Equations \eqref{alfa1} we know that $\widetilde{w}_{1}$ is a highest weight vector. Let us compute its weight.
		By \eqref{notazionevwgrado1} we know that $v_{1}=-i\widetilde{w}_{1},v_{2}=\widetilde{w}_{1}, v_{3}=0, v_{4}=0$.\\
		Equation \eqref{B.45} for $a=1,b=3,c=4$ gives
		\begin{align*}
		(-i\xi_{3,4}+C).\widetilde{w}_{1}=0. %\label{B.45 134 1a}
		\end{align*}
		Equation \eqref{B.58} gives $(t-i\xi_{12}).\widetilde{w}_{1}=0$. These two conditions imply $\mu=(m,n,-\frac{m+n}{2},\frac{m-n}{2}) $ with $m,n \in \Z_{\geq 0}$.\\
		These conditions lead to the vector 
		\begin{align*}
		\vec{m}_{1a}=w_{11}\otimes x^{m}_{1}y^{n}_{1}\in M(m,n,-\frac{m+n}{2},\frac{m-n}{2}).
		\end{align*}
		which is indeed a singular vector.
		\item[\textbf{1b)}] Let $\widetilde{w}_{1} \neq 0$ and $\widetilde{w} _{2}\neq 0$.\\
		Equations \eqref{alfa1} imply that $\widetilde{w}_{2} $ is a highest weight vector, let us compute its weight.\\
		By \eqref{notazionevwgrado1}, we know that $v_{1}=-i\widetilde{w}_{1},v_{2}=\widetilde{w}_{1}, v_{3}=-i\widetilde{w}_{2}, v_{4}=\widetilde{w}_{2}$.\\
		Equation \eqref{B.45} for $a=1,b=2,c=3$ gives
		 \[ -i\xi_{12}.\widetilde{w}_{2} +\xi_{13}.\widetilde{w}_{1}-i \xi_{23}. \widetilde w_{1}+C\widetilde{w}_{2}=0 ,\]
	and for $a=1,b=2,c=4$ gives
		\[-\xi_{12}.\widetilde{w}_{2}+\xi_{14}. \widetilde{w}_{1}-i\xi_{24}.\widetilde{w}_{1}-iC\widetilde{w}_{2}=0.
		\]
		These two equations imply
		\[
		0=-2i\xi_{12}.\widetilde{w}_{2}+2C\widetilde{w}_{2} -(e_1+e_2).\widetilde{w}_{1}=2(-i\xi_{12}+C+1).\widetilde w_2.
		\]
		that is equivalent to:
		\begin{align*}
		&(-i\xi_{12}+C+1).\widetilde w_2=0.
		\end{align*}
		Equation \eqref{B.58} provides
		\begin{align*}
		(t-i\xi_{34}).\widetilde{w}_{2}=-e_x.\widetilde{w}_{1}=\widetilde{w}_{2}
		\end{align*}
		and so these equations imply  $\mu=(m,n,1+\frac{m-n}{2},-\frac{m+n}{2}-1)$, for $m,n\geq 0$. We point out that $m \in \Z_{>0}$ since $e_x.\widetilde{w}_{1}=-\widetilde{w}_{2}\neq 0$. 
		
		These conditions lead to the vector 
		\begin{align*}
		\vec{m}_{1b}=w_{21}\otimes x^{m}_{1}y^{n}_{1}-w_{11}\otimes x^{m-1}_{1}x_{2}y^{n}_{1}\in M(m,n,-\frac{m+n}{2},\frac{m-n}{2}),
		\end{align*}
		which is indeed a singular vector. 
	\end{itemize}
	\textbf{2)} Let $w_{1}\neq 0$ and $w_{2}\neq 0$.\\ By \eqref{alfa1} we have  $\widetilde{w}_{2} \neq 0$  $\widetilde{w}_{1} \neq 0$ and that $w_{1}$ is a highest weight vector.\\
	By \eqref{notazionevwgrado1}, Equation \eqref{B.45} for $a=1,b=3,c=4$ gives
	\[(-\xi_{13}-i\xi_{14}).w_{2}+(-\xi_{13}+i\xi_{14}).\widetilde{w}_{2}+(i\xi_{34}+C).w_{1}+(-i\xi_{34}+C).\widetilde{w}_{1}=0,
	\]
	and for $a=2,b=3,c=4$ gives
	\[(-\xi_{23}-i\xi_{24}).w_{2}+(-\xi_{23}+i\xi_{24}).\widetilde{w}_{2}+(-\xi_{34}+iC).w_{1}+(-\xi_{34}-iC).\widetilde{w}_{1}=0.
	\]
	By these equations we deduce
	\[(e_1+e_2).w_2+(e_1-e_2).\widetilde w_2+2(i\xi_{34}+C).w_1=2(i\xi_{34}+C).w_1=0.\]
	
	Moreover, Equation \eqref{B.59} provides
	\[
	(t+i\xi_{12}).w_1=2w_1.
	\]
	These conditions imply that $\mu=(m,n,\frac{m+n}{2}+2,\frac{n-m}{2})$. Note that $m,n>0$ since $e_x.w_2=w_1\neq 0$ and $e_y.\widetilde w_2=-w_1\neq 0$.
	%Now we want to express $w_{2},\widetilde{w}_{1},\widetilde{w}_{2}$ in function of $w_{1}$.\\
	
	%By \eqref{alfa3} and \eqref{beta3}, $\alpha_{1,2}(w_{2})+\beta_{1,2}(w_{2})=E_{\varepsilon_{1}-\varepsilon_{2}}w_{2}=2w_{1}$. Therefore $w_{2}=\frac{-1}{2\mu_{1}-2C}E_{-(\varepsilon_{1}-\varepsilon_{2})}w_{1}$.\\  
	%By \eqref{alfa3} and \eqref{beta3}, $\alpha_{1,2}(\widetilde{w}_{2})-\beta_{1,2}(\widetilde{w}_{2})=E_{\varepsilon_{1}+\varepsilon_{2}}\widetilde{w}_{2}=-2w_{1}$. Therefore $\widetilde{w}_{2}=\frac{1}{2\mu_{1}+2C}E_{-(\varepsilon_{1}+\varepsilon_{2})}w_{1}$.\\ 
	%By \eqref{alfa1} and \eqref{beta1}, $\alpha_{1,2}(\widetilde{w}_{1})+\beta_{1,2}(\widetilde{w}_{1})=E_{\varepsilon_{1}-\varepsilon_{2}}\widetilde{w}_{1}=-2\widetilde{w}_{2}$ and, by \eqref{alfa2} and \eqref{beta2}, $\alpha_{1,2}(\widetilde{w}_{1})-\beta_{1,2}(\widetilde{w}_{1})=E_{\varepsilon_{1}+\varepsilon_{2}}\widetilde{w}_{1}=2w_{2}$. We obtain:
	%\begin{align*}
	%\widetilde{w}_{1}=\frac{1}{(2\mu_{1}+2C)(2\mu_{1}-2C)}E_{-(\varepsilon_{1}-\varepsilon_{2})}E_{-(\varepsilon_{1}+\varepsilon_{2})}w_{1}.
	%\end{align*}
	%Finally, the highest weight of $w_{1}$ with respect to $h_{x}$, $h_{y}$, $E_{00}$, $C$ is $(m,n,\frac{m+n}{2}+2,\frac{n-m}{2})$ with $m,n \in \Z_{>0}$.\\
	
	All these conditions lead to the vector 
	\begin{align*}
	\vec{m}_{1c}&=w_{22}\otimes x^{m}_{1}y^{n}_{1}-w_{12}\otimes x^{m-1}_{1}x_{2}y^{n}_{1}-w_{21}\otimes x^{m}_{1}y^{n-1}_{1}y_{2}+w_{11}\otimes x^{m-1}_{1}x_{2}y^{n-1}_{1}y_{2}\\&\in M(m,n,\frac{m+n}{2}+2,\frac{n-m}{2}),
	\end{align*}
	which is indeed a singular vector.\\
	\textbf{3)} Let $w_{1}= 0$ and $w_{2}\neq 0$. Note that $\widetilde{w}_{1} \neq 0$, since $(e_1-e_2).\widetilde{w}_{1}=2w_{2}\neq 0$ by \eqref{alfa1}. 
	
	\begin{itemize}%[leftmargin=0.8cm]
		\item[\textbf{3a)}] Let $\widetilde{w}_{2} =0$.\\
		Note that $w_{2}$ is a highest weight vector by \eqref{alfa1}. Let us compute its weight.\\
		Using \eqref{notazionevwgrado1}, Equation \eqref{B.45} for $a=1,b=2,c=3$ gives
		\[i\xi_{12}.w_{2}+\xi_{13}.\widetilde{w}_{1}-i\xi_{23}.\widetilde{w}_{1}+Cw_{2}=0
		\]
		and Equation \eqref{B.45} for $a=1,b=2,c=4$ gives
		\[-\xi_{12}.w_{2}+\xi_{14}.\widetilde{w}_{1}-i\xi_{24}.\widetilde{w}_{1}+iCw_{2}=0. \label{B.45 124}
		\]
		These two equations imply
		%\eqref{B.45 123}-i\eqref{B.45 124}, we have:
		\[0=2i\xi_{12}.w_{2}-e_1.\widetilde{w}_{1}+e_2.\widetilde{w}_{1}+2Cw_{2}=2(i\xi_{12}-1+C).w_{2}\]
		and Equation \eqref{B.59} provides
		\begin{align*}
		&(t+i\xi_{34}).w_{2}=w_{2}.
		\end{align*}
		These equations imply $\mu=(m,n,\frac{n-m}{2}+1,\frac{m+n}{2}+1)$, with $m,n\geq 0$. Moreover we have  $n>0$ since, by \eqref{alfa1} $e_y.\widetilde{w}_{1}=w_{2} \neq 0$.
		
		All these conditions lead to the vector 
		\begin{align*}
		\vec{m}_{1d}=w_{12}\otimes x^{m}_{1}y^{n}_{1}-w_{11}\otimes x^{m}_{1}y^{n-1}_{1}y_{2}\in M(m,n,\frac{n-m}{2}+1,\frac{m+n}{2}+1),
		\end{align*}
		which is indeed a singular vector.
		\item[\textbf{3b)}]Let $\widetilde{w}_{2} \neq 0$.\\
		By Equations \eqref{alfa1}, $w_{2}$ and $\widetilde{w}_{2}$ are highest weight vectors. Let us compute their  weight. \\
		By \eqref{notazionevwgrado1}, Equation \eqref{B.45} for $a=1,b=2,c=3$ gives
		\[i\xi_{12}.w_{2}-i\xi_{12}.\widetilde{w}_{2}+\xi_{13}.\widetilde{w}_{1}-i\xi_{23}.\widetilde{w}_{1}+Cw_{2}+C\widetilde{w}_{2}=0,
		\]
		and Equation \eqref{B.45} for $a=1,b=2,c=4$ gives
		\[-\xi_{12}.w_{2}-\xi_{12}.\widetilde{w}_{2}+\xi_{14}.\widetilde{w}_{1}-i\xi_{24}.\widetilde{w}_{1}+iCw_{2}-iC \widetilde{w}_{2}=0 \]
		These two equations imply
		\[(i\xi_{12}-1+C).w_{2}=0\]
		and
		\[
		(-i\xi_{12} +1+C).\widetilde{w}_{2}=0
		\]
		and in particular $C=0$. But, for $C=0$, the $\lambda-$action of Proposition \ref{actiondual} reduces to the action found in Theorem 4.3 of \cite{kac1} where the vectors of degree 1 were classified and this case was ruled out.
	\end{itemize}
	\begin{ack}
	The authors would like to thank Nicoletta Cantarini and Victor Kac for useful comments and suggestions. L.B. was partially supported by the EU Horizon 2020 project GHAIA, MCSA RISE project GA No 777822.
	\end{ack}
	
\end{proof}


\begin{thebibliography}{-}
	\addcontentsline{toc}{chapter}{Bibliography}
	\bibitem{kac1} Boyallian C., Kac V.G., Liberati, J. \textit{Irreducible modules over finite simple Lie conformal superalgebras of type K}, J. Math. Phys. 51 (2010), 1-37.
	\bibitem{ck6} Boyallian C., Kac V. G., Liberati J. \textit{Classification of finite irreducible modules over the Lie conformal superalgebra $CK_{6}$}, Comm. Math. Phys. 317 (2013), 503-546.
	\bibitem{bklr} Boyallian C., Kac  V. G., Liberati J., Rudakov A. \textit{Representations of simple finite Lie conformal superalgebras of type W and S}, J. Math. Phys. 47 (2006), 1-25.
	\bibitem{cantacaselliE510} Cantarini N., Caselli F. \textit{Low Degree Morphisms of E(5, 10)-generalized Verma Modules}, Algebras and Representation Theory 23, (2020) 2131-2165.
	\bibitem{cantacasellikac} Cantarini N., Caselli F., Kac, V.G. \textit{Lie conformal superalgebras and duality of modules over linearly compact Lie superalgebras}, Adv. Math. 378 (2021), 107523.
	\bibitem{cantacasellikacE510} Cantarini N., Caselli F., Kac, V.G. \textit{Classification of degenerate Verma modules for E(5,10)}, Commun. Math. Phys. (2021). https://doi.org/10.1007/s00220-021-04031-z.
	\bibitem{chengcantakac} Cheng S., Cantarini N., Kac V. G., \textit{Errata to Structure of Some $\Z$-graded Lie Superalgebras of Vector Fields}, Transf. Groups 9 (2004), 399-400.
	\bibitem{chengkac} Cheng S., Kac, V. G. \textit{Conformal modules}, Asian J. Math. 1, 181 (1997); 2, 153(E) (1998).
	\bibitem{new6} Cheng S., Kac V.G. \textit{A new $N = 6$ superconformal algebra}, Comm. Math. Phys. 186 (1997), 219-231.
	\bibitem{chengkac2} Cheng S., Kac, V. G. \textit{Structure of some $\Z$-graded Lie superalgebras of vector fields}, Transf. Groups, 4 (1999), 219-272.
	\bibitem{chenglam} Cheng, S., Lam, N. \textit{Finite conformal modules over N=2,3,4 superconformal algebras}, J. Math. Phys. 42 (2001), 906-933.
	\bibitem{dandrea} D'Andrea A, Kac V. G. \textit{Structure theory of finite conformal algebras}, Selecta Math., (N.S.) 4 (1998) 377-418.
	\bibitem{fattorikac} Fattori D., Kac V. G. \textit{Classification of finite simple Lie conformal superalgebras}, J. Algebra 258, (2002) 23-59., Special issue in celebration of Claudio Procesi's 60th birthday.
	\bibitem{kac2} Kac V.G., Van de Leur J.W. \textit{On classification of superconformal algebras}, in: S.J. Gates, et al. (Eds.), Strings 88, World Scientific, Singapore, 1989, 77-106.
	\bibitem{kac1vertex} Kac V.G. \textit{Vertex algebras for beginners}, Univ. Lecture Ser., Vol. 10, AMS, Providence, RI, 1996, 2nd ed., (1998).
	\bibitem{kac98} Kac V.G. \textit{Classification of Infinite-Dimensional Simple Linearly Compact Lie Superalgebras}, Advanced in Mathematics 139 (1998), 11-55 
	\bibitem{kacrudakovE36} Kac V.G., Rudakov A. \textit{Representations of the Exceptional Lie Superalgebra $E(3,6)$ II: Four Series of Degenerate Modules}, Comm. Math. Phys. 222, (2001) 611-661 .
	\bibitem{kacrudakov} Kac V. G., Rudakov A. \textit{Representations of the exceptional Lie superalgebra E(3,6). I. Degeneracy conditions}, Transform. Groups 7, (2002) 67-86. 
	\bibitem{kacrudakovE38} Kac V. G., Rudakov A. \textit{Complexes of modules over the exceptional Lie superalgebras E(3, 8) and E(5, 10)}, Int. Math. Res. Not. 19 (2002), 1007-1025.
	\bibitem{E36III} Kac V. G., Rudakov A. \textit{Representations of the exceptional Lie superalgebra E(3, 6). III. Classification of singular vectors}, J. Algebra Appl. 4 (2005), 15-57.
	\bibitem{knapp} Knapp A. W. \textit{Lie Groups Beyond an Introduction}, Progress in Mathematics Vol. 140 (Birkh\"auser, Boston, 1996).
	\bibitem{maclane} Mac Lane S. \textit{Homology} (fourth printing) Springer-Verlag, (1995).
	\bibitem{zm} Mart\'inez C., Zelmanov E., \textit{Irreducible representations of the exceptional Cheng-Kac superalgebra}, Trans. Amer. Math. Soc. 366 (2014), 5853-5876.
	\bibitem{rudakovE510} Rudakov A. \textit{Morphisms of Verma modules over exceptional Lie superalgebra E(5, 10)}, arXiv 1003.1369v1.
\end{thebibliography}
\end{document}